\documentclass[11pt]{amsart}
\usepackage{mathrsfs}
\usepackage{amssymb}
\usepackage{verbatim}
\usepackage{hyperref}
\usepackage{color}
\usepackage{amsfonts}
\usepackage{mathrsfs}
\usepackage{amsmath}
\usepackage{amssymb}
\usepackage{graphicx}
\begin{document}
% define theorem environments
\newtheorem{theorem}{Theorem}    %[section]
\newtheorem{proposition}[theorem]{Proposition}
\newtheorem{conjecture}[theorem]{Conjecture}
\def\theconjecture{\unskip}
\newtheorem{corollary}[theorem]{Corollary}
\newtheorem{lemma}[theorem]{Lemma}
\newtheorem{sublemma}[theorem]{Sublemma}
\newtheorem{observation}[theorem]{Observation}
\theoremstyle{definition}
\newtheorem{definition}{Definition}
\newtheorem{notation}[definition]{Notation}
\newtheorem{remark}[definition]{Remark}
\newtheorem{question}[definition]{Question}
\newtheorem{questions}[definition]{Questions}
\newtheorem{example}[definition]{Example}
\newtheorem{problem}[definition]{Problem}
\newtheorem{exercise}[definition]{Exercise}

\def\earrow{{\mathbf e}}
\def\rarrow{{\mathbf r}}
\def\uarrow{{\mathbf u}}
\def\varrow{{\mathbf V}}
\def\tpar{T_{\rm par}}
\def\apar{A_{\rm par}}

\def\reals{{\mathbb R}}
\def\torus{{\mathbb T}}
\def\heis{{\mathbb H}}
\def\integers{{\mathbb Z}}
\def\naturals{{\mathbb N}}
\def\complex{{\mathbb C}\/}
\def\distance{\operatorname{distance}\,}
\def\support{\operatorname{support}\,}
\def\dist{\operatorname{dist}\,}
\def\Span{\operatorname{span}\,}
\def\degree{\operatorname{degree}\,}
\def\kernel{\operatorname{kernel}\,}
\def\dim{\operatorname{dim}\,}
\def\codim{\operatorname{codim}}
\def\trace{\operatorname{trace\,}}
\def\Span{\operatorname{span}\,}
\def\dimension{\operatorname{dimension}\,}
\def\codimension{\operatorname{codimension}\,}
\def\nullspace{\scriptk}
\def\kernel{\operatorname{Ker}}
\def\ZZ{ {\mathbb Z} }
\def\p{\partial}
\def\rp{{ ^{-1} }}
\def\Re{\operatorname{Re\,} }
\def\Im{\operatorname{Im\,} }
\def\ov{\overline}
\def\eps{\varepsilon}
\def\lt{L^2}
\def\diver{\operatorname{div}}
\def\curl{\operatorname{curl}}
\def\etta{\eta}
\newcommand{\norm}[1]{ \|  #1 \|}
\def\expect{\mathbb E}
\def\bull{$\bullet$\ }

\def\xone{x_1}
\def\xtwo{x_2}
\def\xq{x_2+x_1^2}
\newcommand{\abr}[1]{ \langle  #1 \rangle}

\numberwithin{theorem}{section} \numberwithin{definition}{section}
\numberwithin{equation}{section}

\newcommand{\Norm}[1]{ \left\|  #1 \right\| }
\newcommand{\set}[1]{ \left\{ #1 \right\} }
\def\one{\mathbf 1}
\def\whole{\mathbf V}
\newcommand{\modulo}[2]{[#1]_{#2}}

\def\scriptf{{\mathcal F}}
\def\scriptg{{\mathcal G}}
\def\scriptm{{\mathcal M}}
\def\scriptb{{\mathcal B}}
\def\scriptc{{\mathcal C}}
\def\scriptt{{\mathcal T}}
\def\scripti{{\mathcal I}}
\def\scripte{{\mathcal E}}
\def\scriptv{{\mathcal V}}
\def\scriptw{{\mathcal W}}
\def\scriptu{{\mathcal U}}
\def\scriptS{{\mathcal S}}
\def\scripta{{\mathcal A}}
\def\scriptr{{\mathcal R}}
\def\scripto{{\mathcal O}}
\def\scripth{{\mathcal H}}
\def\scriptd{{\mathcal D}}
\def\scriptl{{\mathcal L}}
\def\scriptn{{\mathcal N}}
\def\scriptp{{\mathcal P}}
\def\scriptk{{\mathcal K}}
\def\frakv{{\mathfrak V}}

\begin{comment}
\def\scriptx{{\mathcal X}}
\def\scriptj{{\mathcal J}}
\def\scriptr{{\mathcal R}}
\def\scriptS{{\mathcal S}}
\def\scripta{{\mathcal A}}
\def\scriptk{{\mathcal K}}
\def\scriptp{{\mathcal P}}
\def\frakg{{\mathfrak g}}
\def\frakG{{\mathfrak G}}
\def\boldn{\mathbf N}
\end{comment}

\newfam\msbfam
\font\tenmsb=msbm10   \textfont\msbfam=\tenmsb
\font\sevenmsb=msbm7  \scriptfont\msbfam=\sevenmsb
\font\fivemsb=msbm5   \scriptscriptfont\msbfam=\fivemsb

\def\Bbb{\fam\msbfam \tenmsb}
\def\rr{{\Bbb R}}
\def\rz{{{\rr}^n}}
\def\zz{{\Bbb Z}}
\def\nn{{\Bbb N}}
\newfam\bigfam
\font\tenbig=msbm10 scaled \magstep2   \textfont\bigfam=\tenbig
\font\sevenbig=msbm7 scaled \magstep2
\scriptfont\bigfam=\sevenbig \font\fivebig=msbm5 scaled
\magstep2   \scriptscriptfont\bigfam=\fivebig
\def\Bbig{\fam\bigfam \tenbig}

\def\dprod{\displaystyle\prod}
\def\br{{\Bbig R}}
\def\fz{\infty}
\def\az{\alpha}
\def\supp{{\rm{\ supp\ }}}
\def\dinf{\displaystyle\inf}
\def\loc{{\rm{\ loc\ }}}
\def\lp{{L^p(X)}}
\def\lq{{L^q(X)}}
\def\bdz{\Delta}
\def\ez{\epsilon}
\def\bz{\beta}
\def\vz{\varphi}
\def\pz{\partial}
\def\tz{\theta}
\def\sz{\sigma}
\def\bsz{\Sigma}
\def\dz{\delta}
\def\gz{\gamma}
\def\bgz{\Gamma}
\def\lz{\lambda}
\def\cz{\cdot}
\def\supp{{\rm supp}}
\def\loc{{\rm loc}}
\def\wt{\widetilde}
\def\bmo{\rm {\rm BMO(\rz)}}
\def\blz{\Lambda}
\def\wh{\widehat}
\def\dis{\displaystyle}
\def\dsum{\displaystyle\sum}
\def\dint{\displaystyle\int}
\def\dfrac{\displaystyle\frac}
\def\dsup{\displaystyle\sup}
\def\dlim{\displaystyle\lim}
\def\bom{\Omega}
\def\om{\omega}
\def\C{\mathbb{C}}
\def\R{\mathbb{R}}
\def\Rn{{\mathbb{R}^n}}
\newtheorem{thm}{\hskip\parindent Theorem}
\renewcommand\thethm{}
\newtheorem{lem}{\hskip\parindent Lemma}
\renewcommand\thelem{}
\newtheorem{prop}{\hskip\parindent Proposition}[section]
\newtheorem{rem}{\hskip\parindent Remark}[section]
\newtheorem{cor}{\hskip\parindent Corollary}[section]
\newtheorem{defn}{\hskip\parindent Definition}[section]
\renewcommand\thedefn{}
\newtheorem{pf}{\hskip\parindent Proof}
\renewcommand\thepf{}
\def\dabmo{\dsum_{|\gamma|=m-1}\|D^\gamma A\|_{\rm {\rm BMO}}}
\def\dg{\dsum_{|\gamma|=m-1}}
\def\dgb{\dsum_{|\alpha|=m}\|D^{\alpha}A\|_{\rm {\rm BMO}} }
\def\gb{\sum_{|\alpha|=m}\|D^{\alpha}A\|_{\rm {\rm BMO}} }

\author{Feng Liu}
\address{Feng Liu
\\
College of Mathematics and Systems Science
\\
Shandong University of Science and Technology
\\
Qingdao, Shandong 266590
\\
People's Republic of China} \email{Fliu@sdust.edu.cn}

\author{Qingying Xue$^{*}$}
\address{Qingying Xue
\\
School of Mathematical Sciences
\\
Beijing Normal University
\\
Laboratory of Mathematics and Complex Systems
\\
Ministry of Education
\\
Beijing 100875
\\
People's Republic of China \\
current address:
Department of Mathematics\\ The University of Kansas, Lawrence, KS  66045, USA} \email{qyxue@bnu.edu.cn}

\author{K\^{o}z\^{o} Yabuta}
\address{K\^{o}z\^{o} Yabuta
\\ Research Center for Mathematical Sciences
\\
Kwansei Gakuin University
\\
Gakuen 2-1, Sanda 669-1337
\\
Japan } \email{kyabuta3@kwansei.ac.jp}
\thanks{The first author was supported partly by NSFC (No. 11701333)
and SP-OYSTTT-CMSS (No. Sxy2016K01). The second author was partly
supported by NSFC (Nos. 11471041, 11671039) and NSFC-DFG (No.
11761131002). The third named author was supported partly by
Grant-in-Aid for Scientific Research (C) Nr. 15K04942, Japan Society
for the Promotion of Science.\\
 \indent Corresponding author: Qingying Xue \indent Email: qyxue@bnu.edu.cn}


%Any opinions, findings, and conclusions
%or recommendations expressed in this paper are those of the author
%and do not necessarily reflect the views of the National Science Foundation.}
%s DMS-0401260 and

\date{}
\title [Regularity and continuity of the strong maximal operators]
{Regularity and continuity of the multilinear strong maximal operators}

\maketitle

\begin{abstract} {Let $m\ge 1$, in this paper, our object of
investigation is the regularity and and continuity properties of the following
multilinear strong maximal operator
$$
{\mathscr{M}}_{\mathcal{R}}(\vec{f})(x)
=\sup_{\substack{R \ni x \\ R\in\mathcal{R}}}\prod\limits_{i=1}^m
\frac{1}{|R|}\int_{R}|f_i(y)|dy,
$$
where $x\in\mathbb{R}^d$ and $\mathcal{R}$ denotes the
family of all rectangles in $\mathbb{R}^d$ with sides
parallel to the axes. When $m=1$, denote
$\mathscr{M}_{\mathcal{R}}$ by $\mathcal {M}_{\mathcal{R}}$.
Then, $\mathcal {M}_{\mathcal{R}}$ coincides with the
classical strong maximal function initially studied by
Jessen, Marcinkiewicz and Zygmund. We showed that
${\mathscr{M}}_{\mathcal{R}}$ is bounded and continuous
from the Sobolev spaces $W^{1,p_1}(\mathbb{R}^d)\times
\cdots\times W^{1,p_m}(\mathbb{R}^d)$ to $W^{1,p}
(\mathbb{R}^d)$, from the Besov spaces $B_{s}^{p_1,q}
(\mathbb{R}^d)\times\cdots\times B_s^{p_m,q}(\mathbb{R}^d)$
to $B_s^{p,q}(\mathbb{R}^d)$, from  the Triebel-Lizorkin
spaces $F_{s}^{p_1,q}(\mathbb{R}^d)\times\cdots\times
F_s^{p_m,q}(\mathbb{R}^d)$ to $F_s^{p,q}(\mathbb{R}^d)$. As
a consequence, we further showed that
${\mathscr{M}}_{\mathcal{R}}$ is bounded and continuous from
the fractional Sobolev spaces $W^{s,p_1}(\mathbb{R}^d)\times
\cdots\times W^{s,p_m}(\mathbb{R}^d)$ to $W^{s,p}
(\mathbb{R}^d)$ for $0<s\leq 1$ and $1<p<\infty$. As an
application, we obtain a weak type inequality for the Sobolev
capacity, which can be used to prove the $p$-quasicontinuity
of $\mathscr{M}_{\mathcal{R}}$. In addition, we proved that
$\mathscr{M}_{\mathcal{R}}(\vec{f})$ is approximately
differentiable a.e. if $\vec{f}=(f_1,\ldots,f_m)$ with each
$f_j\in L^1(\mathbb{R}^d)$ being approximately differentiable
a.e. The discrete type of the strong maximal operators has
also been considered. We showed that this discrete type of
the maximal operators enjoys somewhat unexpected regularity
properties.}
\end{abstract}

\section{Introduction}
\subsection{Hardy-Littlewood maximal functions} Let $f\in
L_{\rm loc}^1(\mathbb{R}^d)$ with $d\ge 1$ and $\mathcal{M}$
be the well-known Hardy-Littlewood maximal operator defined
on $\mathbb{R}^n$ as follows.
$$
\mathcal{M}f(x)=\sup\limits_{r>0}\frac{1}{|B_r(x)|}\int_{B_r(x)}|f(y)|dy,
$$
where $B_r(x)$ is the open ball in $\mathbb{R}^d$ centered
at $x$ with radius $r$ and $|B_r(x)|$ denotes the volume
of $B_r(x)$. Analogously, the uncentered maximal function
$\widetilde{\mathcal{M}}f$ at a point $x$ is defined by
taking the supremum of averages over open balls that
contain the point. It was well known that the maximal
functions and their purpose in differentiation on $\mathbb{R}$
were first introduced by Hardy and Littlewood \cite{HL}, and
on $\mathbb{R}^d$ were treated by Wiener \cite{W}. The celebrated
theorem of Hardy-Littlewood-Wiener states that the operator
$\mathcal{M}$ is of type $(p,p)$ for $1<p\leq\infty$ and weak
type $(1,1)$. As a basic and important tool in Harmonic
analysis and other fields, such as PDE, the maximal functions
and their variants are often used to control some other
important operators and give some good absolute size estimates
(see \cite{BH}, \cite{LOPTT} and \cite{Le}).

There is a basic question in the theory of Hardy-Littlewood
maximal operators: How does the Hardy-Littlewood maximal
operator preserve the smoothness properties of a function?
Achievements have been made in this direction in the past
few years. Among them is the nice work of Kinnunen \cite{Ki}
in 1997, where the regularity properties of maximal
operators on the $W^{1,p}$ spaces has been studied. Recall
that the Sobolev spaces $W^{1,p}(\mathbb{R}^d)$,
$1\leq p\leq\infty$, are defined by
$$
W^{1,p}(\mathbb{R}^d):=\{f:\mathbb{R}^d\rightarrow\mathbb{R}:
\ \|f\|_{1,p}=\|f\|_{L^p(\mathbb{R}^d)}+\|\nabla{f}\|_{L^p(\mathbb{R}^d)}
<\infty\},
$$
where $\nabla{f}=(D_1f,\ldots,D_df)$ is the weak gradient
of $f$. Kinnunen showed that $\mathcal{M}$ is bounded from
$W^{1,p}(\mathbb{R}^d)$ to $W^{1,p}(\mathbb{R}^d)$ for
$1<p\leq\infty$. It was noticed that the $W^{1,p}$-bound
for $\widetilde{\mathcal{M}}$ also holds by a simple
modification of Kinnunen's arguments or Theorem 1 of
\cite{HO}. Later on, the result of Kinnunen has been
extended to a local version in \cite{KL}, to a fractional
version in \cite{KS}, to a multisublinear version in
\cite{CM2,LW1} and to a one-sided version in \cite{LM}. Whether the continuity for
$\mathcal{M}$ on $W^{1,p}(\mathbb{R}^d)$ space holds or not
is another certainly nontrivial problem, since the maximal operator is not necessarily
sublinear at the derivative level. This problem was first
posed by Haj{\l}asz and Onninen \cite{HO} and was later
settled affirmatively by Luiro \cite{Lu1}. 

Due to the lack
of reflexivity of $L^1$, it makes the understanding of the
$W^{1,1}(\mathbb{R}^d)$ regularity more subtler. One interesting question was raised by Haj{\l}asz
and Onninen in \cite{HO}: Is the operator $f\mapsto|\nabla
\mathcal{M}f|$ bounded from $W^{1,1}(\mathbb{R}^d)$ to
$L^1(\mathbb{R}^d)$? A complete answer was addressed only
in dimension $d=1$ in \cite{AL1,Ku,LCW,Ta} and partial
progress on the general case $d\geq2$ was given by
Haj{\l}asz and Mal\'{y} \cite{HM} and Luiro \cite{Lu3}. For
more previous works or related topic we refer the readers
to consult \cite{AL2,CM1,CMP,CS,Ku,Liu1,LW2}, and the 
references therein.

Now we know that $\mathcal{M}$ is bounded on
$L^p(\mathbb{R}^d)=W^{0,p}(\mathbb{R}^d)$ and $W^{1,p}
(\mathbb{R}^d)$ for $p>1$. Therefore a natural question
arises: what is the properties of $\mathcal{M}$ on the
fractional Sobolev spaces $W^{s,p}(\mathbb{R}^d)$ defined
by the Bessel potentials when $0<s<1$? This question
was first studied by Korry \cite{Ko2} who observed that
$\mathcal{M}:W^{s,p}(\mathbb{R}^d)\rightarrow W^{s,p}
(\mathbb{R}^d)$ is bounded for all $0<s<1$ and $1<p<\infty$.
Notice that $F_s^{p,2}(\mathbb{R}^d)=W^{s,p}(\mathbb{R}^d)$
for any $s>0$ and $1<p<\infty$ (see \cite{Gr}). It may be
further expected that $\mathcal{M}$ still enjoys the
boundedness on Triebel-Lizorkin spaces $F_s^{p,q}(\mathbb{R}^d)$.
This was done by Korry \cite{Ko1}, who indeed proved that
$\mathcal{M}$ is bounded on the inhomogeneous Triebel-Lizorkin
spaces $F_s^{p,q}(\mathbb{R}^d)$ and Besov spaces $B_s^{p,q}
(\mathbb{R}^d)$ for all $0<s<1$ and $1<p,\,q<\infty$. Recently,
Luiro \cite{Lu2} established the continuity of $\mathcal{M}$
on $F_s^{p,q}(\mathbb{R}^d)$ for all $0<s<1$ and $1<p,\,q<\infty$.
Still more recently, Liu and Wu \cite{LW4} extended the above
results to the maximal operators associated with polynomial
mappings.
\subsection{Multilinear strong maximal operators}
Over the past few decades, many celebrated works have been done
in the study of the maximal functions associated with different
kinds of basis. These bases mainly including: some differentiation
bases (balls or cubes, rectangles with some restrictions see
\cite{JMZ}, \cite{ZY} and \cite{ZY2}), translation in-variant
basis of rectangles \cite{CF}, basis formed by convex sets, using
rectangles with a side parallel to some direction (lacunary
parabolic set of directions in \cite{NSW}, Cantor set of directions
in \cite{KA1}, arbitrary set of directions \cite{AS}, \cite{KA2}).
In this paper, we will focus on the translation in-variant basis
of rectangles studied by C\'{o}rdoba and Fefferman \cite{CF}.

Let $\vec{f}=(f_1,\ldots,f_m)$ be an $m$-dimensional vector
of locally integrable functions and $\mathcal{R}$ denotes
the collection of all open rectangles $R\subset\mathbb{R}^d$
with sides parallel to the coordinate axes. In \rm{2011},
Grafakos, Liu, P\'{e}rez and Torres \cite{GLPT} introduced
and studied the weighted strong and endpoint estimates for
the multilinear strong maximal function
${\mathscr{M}}_{\mathcal{R}}$, which is defined by
\begin{equation}
{\mathscr{M}}_{\mathcal{R}}(\vec{f})(x)=
\sup_{\substack{R \ni x \\ R\in\mathcal{R}}}
\prod_{i=1}^m \frac{1}{|R|} \int_R |f_i(y_i)| dy_i,
\end{equation}
where $x\in\mathbb{R}^d$ and $\mathcal{R}$ denotes the
family of all rectangles in $\mathbb{R}^d$ with sides
parallel to the axes.

Whenever $m=1$, we simply denote $\mathscr{M}_{\mathcal{R}}$
by $\mathcal{M}_{\mathcal{R}}$. Then
$\mathcal{M}_{\mathcal{R}}$ coincides with the classical
strong maximal operator. As the most prototypical
representative of the multi-parameter operators,
$\mathcal{M}_{\mathcal{R}}$ can be looked as a geometric
maximal operator which commutes with full $d$-parameter
group of dilations $(x_1,x_2,\ldots,x_d)\rightarrow
(\delta_1x_1,\delta_2x_2,\ldots,\delta_dx_d)$. It was
proved by Garc\'{\i}a-Cuerva and Rubio de Francia that
$\mathcal{M}_{\mathcal{R}}$ is bounded on
$L^p(\mathbb{R}^d)$ for all $1<p<\infty$ (see
\cite[p.452]{GR}). In \rm{1935}, a maximal theorem was
given by Jessen, Marcinkiewicz and Zygmund in \cite{JMZ}.
They pointed out that unlike the classical Hardy-Littlewood
maximal operator, the strong maximal function is not of
weak type $(1,1)$. As a replacement, they showed that it
is bounded from $L(\log^+L)(\mathbb{R}^d)$ to
$L^1(\mathbb{R}^d)$. Subsequently, an additional proof of
the maximal theorem was given by C\'{o}rdoba and Fefferman
in \rm{1975}, using an alternative geometric method \cite{CF}.
The basis of the work
of C\'{o}rdoba and Fefferman is a selection theorem for
families of rectangles in $\mathbb{R}^d$. Some delicate properties of rectangles in $\mathbb{R}^d$
were also quantified in that study. 

Furthermore, if $m=1$ and $d=1$, the operator
${\mathscr{M}}_{\mathcal{R}}=\widetilde{\mathcal{M}}$. It
was known that $\widetilde{\mathcal{M}}$ is bounded and
continuous on $W^{1,p}(\mathbb{R})$ for $1<p<\infty$. It
follows from \cite{AL1,LCW} that if $f\in W^{1,1}
(\mathbb{R})$, then $\widetilde{\mathcal{M}}f$ is
absolutely continuous on $\mathbb{R}$ and it holds that
$\|( \widetilde{\mathcal{M}}f)'\|_{L^1(\mathbb{R})}
\leq\|f'\|_{L^1(\mathbb{R})}.$
For $d\geq1$, Aldaz
and P\'{e}rez L\'{a}zaro \cite{AL2} considered a class of
local strong maximal operator and proved that it maps
${\rm BV}(U)$ into $L^1(U)$, where $U$ is an open set of
$\mathbb{R}^d$ and ${\rm BV}(U)$ is a subclass of $L^1(U)$
functions. See \cite[Definition 1.3]{Gi} and
\cite[Definition 3.4]{AFP} for instance.

The results in \cite{GLPT} indicate that $\mathscr{M}_{\mathcal{R}}$ 
is bounded from $L^{p_1}(\mathbb{R}^d)\times\cdots\times L^{p_m}
(\mathbb{R}^d)$ to $L^p(\mathbb{R}^d)$ for all $1<p_1,
\ldots,p_m,p\leq\infty$ and $1/p=\sum_{i=1}^m1/p_i$.
Moreover, for $\vec{f}=(f_1,\ldots,f_m)$ with each
$f_i\in L^{p_i}(\mathbb{R}^d)$, the following norm inequality holds
\begin{equation}\label{1.2}
\|\mathscr{M}_{\mathcal{R}}(\vec{f})\|_{L^p(\mathbb{R}^d)}
\lesssim_{p_1,\ldots,p_m}\prod\limits_{i=1}^m\|f_i\|_{L^{p_i}(\mathbb{R}^d)}.
\end{equation}
It is well known that the geometry of rectangles in
$\mathbb{R}^d$ is more intricate than that of cubes or
balls, even when both classes of sets are restricted to
have sides parallel to the axes. Even for $m=1$, a basic
observation is that $\mathcal{M}f(x)\lesssim_d
\mathcal{M}_{\mathcal{R}}f(x)$ for all $x\in\mathbb{R}^d$.
However, there does not exist any constant $C>0$ such
that $\mathcal{M}_{\mathcal{R}}f(x)\leq C\mathcal{M}f(x)$
for all $x\in\mathbb{R}^d$. This indicates fully that the
strong maximal functions are uncontrollable. For these
reasons, this makes the investigation of the strong
maximal functions very complex, but also quite interesting.

Based on the facts concerning the previous results on the
Hardy-Littlewood maximal operators, it is therefore a
natural question to ask whether the multilinear strong
maximal operators are bounded and continuous on the products of the first
order Sobolev spaces $W^{1,p}(\mathbb{R}^d)$ or the
fractional Sobolev spaces $W^{s,p}(\mathbb{R}^d)$ or on
its generalizations $F_s^{p,q}(\mathbb{R}^d)$ and $B_s^{p,q}
(\mathbb{R}^d)$. This is the main motivation of this work. 
In the first part of this work, the regularityand continuity properties
of the strong maximal functions will be studied. We will
show that $ {\mathscr{M}}_{\mathcal{R}}$ is bounded and
continuous from the Sobolev spaces $W^{1,p_1}(\mathbb{R}^d)
\times\cdots\times W^{1,p_m}(\mathbb{R}^d)$ to
$W^{1,p}(\mathbb{R}^d)$, from the Besov spaces
$B_{s}^{p_1,q}(\mathbb{R}^d)\times\cdots\times B_s^{p_m,q}
(\mathbb{R}^d)$ to $B_s^{p,q}(\mathbb{R}^d)$, from the
Triebel-Lizorkin spaces $F_{s}^{p_1,q}(\mathbb{R}^d)\times
\cdots\times F_s^{p_m,q}(\mathbb{R}^d)$ to $F_s^{p,q}
(\mathbb{R}^d)$. We further showed that 
${\mathscr{M}}_{\mathcal{R}}$ is bounded and continuous 
from the fractional Sobolev spaces $W^{s,p_1}(\mathbb{R}^d)
\times\cdots\times W^{s,p_m}(\mathbb{R}^d)$ to $W^{s,p}
(\mathbb{R}^d)$ for $0<s<1$ and $1<p<\infty$. As an
application, we obtain a weak type inequality for the
Sobolev capacity, which can be used to prove
$p$-quasicontinuity of the strong maximal function of a
Sobolev function. In addition, we  also show that $\mathscr{M}_{\mathcal{R}}
(\vec{f})$ is approximately differentiable a.e. if $\vec{f}=
(f_1,\ldots,f_m)$ with each $f_j\in L^1(\mathbb{R}^d)$ being
approximately differentiable a.e.
\subsection{Discrete multilinear strong maximal operators.}
Another aim of this paper is to investigate the regularity
properties of the discrete multilinear strong maximal
operators. For a vector-valued function $\vec{f}=(f_1,\ldots,
f_m)$ with each $f_j$ being a discrete function defined on
$\mathbb{Z}^d$, we define the discrete multilinear strong
maximal operator $\mathbb{M}_{\mathcal{R}}$ by
\begin{equation}\label{1.3}
\mathbb{M}_{\mathcal{R}}(\vec{f})(\vec{n})
=\sup_{\substack{R \ni\vec{n}\atop R\in\mathcal{R}}}\frac{1}{N(R)^{m}}
\prod\limits_{i=1}^m\sum\limits_{\vec{k}\in R\cap\mathbb{Z}^n}|f_i(\vec{k})|,
\end{equation}
where $N(R)$ is the number of elements in the set
$R\cap\mathbb{Z}^d$. When $m=1$, the operator
$\mathbb{M}_{\mathcal{R}}$ reduces to the discrete strong
maximal operator $M_{\mathcal{R}}$.

Let us recall some pertinent definitions, notations and
backgrounds. We shall generally denote by $\vec{n}=(n_1,
n_2,\ldots,n_d)$ a vector in $\mathbb{Z}^d$. For a discrete
function $f:\mathbb{Z}^d\rightarrow\mathbb{R}$, we define
the $\ell^p(\mathbb{Z}^d)$-norm for $1\leq p<\infty$ by
$\|f\|_{\ell^p(\mathbb{Z}^d)}=(\sum_{\vec{n}\in\mathbb{Z}^d}
|f(\vec{n})|^p)^{1/p}$ and $\ell^\infty(\mathbb{Z}^d)$-norm
by $\|f\|_{\ell^\infty(\mathbb{Z}^d)}=\sup_{\vec{n}\in
\mathbb{Z}^d}|f(\vec{n})|$. Next, we recall the definitions
of discrete Sobolev space $W^{1,p}(\mathbb{Z}^d)$
and ${\rm BV}_q(\mathbb{Z}^d)$ function class.
\begin{definition} [\textbf{Discrete Sobolev
space $W^{1,p}(\mathbb{Z}^d)$}, (\cite{BCHP})] For $1\le l\le d$,
let $\vec{e}_l$ be the canonical
$l$-th base vector defined by $\vec{e}_l=(0,\ldots,0,1,0,\ldots,0)$. 
Let $D_l f(\vec{n})$ be the partial derivative of $f$ given by 
$D_l f(\vec{n})=f(\vec{n}+\vec{e}_l)-f(\vec{n})$ and $\nabla f$ 
be the gradient of $f$ defined by $\nabla f(\vec{n})=
(D_1f(\vec{n}),\ldots,D_df(\vec{n}))$. Then, the discrete 
Sobolev spaces is defined by
$$
W^{1,p}(\mathbb{Z}^d):=\{f:\mathbb{Z}^d\rightarrow\mathbb{R}\mid\ \|f\|_{1,p}
=\|f\|_{\ell^p(\mathbb{Z}^d)}+\|\nabla f\|_{\ell^p(\mathbb{Z}^d)}<\infty\}.
$$
\end{definition}
Note that
\begin{equation}\label{1.4}
\|\nabla f\|_{\ell^p(\mathbb{Z}^d)}\leq 2d\|f\|_{\ell^p(\mathbb{Z}^d)}
\ \ \ {\rm for}\ 1\leq p\leq\infty.
\end{equation}
It follows that
\begin{equation}\label{1.5}
\|f\|_{\ell^p(\mathbb{Z}^d)}\leq\|f\|_{1,p}
\leq(2d+1)\|f\|_{\ell^p(\mathbb{Z}^d)}\ \ {\rm for}\ 1\leq p\leq\infty.
\end{equation}
This implies that the discrete Sobolev space $W^{1,p}
(\mathbb{Z}^d)$ is just $\ell^p(\mathbb{Z}^d)$ with an
equivalent norm. It might make our efforts to study the
$W^{1,p}(\mathbb{Z}^d)$ regularity of discrete maximal
operators seem almost vacuous since any $\ell^p$-bound
automatically implies a $W^{1,p}$-bound. However, the
endpoint $p=1$ is highly nontrivial because of the
lack of $\ell^1$-bound for discrete strong maximal operators.

To investigate the endpoint regularity of $\mathbb{M}_{\mathcal{R}}$, 
we now introduce the following function class.

\begin{definition}[\textbf{${\rm BV}(\mathbb{Z}^d)$
function class}, (\cite{CH})] We denote by ${\rm BV}(\mathbb{Z}^d)$
the set of all functions of bounded variation defined
on $\mathbb{Z}^d$, where the total variation of
$f:\mathbb{Z}^d\rightarrow\mathbb{R}$ is defined by
$${\rm Var}(f)=\|\nabla f\|_{\ell^1(\mathbb{Z}^d)}.$$\end{definition}
\eqref{1.4} together with \eqref{1.5} and a simple example
$f(\vec{n})=1$ yields that
$$
{\rm BV}(\mathbb{Z}^d)\subsetneq \ell^1(\mathbb{Z}^d)=W^{1,1}(\mathbb{Z}^d).
$$
Recently, the investigation of the regularity of discrete
maximal operators has also attracted the attention of many
authors (see \cite{BCHP,CH,CM1,Liu,LM,LW3,LW5,Mad,Te} et al.).
Recall that the discrete uncentered version of maximal
function is defined by
$${M}f(\vec{n})=\sup\limits_{r>0,\vec{n}\in B_r}\frac{1}{N(B_r)}
\sum\limits_{\vec{k}\in B_r\cap\mathbb{Z}^d}|f(\vec{k})|,$$
where the surpremum is taken over all open balls $B_r$ in
$\mathbb{R}^d$ containing the point $\vec{n}$ with radius
$r$ and $N(B_r)$ denotes the number of lattice
points in the set $B_r$. We denote the centered version of
discrete maximal function by $\widetilde{M}$.

When $d=1$, the regularity properties of the discrete maximal 
type operators were studied by Bober et al. \cite{BCHP}, Temur 
\cite{Te} and Madrid \cite{Mad}, Carneiro and Madrid \cite{CM1} 
and Liu \cite{Liu}. The following sharp inequalities have been 
established.
\begin{equation}{\rm Var}(\widetilde{M}f)\leq {\rm Var}(f)\end{equation}
and
\begin{equation}{\rm Var}(Mf)\leq 2\|f\|_{\ell^1(\mathbb{Z})}.\end{equation}

For $d\ge1$, Carneiro
and Hughes \cite{CH} proved that ${M}$ maps $\ell^1
(\mathbb{Z}^d)$ into ${\rm BV}(\mathbb{Z}^d)$ boundedly
and continuously.  In (\ref{1.3}), if one replace the rectangles $R$ by balls $B_r$, then we denote $\mathbb{M}_{\mathcal{R}}$ by  $\mathfrak{M}$. Still more recently, the results in \cite{CH} was extended by Liu and Wu \cite{LW3} as follows.

\quad \hspace{-20pt}{\bf Theorem A (\cite{LW3})
}. {\it Let $d\geq1$.
Then $\mathfrak{M}$ maps $\ell^1(\mathbb{Z}^d)\times
\cdots\times\ell^1(\mathbb{Z}^d)$ into ${\rm BV}
(\mathbb{Z}^d)$ boundedly and continuously.}

It is observed that $\mathfrak{M}(\vec{f})(\vec{n})
\lesssim_{d,m}\mathbb{M}_{\mathcal{R}}(\vec{f})(\vec{n})$
for all $\vec{n}\in\mathbb{Z}^d$. Specially,
$\mathbb{M}_{\mathcal{R}}=\mathfrak{M}$ when $d=1$.
However, when $d\geq2$, there does not exist any constant
$C>0$ such that $\mathbb{M}_{\mathcal{R}}(\vec{f})(\vec{n})
\leq C\mathfrak{M}(\vec{f})(\vec{n})$ for all
$\vec{n}\in\mathbb{Z}^d$. Based on the above analysis,
it is interesting and natural to ask whether the discrete
strong maximal operators still enjoy some sort of
regularity properties. We will show that the discrete
type of the strong maximal operators does enjoy somewhat
unexpected regularities in the end of next part.

\subsection{Main results} We now state our main results
as follows.
\begin{theorem} [\textbf{Properties on Sobolev spaces}]\label{thm1}Let
$1<p_1,\ldots,p_m,p<\infty$ and $1/p=\sum_{i=1}^m1/p_i$.
Then $\mathscr{M}_{\mathcal{R}}$ is bounded and continuous
from $W^{1,p_1}(\mathbb{R}^d)\times\cdots\times W^{1,p_m}
(\mathbb{R}^d)$ to $W^{1,p}(\mathbb{R}^d)$. Moreover, if
$\vec{f}=(f_1,\ldots,f_m)$ with each $f_i\in W^{1,p_i}
(\mathbb{R}^d)$, then, for $1\leq l\leq d$, it holds that
$$
|D_l\mathscr{M}_{\mathcal{R}}(\vec{f})(x)|
\lesssim_{m,d,p_1,\ldots,p_m}\sum\limits_{\mu=1}^m
\mathscr{M}_{\mathcal{R}}(\vec{f}_\mu^l)(x),\ \ {\rm a.e.}\ x\in\mathbb{R}^d,
$$
where $\vec{f}_\mu^l=(f_1,\ldots,f_{\mu-1},D_lf_\mu,
f_{\mu+1},\ldots,f_m)$.
\end{theorem}
\remark{The case $p=\infty$ is also valid in Theorem
\ref{thm1}, which follows from the similar arguments to those
used in \cite[Remark (iii)]{Ki}.}
\begin{theorem} [\textbf{Properties  on Besov spaces}]\label{thm2}
Let $1<p_1,\ldots,p_m,p,q<\infty$, $1/p=\sum_{i=1}^m1/p_i$
and $0<s<1$. Then $\mathscr{M}_{\mathcal{R}}$ is bounded and
continuous from $B_{s}^{p_1,q}(\mathbb{R}^d)\times\cdots
\times B_s^{p_m,q}(\mathbb{R}^d)$ to $B_s^{p,q}(\mathbb{R}^d)$.
\end{theorem}
\begin{theorem} [\textbf{Properties  on Triebel-Lizorkin
spaces}] \label{thm3}Let $1<p_1,\ldots,p_m,p,q<\infty$,
$1/p=\sum_{i=1}^m1/p_i$ and $0<s<1$. Then
$\mathscr{M}_{\mathcal{R}}$ is bounded and continuous
from $F_{s}^{p_1,q}(\mathbb{R}^d)\times\cdots\times
F_s^{p_m,q}(\mathbb{R}^d)$ to $F_s^{p,q}(\mathbb{R}^d)$.
\end{theorem}
Noting that $F_s^{p,2}(\mathbb{R}^d)=W^{s,p}(\mathbb{R}^d)$
for any $s>0$ and $1<p<\infty$, then Theorem \ref{thm3} implies
the following result immediately.
\begin{corollary} [\textbf{Properties on Fractional
Sobolev spaces}]
Let $1<p_1,\ldots,p_m,p<\infty$, $1/p=\sum_{i=1}^m1/p_i$
and $0<s<1$. Then $\mathscr{M}_{\mathcal{R}}$ is bounded
and continuous from the fractional Sobolev spaces
$W^{s,p_1}(\mathbb{R}^d)\times\cdots\times W^{s,p_m}
(\mathbb{R}^d)$ to $W^{s,p}(\mathbb{R}^d)$.
\end{corollary}
Theorem \ref{thm1} can be used to obtain a weak type inequality
for the Sobolev capacity, which can be further employed
to prove the quasicontinuity of the strong maximal
function of a Sobolev function. We first need to give
the definition of Sobolev $p$-capacity.
\begin{definition}[{\bf {Sobolev $p$-capacity}}, (\cite{KKM})] For
$1<p<\infty$, the Sobolev $p$-capacity of the set
$E\subset\mathbb{R}^d$ is defined by
\begin{equation}
C_p(E):=\inf\limits_{f\in\mathcal{A}(E)}
\int_{\mathbb{R}^d}(|f(y)|^p+|\nabla f(y)|^p)dy,
\end{equation}
where $\mathcal{A}(E)=\{f\in W^{1,p}(\mathbb{R}^d):
\ f\geq1\ {\rm on\ a\ neighbourhood\ of}\ E\}$.
We set $C_p(E)=\infty$ if $\mathcal{A}(E)=\emptyset$.
\end{definition}
It was shown in \cite{FZ} that the Sobolev $p$-capacity
is a monotone and a countably subadditive set function.
Also, it is an outer measure over $\mathbb{R}^d$.
\begin{definition} [\textbf{$p$-quasicontinuous and
$p$-quasieverywhere}, \cite{FZ}]{A function $f$ is said to be
$p$-quasicontinuous in $\mathbb{R}^d$ if for every
$\epsilon>0$, there exists a set $F\subset\mathbb{R}^d$
such that $C_p(F)<\epsilon$ and the restriction of $f$
to $\mathbb{R}^d\setminus F$ is continuous and finite.
A property holds $p$-quasieverywhere if it holds outside
a set of the Sobolev $p$-capacity zero.}
\end{definition}
\begin{remark}It was known that each Sobolev function
has a quasicontinuous representative, that is, for each
$u\in W^{1,p}(\mathbb{R}^d)$, there is a
$p$-quasicontinuous function $v\in W^{1,p}(\mathbb{R}^d)$
such that $u=v$ a.e. in $\mathbb{R}^d$.
This representative is unique in the sense that if
$v$ and $w$ are $p$-quasicontinuous and $v=w$ a.e. in
$\mathbb{R}^d$, then $w=v$ $p$-quasieverywhere in
$\mathbb{R}^d$, see \cite{FZ} for more details.
\end{remark}
In 1997, Kinnunen proved that $\mathcal{M}f$ is
$p$-quasicontinuous if $f\in W^{1,p}(\mathbb{R}^d)$ for
any $1<p<\infty$. Motivated by Kinnunen's work \cite{Ki},
we shall prove the following result:
\begin{theorem} [\textbf{$p$-quasicontinuity}]\label{thm5}Let $1<p_1,
\ldots,p_m<\infty$, and $1/p=\sum_{i=1}^m1/p_i$. Suppose that
$\vec{f}=(f_1,\ldots,f_m)$ with each $f_i\in W^{1,p_i}
(\mathbb{R}^d)$, then $\mathscr{M}_{\mathcal{R}}(\vec{f})$
is $p$-quasicontinuous.
\end{theorem}
In 2010, Haj{\l}asz and Mal\'{y} \cite{HM} proved that
$\mathcal{M}f$ is approximately differentiable a.e. provided that
$f\in L^{1}(\mathbb{R}^d)$. Motivated by Haj{\l}asz
and Mal\'{y}'s work, we shall establish the following
result:
\begin{theorem}\label{thm6}
Let $\vec{f}=(f_1,\ldots,f_m)$ with each $f_j\in L^1
(\mathbb{R}^d)$ being approximately differentiable a.e.,
then $\mathscr{M}_{\mathcal{R}}(\vec{f})$ is approximately
differentiable a.e.
\end{theorem}
\remark{Since every function in $W^{1,1}(\mathbb{R}^d)$ space
is approximately differentiable a.e., thus Theorem \ref{thm6}
yields that if each
$f_j\in W^{1,1}(\mathbb{R}^d)$, then
$\mathscr{M}_{\mathcal{R}}(\vec{f})$ is approximately
differentiable a.e. However, it is unknown that whether
$\mathscr{M}_{\mathcal{R}}(\vec{f})$ is weak differentiable
when each $f_j\in W^{1,1}(\mathbb{R}^d)$, even in the case $m=1$
and $d\geq2$.}

As for the discrete type strong maximal functions, we have
the following conclusion.
\begin{theorem}  [\textbf{Properties of discrete strong
maximal functions}]\label{thm7}Let $d\geq1$ and $m\geq2$. Then
$\mathbb{M}_{\mathcal{R}}$ is bounded and continuous
from $\ell^1(\mathbb{Z}^d)\times\cdots\times\ell^1
(\mathbb{Z}^d)$ to ${\rm BV}(\mathbb{Z}^d)$.
Equivalently, the operator $\vec{f}\mapsto\nabla
\mathbb{M}_{\mathcal{R}}(\vec{f})$ is bounded and
continuous from $\ell^1(\mathbb{Z}^d)\times\cdots
\times\ell^1(\mathbb{Z}^d)$ to $\ell^1(\mathbb{Z}^d)$.
Moreover, if $f_j\in\ell^1(\mathbb{Z}^d)$ for
$1\le j\le m$. Then
$$
\|\nabla\mathbb{M}_{\mathcal{R}}(\vec{f})\|_{\ell^1(\mathbb{Z}^d)}
\lesssim_{d}\sum\limits_{l=1}^{m}\|\nabla f_l\|_{\ell^1(\mathbb{Z}^d)}
\prod\limits_{j\neq l,1\leq j\leq m}\|f_j\|_{\ell^1(\mathbb{Z}^d)}.
$$
\end{theorem}
\remark \label {1.8}we need to address the facts that:
\begin{enumerate}
\item[{\text{(i)}}] $\mathbb{M}_{\mathcal{R}}$ is bounded
and continuous from $W^{1,p_1}(\mathbb{Z}^d)\times\cdots
\times W^{1,p_m}(\mathbb{Z}^d)$ to $W^{1,p}(\mathbb{Z}^d)$
for all $1<p_1,\ldots,p_m,p\leq\infty$ and $1/p=\sum_{i=1}^m
1/p_i$. This conclusion is basically implied by the following
two facts. First, one can check that $\mathbb{M}_{\mathcal{R}}$
is bounded from $\ell^{p_1}(\mathbb{Z}^d)\times\cdots\times
\ell^{p_m}(\mathbb{Z}^d)$ to $\ell^p(\mathbb{Z}^d)$. Secondly,
it holds easily that
$
|\mathbb{M}_{\mathcal{R}}(\vec{f})-\mathbb{M}_{\mathcal{R}}(\vec{g})|
\leq\sum_{\mu=1}^m\mathbb{M}_{\mathcal{R}}(\vec{F}_{\mu}),
$
where $\vec{f}=(f_1,\ldots,f_m)$, $\vec{g}=(g_1,\ldots,g_m)$
and $\vec{F}_\mu=(f_1,\ldots,f_{\mu-1},f_\mu-g_\mu,g_{\mu+1},
\ldots,g_{m})$. This together with \eqref{1.5} implies the 
continuity for $\mathbb{M}_{\mathcal{R}}$ from $W^{1,p_1}
(\mathbb{Z}^d)\times\cdots\times W^{1,p_m}(\mathbb{Z}^d)$ 
to $W^{1,p}(\mathbb{Z}^d)$;

\item[{\text{(ii)}}] When $d\geq2$, the operator $f\mapsto
\nabla M_{\mathcal{R}}f$ is bounded and continuous from
$\ell^1(\mathbb{Z}^d)$ to $\ell^p(\mathbb{Z}^d)$ for
$1<p\leq\infty$. However, the operator $f\mapsto
\nabla M_{\mathcal{R}}f$ is not bounded from $\ell^1
(\mathbb{Z}^d)$ to $\ell^1(\mathbb{Z}^d)$. This
conclusions are basically implied by two facts. First, 
one can easily check that the operator $f\mapsto\nabla 
M_{\mathcal{R}}f$ is bounded and continuous from 
$\ell^1(\mathbb{Z}^d)$ to $\ell^p(\mathbb{Z}^d)$. Secondly, 
let $f(\vec{n})=\chi_{\{\vec{0}\}}(\vec{n})$. Note that 
$\|f\|_{\ell^1(\mathbb{Z}^d)}=1$ and $M_{\mathcal{R}}f
(\vec{n})=\prod_{i=1}^d(|n_i|+1)^{-1}$ for each 
$\vec{n}=(n_1,\ldots,n_d)\in\mathbb{Z}^d$. It follows 
that $\|\nabla\mathbb{M}_{\mathcal{R}}f\|_{\ell^1
(\mathbb{Z}^d)}=+\infty$. Thus, the operator $f\mapsto
\nabla M_{\mathcal{R}}f$ is not bounded from $\ell^1
(\mathbb{Z}^d)$ to $\ell^1(\mathbb{Z}^d)$;

\item[{\text{(iii)}}] When $d\geq2$, from Remark (ii)
we know that the discrete strong maximal operator
$M_{\mathcal{R}}$ is not bounded from $\ell^1(\mathbb{Z}^d)$
to ${\rm BV}(\mathbb{Z}^d)$. However, it was known that the
discrete maximal operator $M$ is bounded from $\ell^1
(\mathbb{Z}^d)$ to ${\rm BV}(\mathbb{Z}^d)$. Thus, the
regularity property of discrete strong maximal operator
$M_{\mathcal{R}}$ is worse than that of $M$ when $d\geq2$;

\item[{\text{(iv)}}] The proof of Theorem A in \cite{LW3}
depends highly on a summability argument over the sequence
of local maximal, local minimal of discrete multilinear
maximal functions and the Brezis-Lieb lemma \cite{BL}.
However, in the proof of Theorem \ref{thm7}, the above
techniques are unnecessary and our proofs are more simple,
direct and different than those in \cite{LW3}.
\end{enumerate}
By (ii) of Remark \ref{1.8}, we can get the following result
immediately.

\begin{corollary}Let $d\geq2$. Then the map $f\mapsto
\nabla M_{\mathcal{R}}f$ is bounded from
$\ell^1(\mathbb{Z}^d)$ to $\ell^q(\mathbb{Z}^d)$ if
and only if $q>1$.
\end{corollary}

This paper will be organized as follows. Section \ref{S2}
will be devoted to present the proof of Theorem \ref{thm1}.
Section \ref{S3} will be devoted to give the proofs of
Theorems \ref{thm2} and \ref{thm3}. The proofs of Theorems
\ref{thm5} and \ref{thm6} will be given in Sections \ref{S4}
and \ref{S5}, respectively. In Section \ref{S6}, we shall
prove Theorem \ref{thm7}. Finally, we introduce some
properties of $u_{x,\vec{f}}$ in Section \ref{S7}. We would
like to remark that the main ideas employed in the proofs
of Theorems \ref{thm1} and \ref{thm7} are greatly motivated
by \cite{Ki,Lu1}, but our methods and techniques are more
delicate and complex than those in \cite{Ki,Lu1}. It should
be pointed out that the main ideas in the proofs of Theorems
\ref{thm2} and \ref{thm3} are motivated by \cite{LW4}. Our
arguments in the proof of the bounded part in Theorem
\ref{thm7} are motivated by \cite{CM1}, but our methods and
techniques are somewhat different and direct than those in
\cite{CM1}. In addition, the Brezis-Lieb lemma \cite{BL} is
not necessary in the proof of the continuity part of Theorem
\ref{thm7}.

Throughout this paper, if there exists a constant $c>0$
depending only on $\vartheta$ such that $A\leq cB$, we then
write $A\lesssim_{\vartheta}B$ or $B\gtrsim_\vartheta A$;
and if $A\lesssim_\vartheta B\lesssim_\vartheta A$, we then
write $A\thicksim_\vartheta B$.

\section {Properties on Sobolev spaces}\label{S2}
\subsection{Prelimary lemmas}
We first present several preliminary lemmas, which play
important roles in the proof of Theorem 1.1. Some basic
ideas will be taken from \cite{Lu1}, where the proof
for the continuity in $W^{1,p}(\mathbb{R}^d)$ of the
Hardy-Littlewood maximal operator has been given. We
only consider the case $d=2$ and other cases are
analogous and more complex.

For $A\subset\mathbb{R}^2$ and $x\in\mathbb{R}^2$, define
$$
d(x,A):=\inf\limits_{a\in A}|x-a|\ \ {\rm and}
\ A_{(\lambda)}:=\{x\in\mathbb{R}^2;d(x,A)\leq\lambda\}
\ \ {\rm for}\ \lambda\geq0.
$$
We denote by $\|f\|_{p,A}$ the $L^p$-norm of $f\chi_{A}$ for
all measurable sets $A\subset\mathbb{R}^2$. Let $1/p=
\sum_{j=1}^m1/p_j$ and $1<p_1,p_2,\ldots,p_m,p<\infty$. Let 
$\vec{f}=(f_1,\ldots,f_m)$ with each $f_j\in L^{p_j}
(\mathbb{R}^2)$. For convenience, we set $\mathbb{R}_{+}=(0,
\infty)$ and $\overline{\mathbb{R}}_{+}=[0,\infty)$. We also
set
$$(\overline{\mathbb{R}}_{+})_1^4=\{(r_1,r_2,0,0):(r_1,r_2)\in\overline{\mathbb{R}}_{+}^2, r_1+r_2>0\},$$
$$(\overline{\mathbb{R}}_{+})_2^4=\{(0,0,r_3,r_4):(r_3,r_4)\in\overline{\mathbb{R}}_{+}^2, r_3+r_4>0\},$$
$$(\overline{\mathbb{R}}_{+})_{1,2}^4=\{(r_1,r_2,r_3,r_4):(r_1,r_2,r_3,r_4)\in\overline{\mathbb{R}}_{+}^4,r_1+r_2>0,r_3+r_4>0\}.$$
Define the function $u_{(x_1,x_2),\vec{f}}:\overline
{\mathbb{R}}_{+}^4\rightarrow\mathbb{R}$ by
$$
u_{(x_1,x_2),\vec{f}}(r_{1,1},r_{1,2},0,0):=\frac{1}{(r_{1,1}+r_{1,2})^m}\prod\limits_{j=1}^m
\int_{x_1-r_{1,1}}^{x_1+r_{1,2}}|f_j(y_1,x_2)|dy_1
\ \ \text{ for\ }(r_{1,1},r_{1,2},0,0)\in(\overline{\mathbb{R}}_{+})_1^4;
$$
$$
u_{(x_1,x_2),\vec{f}}(0,0,r_{2,1},r_{2,2}):=\frac{1}{(r_{2,1}+r_{2,2})^m}\prod\limits_{j=1}^m
\int_{x_2-r_{2,1}}^{x_2+r_{2,2}}|f_j(x_1,y_2)|dy_2
\ \ \text{ for\ }(0,0,r_{2,1},r_{2,2})\in(\overline{\mathbb{R}}_{+})_2^4;
$$
$$
\aligned
u_{(x_1,x_2),\vec{f}}(r_{1,1},r_{1,2},r_{2,1},r_{2,2})
&:=\displaystyle\prod\limits_{i=1}^2\frac{1}{(r_{i,1}+r_{i,2})^m}
\prod\limits_{j=1}^m\int_{x_1-r_{1,1}}^{x_1+r_{1,2}}\int_{x_2-r_{2,1}}^{x_2+r_{2,2}}
|f_j(y_1,y_2)|dy_1dy_2,  \\& \quad\quad\quad\quad\quad\quad\quad\quad\quad\quad\quad\text{ for \ }(r_{1,1},r_{1,2},r_{2,1},r_{2,2})\in(\overline{\mathbb{R}}_{+})_{1,2}^4.
\endaligned
$$
In particular, we denote
$u_{(x_1,x_2),\vec{f}}(0,0,0,0)=\prod_{j=1}^m|f_j(x_1,x_2)|$.
We can write
$$
{\mathscr{M}}_{\mathcal{R}}(\vec{f})(x)
=\sup\limits_{r_{1,1},r_{1,2},r_{2,1},r_{2,2}>0}u_{(x_1,x_2),\vec{f}}(r_{1,1},r_{1,2},r_{2,1},r_{2,2}).
$$

For a fixed point $x=(x_1,x_2)\in\mathbb{R}^2$, we define
the sets
$\mathcal{B}_i(\vec{f})(x_1,x_2)$ ($i=1,2,3$) by
$$
\aligned
\mathcal{B}_1(\vec{f})(x_1,x_2)&:=\Big\{(r_{1,1},r_{1,2},r_{2,1},r_{2,2})
\in\overline{\mathbb{R}}_{+}^4:
\  {\mathscr{M}}_{\mathcal{R}}(\vec{f})(x_1,x_2)=
\\
& \displaystyle
\quad \limsup\limits_{\substack{(r_{1,1,k},r_{1,2,k},r_{2,1,k},r_{2,2,k})
\\ \rightarrow(r_{1,1},r_{1,2},r_{2,1},r_{2,2})}}u_{(x_1,x_2),\vec{f}}(r_{1,1,k},r_{1,2,k},r_{2,1,k},r_{2,2,k})\\
&\quad\quad\quad\quad\quad\quad\quad\quad\quad\quad\quad\quad\quad\quad\ {\rm for\ some}\ r_{1,1,k},\,r_{1,2,k},\,r_{2,1,k},\,r_{2,2,k}>0\Big\}.
\endaligned
$$
$$
\aligned
\mathcal{B}_2(\vec{f})(x_1,x_2)
&:=\Big\{(r_{1,1},r_{1,2},r_{2,1},r_{2,2})
\in\overline{\mathbb{R}}_{+}^2\times\{(0,0)\}:
 \displaystyle\mathscr{M}_{\mathcal{R}}(\vec {f})(x_1,x_2)=
\\&\quad\quad\limsup\limits_{(r_{1,1,k},r_{1,2,k})\rightarrow (r_{1,1},r_{1,2})}u_{(x_1,x_2),\vec{f}}(r_{1,1,k},r_{1,2,k},0,0)\ \ {\rm for\ some}\ r_{1,1,k},r_{1,2,k}>0\Big\}.
\endaligned
$$
$$
\aligned
\mathcal{B}_3(\vec{f})(x_1,x_2)
 \displaystyle
&:=\Big\{(r_1,r_2)\in\{(0,0)\}\times\overline{\mathbb{R}}_{+}^2:
 \displaystyle\ \mathscr{M}_{\mathcal{R}}(\vec{f})(x_1,x_2)=
\\&\quad\quad\limsup\limits_{(r_{2,1,k},r_{2,2,k})\rightarrow(r_{2,1},r_{2,2})}u_{(x_1,x_2),\vec{f}}(0,0,r_{2,1,k},r_{2,2,k})\ \ {\rm for\ some}\ r_{2,1,k},r_{2,2,k}>0\Big\}.
\endaligned
$$
The function $u_{(x_1,x_2),\vec{f}}$ enjoys the following
properties:

\newpage
%%%%%%%%%%%%%%%% Lemma 2.1 %%%%%%%%%%%%%%%%%%%%%%%%%%
\begin{lemma}\label{l2.1}
Let $\vec{f}=(f_1,\ldots,f_m)$ with each $f_j\in L^{p_j}
(\mathbb{R}^2)$ for $1<p_j<\infty$, $(j=1,2,\dots,m)$. Then
the following statements hold:
\begin{enumerate}
\item[{\text{(i)}}] {\it
$u_{(x_1,x_2),\vec{f}}$ are continuous on
$(\overline{\mathbb{R}}_{+})_+^4:=\{(r_{1,1},r_{1,2},r_{2,1},
r_{2,2})\in\overline{\mathbb{R}}_{+}^4: r_{1,1}+r_{1,2},
r_{2,1}+r_{2,2}>0\}$ for all $(x_1,x_2)\in\mathbb{R}^2$, and
continuous on $\overline{\mathbb{R}}_{+}^4$ for a.e. $(x_1,
x_2)\in\mathbb{R}^2$;
$$
\lim_{(r_{1,1},r_{1,2},r_{2,1},r_{2,2})\in\overline{\mathbb{R}}_{+}^4\atop
r_{1,1}+r_{1,2},r_{2,1}+r_{2,2}\rightarrow\infty}
u_{(x_1,x_2),\vec{f}}(r_{1,1},r_{1,2},r_{2,1},r_{2,2})=0,\quad \hbox{for}\ a.e.\  (x_1,x_2)\in\mathbb{R}^2;
$$$\mathcal{B}_1
(\vec{f})(x_1,x_2)$ are nonempty and closed for every
$(x_1,x_2)\in\mathbb{R}^2$;}

\item[{(ii)}] {\it $u_{(x_1,x_2),\vec{f}}$ are continuous on
$\{(r_{1,1},r_{1,2})\in\overline{\mathbb{R}}_{+}^2: r_{1,1}+
r_{1,2}>0\}\times\{(0,0)\}$ for all $x_1\in\mathbb R$ and
a.e. $x_2\in\mathbb{R}$, and continuous at $(0,0,0,0)$ for
a.e. $(x_1,x_2)\in\mathbb{R}^2$;
$$
\lim_{(r_{1,1},r_{1,2})\in\overline{\mathbb{R}}_{+}^2\atop
r_{1,1}+r_{1,2}\rightarrow\infty}u_{(x_1,x_2),\vec{f}}(r_{1,1},r_{1,2},0,0)=0,\quad \hbox{for all} \ x_1\in\mathbb R\quad \hbox{and a.e. }\ x_2\in\mathbb{R};
$$
$\mathcal{B}_2(\vec{f})(x_1,x_2)$ are nonempty and closed
for a.e. $(x_1,x_2)\in\mathbb{R}^2$;}

\item[{(iii)}]  {\it $u_{(x_1,x_2),\vec{f}}$ are continuous
on $\{(0,0)\}\times\{(r_{2,1},r_{2,2})\in\overline
{\mathbb{R}}_{+}^2: r_{2,1}+r_{2,2}>0\}$ for all
$x_2\in\mathbb R$ and a.e. $x_1\in\mathbb{R}$ and continuous
at $(0,0,0,0)$ for a.e. $(x_1,x_2)\in\mathbb{R}^2$;
$$
\lim_{(r_{2,1},r_{2,2})\in\overline{\mathbb{R}}_{+}^2\atop
r_{2,1}+r_{2,2}\rightarrow\infty}
u_{(x_1,x_2),\vec{f}}(0,0,r_{2,1},r_{2,2})=0,\quad \hbox{for all} \ x_2\in\mathbb R\quad \hbox{and a.e. }\ x_1\in\mathbb{R};
$$
$\mathcal{B}_3(\vec{f})(x_1,x_2)$ are nonempty and closed
for a.e. $(x_1,x_2)\in\mathbb{R}^2$.}
\end{enumerate}
\end{lemma}
%%%%%%%%%%%%%%%%%%%%%%%%%%%%%%%%%%%%%%%%%%%%%
\begin{proof}(i) The first statement follows from the
integrability of $f_j$. The proof of the continuity on
$\overline{\mathbb{R}}_{+}^4$ for a.e. $(x_1,x_2)\in
\mathbb{R}^2$ is  very delicate. So, we shall prove it
in the last section. We can see easily that for any
$(x_1,x_2)\in\mathbb{R}^2$, it holds that
$$
\lim_{(r_{1,1},r_{1,2},r_{2,1},r_{2,2})\in\overline{\mathbb{R}}_{+}^4\atop
r_{1,1}+r_{1,2},r_{2,1}+r_{2,2}\rightarrow\infty}u_{(x_1,x_2),
\vec{f}}(r_{1,1},r_{1,2},r_{2,1},r_{2,2})=0,
$$
since $u_{(x_1,x_2),f}(r_{1,1},r_{1,2},r_{2,1},r_{2,2})\leq
|(r_{1,1}+r_{1,2})(r_{2,1}+r_{2,2})|^{-1/p}\prod_{j=1}^m
\|f_j\|_{L^{p_j}(\mathbb{R}^2)}$. But, when $0<r_{1,1}+r_{1,2}
+r_{2,1}+r_{2,2}\rightarrow\infty$, we should treat more
carefully, and we shall prove it in the last section. The last
statement can be checked easily.

(ii) The first statement follows from the integrability of
$f_j$. The continuity at $(0,0,0,0)$ will be checked in the
last section. Since $u_{(x_1,x_2),f}(\vec r)\leq |r_{1,1}+
r_{1,2}|^{-1/p}\prod_{j=1}^m\|f_j(\cdot, x_2)\|_{L^{p_j}
(\mathbb{R})}$ for any $(r_{1,1},r_{1,2},0,0)\in
\mathbb{R}_{+}\times\{(0,0)\}$ and all $x_1\in\mathbb R$ and
a.e. $x_2\in\mathbb{R}$, we get
$$
\lim_{(r_{1,1},r_{1,2})\in\overline{\mathbb{R}}_{+}^2\atop
r_{1,1}+r_{1,2}\rightarrow\infty}u_{(x_1,x_2),\vec{f}}(r_{1,1},r_{1,2},0,0)=0.
$$
The last statement can be checked easily.

(iii) (iii) is the same as in (ii).
\end{proof}
\begin{lemma}\label {l2.2}

The following
relationships between $ {\mathscr{M}}_{\mathcal{R}}
(\vec{f})$ and $u_{(x_1,x_2),\vec{f}}$ are valid. \begin{enumerate}
\item[{(iv)}]
${\mathscr{M}}_{\mathcal{R}}(\vec{f})(x_1,x_2)=
\prod\limits_{j=1}^m|f_j(x_1,x_2)|$ for a.e. $(x_1,x_2)\in
\mathbb{R}^2$ such that $\vec{0}\in\bigcup\limits_{i=1}^3
\mathcal{B}_i(f)(x_1,x_2)$;

\item[{(v)}]${\mathscr{M}}_{\mathcal{R}}(\vec{f})(x_1,x_2)=u_{(x_1,x_2),
\vec{f}}(\vec{r})$ for a.e. $(x_1,x_2)\in\mathbb{R}^2$
such that $\vec{r}\in\mathcal{B}_1(\vec{f})(x_1,x_2)
\cap(\overline{\mathbb{R}}_{+})_{1}^4;$
\item[{(vi)}]${\mathscr{M}}_{\mathcal{R}}(\vec{f})(x_1,x_2)=u_{(x_1,x_2),
\vec{f}}(\vec{r})$ for a.e. $(x_1,x_2)\in\mathbb{R}^2$
such that $\vec{r}\in\mathcal{B}_1(\vec{f})(x_1,x_2)
\cap(\overline{\mathbb{R}}_{+})_{2}^4;$
\item[{(vii)}]${\mathscr{M}}_{\mathcal{R}}(\vec{f})(x_1,x_2)=u_{(x_1,x_2),
\vec{f}}(\vec r)$ for a.e. $(x_1,x_2)\in\mathbb{R}^2$
such that $\vec r\in\mathcal{B}_2(\vec{f})(x_1,x_2)$;

\item[{(viii)}]${\mathscr{M}}_{\mathcal{R}}(\vec{f})
(x_1,x_2)=u_{(x_1,x_2),\vec{f}}(\vec r)$ for a.e.
$(x_1,x_2)\in\mathbb{R}^2$ such that $\vec r\in
\mathcal{B}_3(\vec{f})(x_1,x_2)$;

\item[{(ix)}] $ {\mathscr{M}}_{\mathcal{R}}(\vec{f})(x_1,x_2)=
u_{(x_1,x_2),\vec{f}}(\vec r)$ if $\vec r\in\mathcal{B}_1
(\vec{f})(x)\cap(\overline{\mathbb{R}}_{+})_{1,2}^4$ for all
$x\in\mathbb{R}^2$.
\end{enumerate}
\end{lemma}
For convenience, for any $\vec{r}=(r_1,r_2,\ldots,r_d)\in
\mathbb{R}_{+}^d$ and $x=(x_1,x_2,\ldots,x_d)\in\mathbb{R}^d$,
we let
$$R_{\vec{r}}(x)=\{y=(y_1,y_2,\ldots,y_d)\in\mathbb{R}^d;
|y_i-x_i|<r_i,\,i=1,2,\ldots,d\}.$$
%%%%%%%%%%  Lemma 2.3  %%%%%%%%%%%%%%%%%%%%%%%%%%%%%%%%%%%
\begin{lemma}\label{l2.3}
Let $\Lambda>0$, $\vec{\Lambda}=(\Lambda,\Lambda)$
and $\vec{0}=(0,0)$, $1<p,p_i<\infty$ and $1/p=\sum_{i=1}^m
1/p_i$, and $\vec{f}=(f_1,\ldots,f_m)$ with each $f_i\in L^{p_i}
(\mathbb{R}^2)$. Let $\vec{f}_j=(f_{1,j},
\ldots,f_{m,j})$ such that $f_{i,j}\rightarrow f_i$ in
$L^{p_i}(\mathbb{R}^2)$ when $j\rightarrow\infty$ for all
$i=1,2,\ldots,m$. Then for all $\lambda>0$ and
$i=1,2,3$, we have

\begin{equation}\label{2.1}
\lim\limits_{j\rightarrow\infty}|\{x\in R_{\vec{\Lambda}}(\vec{0});
\mathcal{B}_i(\vec{f_j})(x)\nsubseteq\mathcal{B}_i(\vec{f})(x)_{(\lambda)}\}|
=0.\end{equation}
\end{lemma}
%%%%%%%%%%%  Proof of Lemma 2.3  %%%%%%%%%%%%%%%%%%%%%%%%
\begin{proof} Without loss of generality we may assume all
$f_{i,j}\geq0$ and $f_i\geq0$. We shall prove \eqref{2.1} 
for the case $i=1$ and the other cases are analogous. Let 
$\lambda>0$ and $\Lambda>0$. We first conclude that the set
$\{x\in\mathbb{R}^2;\ \mathcal{B}_1(\vec{f}_j)(x)\nsubseteq
\mathcal{B}_1(\vec{f})(x)_{(\lambda)}\}$ is measurable for
all $j\geq1$. To see this, let $E$ be the set of all points
which are not Lebesgue points of any of the functions
$f_{i,j}$ and $f_i$. Obviously, $|E|=0$. We denote by
$\mathbb{Q}_{+}$ the set of positive rationals. Fix $j\geq1$,
we can write
$$
\begin{array}{ll}
&\{x\in\mathbb{R}^2\setminus E:\mathcal{B}_1(\vec{f}_j)(x)
\nsubseteq\mathcal{B}_1(\vec{f})(x)_{\lambda}\}
\\
&=\displaystyle\bigcup\limits_{i=1}^\infty\bigcap\limits_{k=1}^\infty
\Big\{x\in\mathbb{R}^2:\exists \vec r\in\mathbb{R}_{+}^4\ {\rm s.t.}
\ d(\vec r,\mathcal{B}_1(\vec{f})(x))>\lambda+\frac{1}{i}\\
&\quad\ \displaystyle{\rm and}\ {\mathscr{M}}_{\mathcal{R}}(\vec{f}_j)(x)
<u_{x,\vec{f}_j}(\vec{r})+\frac{1}{k}\Big\}
\\
&=\displaystyle\bigcup\limits_{i=1}^\infty\bigcap\limits_{k=1}^\infty
\bigcup\limits_{\vec t\in\mathbb{Q}_{+}^4}
\Big(\Big\{x\in\mathbb{R}^2:d(\vec t,\mathcal{B}_1(\vec{f})(x))>\lambda+\frac{1}{i}\Big\}\bigcap
\\
&\quad\displaystyle\Big\{x\in\mathbb{R}^2:{\mathscr{M}}_{\mathcal{R}}
(\vec{f}_j)(x)<u_{x,\vec{f}_j}(\vec{t})+\frac{1}{k}\Big\}\Big).
\end{array}
$$
On the other hand, for any fixed $\vec{t}\in\mathbb{Q}_{+}^4$, 
we have
$$
\begin{array}{ll}
\quad\displaystyle\{x:d(\vec t,\mathcal{B}_1(\vec{f})(x))>\lambda\}
&=\displaystyle\bigcup\limits_{l=1}^\infty
\bigcap\limits_{\vec s\in\mathbb{Q}_{+}^4
\cap\{\vec s:|\vec s-\vec t|\leq\lambda\}}
\Big\{x\in\mathbb{R}^2: {\mathscr{M}}_{\mathcal{R}}(\vec{f})(x)
>u_{x,\vec{f}}(\vec{s})+\frac{1}{l}\Big\}.
\end{array}
$$
Therefore, we get the measurability of
$\{x\in\mathbb{R}^2;\ \mathcal{B}_1(\vec{f}_j)(x)\nsubseteq
\mathcal{B}_1(\vec{f})(x)_{(\lambda)}\}$ for any $j\geq1$.

Now, we claim that for a.e. $x\in R_{\vec{\Lambda}}(\vec{0})$,
there exists $\gamma(x)\in\mathbb{N}\setminus\{0\}$ such that
\begin{equation}\label{2.2}
u_{x,\vec{f}}(\vec{r})<{\mathscr{M}}_{\mathcal{R}}(\vec{f})(x)
-\frac{1}{\gamma(x)},
\ \ {\rm when}\ d(\vec{r},\mathcal{B}_1(\vec{f})(x))>\lambda.\end{equation}
Actually, if \eqref{2.2} does not hold, then for a.e.
$x\in R_{\vec{\Lambda}}(\vec{0})$, there exists a bounded
sequence of $\{\vec{r_k}\}_{k=1}^\infty$ such that
$$
\lim\limits_{k\rightarrow\infty}u_{x,\vec{f}}(\vec{r_k})
= {\mathscr{M}}_{\mathcal{R}}(\vec{f})(x)
\ \ {\rm and}\ d(\vec{r_k},\mathcal{B}_1(\vec{f})(x))>\lambda.
$$
Hence, we may choose a subsequence $\{\vec{s_k}\}_{k=1}^\infty$
of $\{\vec{r_k}\}_{k=1}^\infty$ such that $\vec{s_k}\rightarrow
\vec{r}$ as $k\rightarrow\infty$. It follows that $\vec{r}\in
\mathcal{B}_1(\vec{f})(x)$ and $d(\vec{r},\mathcal{B}_1
(\vec{f})(x))\geq\lambda$, which is a contradiction. Therefore,
\eqref{2.2} holds. Let
\begin{enumerate}
\item[{\text{}}]
$A_{1,j}:=\{x\in\mathbb{R}^2:| {\mathscr{M}}_{\mathcal{R}}(\vec{f}_j)(x)
- {\mathscr{M}}_{\mathcal{R}}(\vec{f})(x)|\geq(4\gamma)^{-1}\},$
\item[{\text{}}]
$A_{2,j}:=\{x\in\mathbb{R}^2: |u_{x,\vec{f}_j}(\vec{r})-u_{x,\vec{f}}(\vec{r})|
\geq(2\gamma)^{-1}
\ \ {\rm if }\ d(\vec{r},\mathcal{B}_1(\vec{f})(x))>\lambda\},$
\item[{\text{}}]
$A_{3,j}:=\{x\in\mathbb{R}^2: u_{x,\vec{f}_j}(\vec{r})
< {\mathscr{M}}_{\mathcal{R}}(\vec{f}_j)(x)-(4\gamma)^{-1},
\ \ {\rm if}\ d(\vec{r},\mathcal{B}_1(\vec{f})(x))>\lambda\}.$
\end{enumerate}
Given $\epsilon\in(0,1)$, we get from \eqref{2.2} that there 
exists $\gamma=\gamma(\Lambda,\lambda,\epsilon)\in\mathbb{N}
\setminus\{0\}$ and a measurable set $E_0$ with $|E_0|<\epsilon$ 
such that
\begin{equation}\label{2.3}
\begin{array}{ll}
R_{\vec{\Lambda}}(\vec{0})
&\subset \big\{x\in\mathbb{R}^2:u_{x,\vec{f}}(\vec{r})
< {\mathscr{M}}_{\mathcal{R}}(\vec{f})(x)-\gamma^{-1},
\ \ {\rm if}\ d(\vec{r},\mathcal{B}_1(\vec{f})(x))>\lambda\big\}\cup E_0
\\
&\subset A_{1,j}\cup A_{2,j}\cup A_{3,j}\cup E_0.
\end{array}\end{equation}

Let $\bar{A}$ be the set of all points $x$ such that $x$
is a Lebesgue point of all $f_j$. Note that
$|\mathbb{R}^{2}\setminus\bar{A}|=0$. One can easily check
that $A_{3,j}\cap\bar{A}\subset\{x\in\mathbb{R}^2:
\mathcal{B}_1(\vec{f}_j)(x)\subset\mathcal{B}_1(\vec{f})
(x)_{(\lambda)}\}$. This together with \eqref{2.3} yields 
that
\begin{equation}\label{2.4}
\{x\in R_{\vec{\Lambda}}(\vec{0});\mathcal{B}_1(\vec{f}_j)(x)\nsubseteq
\mathcal{B}_1(\vec{f})(x)_{(\lambda)}\}\subset A_{1,j}\cup A_{2,j}
\cup E_0\cup(\mathbb{R}^{2}\setminus\bar{A}).
\end{equation}
Since $f_{i,j}\rightarrow f_i$ in $L^{p_i}(\mathbb{R}^2)$
when $j\rightarrow\infty$, then there exists $N_i=N_i
(\epsilon,\gamma)\in\mathbb{N}$ such that
\begin{equation}\label{2.5}
\|f_{i,j}-f_i\|_{L^{p_i}(\mathbb{R}^2)}<\frac{\epsilon}{\gamma}\ \ {\rm and}\ \|f_{i,j}\|_{L^{p_i}(\mathbb{R}^2)}\leq \|f_i\|_{L^{p_i}(\mathbb{R}^2)}+1\ \ \forall j\geq N_i.
\end{equation}
Moreover, it holds that
\begin{equation}\label{2.6}
\aligned
{}&| {\mathscr{M}}_{\mathcal{R}}(\vec{f}_j)(x)- {\mathscr{M}}_{\mathcal{R}}(\vec{f})(x)|\\
&\leq\displaystyle\sup_{\substack{R \ni x \\ R\in\mathcal{R}}}\frac{1}{|R|^{m}}\Big|\prod\limits_{i=1}^m\int_{R}f_{i,j}(y)dy-\prod\limits_{i=1}^m\int_{R}f_{i}(y)dy\Big|\\
&\leq\displaystyle\sum\limits_{l=1}^m\sup_{\substack{R \ni x \\ R\in\mathcal{R}}}\frac{1}{|R|^{m}}\prod\limits_{\mu=1}^{l-1}\int_{R}f_{\mu}(y)dy\prod\limits_{\nu=l+1}^{m}\int_{R}f_{\nu,j}(y)dy
\int_{R}|f_{l,j}(y)-f_l(y)|dy\\
&\leq\displaystyle\sum\limits_{l=1}^m\mathscr{M}_{\mathcal{R}}(\vec{F}_j^l)(x),
\endaligned \end{equation}
where $\vec{F}_j^l=(f_1,\ldots,f_{l-1}, f_{l,j}-f_l,
f_{l+1,j},\ldots,f_{m,j})$.
Let $N_0=\max_{1\leq j\leq m}N_j$.Then, for any $j\geq N_0$, we get from \eqref{2.5} and \eqref{2.6} that
\begin{equation}\label{2.7}
\aligned
|A_{1,j}|&\leq(4\gamma)^p
\| {\mathscr{M}}_{\mathcal{R}}(\vec{f_j})
- {\mathscr{M}}_{\mathcal{R}}(\vec{f})\|_{L^p(\mathbb{R}^2)}^p
\\&\leq\displaystyle(4\gamma m)^p\sum\limits_{l=1}^m\prod\limits_{\mu=1}^{l-1}\|f_\mu\|_{L^{p_{\mu}}(\mathbb{R}^2)}^p
\prod\limits_{\nu=l+1}^{m}\|f_{\nu,j}\|_{L^{p_{\nu}}(\mathbb{R}^d)}^p\|f_{l,j}-f_l\|_{L^{p_l}(\mathbb{R}^2)}^p\\
&\lesssim_{m,p_1,\ldots,p_m,p}\epsilon.
\endaligned \end{equation}
Since $$|u_{x,\vec{f}_j}(\vec{r})-u_{x,\vec{f}}(\vec{r})|
\leq\displaystyle\sum\limits_{l=1}^m\mathscr{M}_{\mathcal{R}}(\vec{F}_j^l)(x).$$
Similarly, $|A_{2,j}|\lesssim_{m,p_1,\ldots,p_m,p}\epsilon$
for any $j\geq N_0$. This together with \eqref{2.4} and 
\eqref{2.7} yields \eqref{2.1}.
\end{proof}

For any fixed $h>0$ and $f_i\in L^{p_i}(\mathbb{R}^2)$
with $1<p_i<\infty$, define
$$(f_i)_{h}^l(x)=\frac{(f_i)_{\tau(h)}^{l}(x)-f_i(x)}{h}\ \ {\rm and}\ \ (f_i)_{\tau(h)}^{l}(x)=f_i(x+h\vec{e}_l),\ \ l=1,2.$$
It is well known that for $l=1,2$ and $1<p_i<\infty$,
$(f_i)_{\tau(h)}^{l}\rightarrow f_i$ in $L^{p_i}
(\mathbb{R}^2)$ when $h\rightarrow0$, and if
$f_i\in W^{1,p_i}(\mathbb{R}^2)$ we have
$(f_i)_{h}^l\rightarrow D_l f_i$ in $L^{p_i}(\mathbb{R}^2)$
when $h\rightarrow0$ (see \cite{GT}). Let $A,\,B$ be two
subsets of $\mathbb{R}^2$, we define the Hausdorff
distance of $A$ and $B$ by
$$\pi(A,B):=\inf\{\delta>0: A\subset B_{(\delta)}\ {\rm and}\ B\subset A_{(\delta)}\}.$$

Applying Lemma \ref{l2.3}, we can get the following 
corollary.
%%%%%%%%%%%%%%%%  Corollary 2.4  %%%%%%%%%%%%%%%%%%%%%%%%%%%%%%%%%%%%%
\begin{corollary}\label{c2.4}
Let $f_i\in L^{p_i}(\mathbb{R}^2)$ with $1<p_i<\infty$.
Then for all $\Lambda>0$, $\lambda>0$, $i=1,2,3$ and
$l=1,2$, we have
$$
\lim\limits_{h\rightarrow0}|\{x\in R_{\vec{\Lambda}}(\vec{0});\pi(\mathcal{B}_i(\vec{f})(x),\mathcal{B}_i(\vec{f})(x+h\vec{e}_l))>\lambda\}|=0.
$$
\end{corollary}
%%%%%%%%%%%%%%%%  end Corollary 2.4  %%%%%%%%%%%%%%%%%%%%%%%%%%%%%%%%%%%%%
\begin{proof}
Fix $i\in\{1,2,3\}$. It suffices to show that
\begin{equation}\label{2.8}
\lim\limits_{h\rightarrow0}|\{x\in R_{\vec{\Lambda}}(\vec{0}):\mathcal{B}_i(\vec{f})(x)\nsubseteq\mathcal{B}_i(\vec{f})(x+h\vec{e}_l)_{(\lambda)}\ {\rm or}\ \mathcal{B}_i(\vec{f})(x+h\vec{e}_l)\nsubseteq\mathcal{B}_i(\vec{f})(x)_{(\lambda)}\}|=0.
\end{equation}
One can easily check that
$$\mathcal{B}_i(\vec{f})(x+h\vec{e}_l)=\mathcal{B}_i(\vec{f}_{\tau(h)}^l)(x)\ {\rm and}\ \mathcal{B}_i(\vec{f})(x)=\mathcal{B}_i(\vec{f}_{\tau(-h)}^l)(x+h\vec{e}_l).$$
Here $\vec{f}_{\tau(h)}^l=(f_1(x+h\vec{e}_l),\ldots,f_m
(x+h\vec{e}_l))$. It follows that
\begin{equation}\label{2.9}
\begin{array}{ll}{}
&\{x\in R_{\vec{\Lambda}}(\vec{0}):\mathcal{B}_i(\vec{f})(x)\nsubseteq\mathcal{B}_i(\vec{f})(x+h\vec{e}_l)_{(\lambda)}\}\\
&=\{x\in R_{\vec{\Lambda}}(\vec{0}):\mathcal{B}_i(\vec{f}_{\tau(-h)}^l)(x+h\vec{e}_l)\nsubseteq\mathcal{B}_i(\vec{f})(x+h\vec{e}_l)_{(\lambda)}\}\\
&\subset\{x\in R_{\vec{\Lambda+1}}(\vec{0}):\mathcal{B}_i(\vec{f}_{\tau(-h)}^l)(x)\nsubseteq\mathcal{B}_i(\vec{f})(x)_{(\lambda)}\}-h\vec{e}_l
\end{array}\end{equation}
where $\vec{\Lambda+1}=(\Lambda+1,\Lambda+1)$ and
$|h|\leq1$. Moreover,
\begin{equation}\label{2.10}
\{x\in R_{\vec{\Lambda}}(\vec{0}):\mathcal{B}_i(\vec{f})(x+he_l)\nsubseteq\mathcal{B}_i(\vec{f})(x)_{(\lambda)}\}=\{x\in R_{\vec{\Lambda}}(\vec{0}):\mathcal{B}_i(\vec{f}_{\tau(h)}^l)(x)_{(\lambda)}\nsubseteq\mathcal{B}_i(\vec{f})(x)_{(\lambda)}\}.\end{equation}
We note that $(f_i)_{\tau(h)}^{l}\rightarrow f_i$ in
$L^{p_i}(\mathbb{R}^2)$ when $h\rightarrow0$. By Lemma
\ref{l2.3}, it yields that
\begin{equation}\label{2.11}
\lim\limits_{h\rightarrow0}|\{x\in R_{\vec{\Lambda+1}}(\vec{0}):\mathcal{B}_i(\vec{f}_{\tau(-h)}^l)(x)\nsubseteq\mathcal{B}_i(\vec{f})(x)_{(\lambda)}\}|=0\end{equation}
and
\begin{equation}\label{2.12}
\lim\limits_{h\rightarrow0}|\{x\in R_{\vec{\Lambda}}(\vec{0})
:\mathcal{B}_i(\vec{f}_{\tau(h)}^l)(x)_{(\lambda)}\nsubseteq
\mathcal{B}_i(\vec{f})(x)_{(\lambda)}\}|=0.\end{equation}
Now, it is easy to see that \eqref{2.8} follows from 
\eqref{2.9}-\eqref{2.12}.
\end{proof}
%%%%%%%%%%%  end  Proof of Corollary 2.4  %%%%%%%%%%%%%%%%%%%%%%%%

We now state some formulas for the derivatives of the multilinear
strong maximal functions, which provide a foundation for our
analysis in the continuity part of Theorem \ref{thm1}.

\medskip
%%%%%%%%%  Lemma 2.4  %%%%%%%%%%%%%%%%%%%%%%%%%%%%%%%%%%%%%%%%%
\begin{lemma}\label {l2.5} {\it Let $f_i\in W^{1,p_i}
(\mathbb{R}^2)$ with $1<p_i<\infty$. Then for any $l=1,2$
and almost every $(x_1,x_2)\in\mathbb{R}^2$, we have
\begin{enumerate}
\item[{\rm (i)}] %%%%%%%%%%%% (i) %%%%%%%%%%%%%%%%%%%%%%%%%%%%%%
For all $\vec r\in\mathcal{B}_1(\vec{f})(x_1,x_2)$
 with $\ r_{1,1}+r_{1,2}>0,\,r_{2,1}+r_{2,2}>0$, it holds that
\begin{equation}
\begin{array}{ll}\label{2.13}
{}&D_l {\mathscr{M}}_{\mathcal{R}}\vec{f}(x_1,x_2)
\\&
=\sum\limits_{\mu=1}^m\displaystyle
\frac{1}{(r_{1,1}+r_{1,2})^m(r_{2,1}+r_{2,2})^m}
\prod\limits_{j\neq \mu,1\le j \le m}^m
\int_{x_1-r_{1,1}}^{x_1+r_{1,2}}
\int_{x_2-r_{2,1}}^{x_2+r_{2,2}}|f_j(y_1,y_2)|dy_1dy_2
\\
&\quad \hspace{6.5cm}\times\displaystyle\int_{x_1-r_{1,1}}^{x_1+r_{1,2}}
\int_{x_2-r_{2,1}}^{x_2+r_{2,2}}D_l|f_{\mu} (y_1,y_2)|dy_1dy_2\end{array}
\end{equation}
\item[{\rm (ii)}]  %%%%%%%%%%% (ii) %%%%%%%%%%%%%%%%%%%%%%%%%%%%%%%%
For all $ \vec r\in \mathcal{B}_1(\vec{f})(x_1,x_2)\cup
\mathcal{B}_2(\vec{f})(x_1,x_2)$ with $r_{1,1}+r_{1,2}>0,\,r_{2,1}=r_{2,2}=0$, 
we have
\begin{equation}\aligned\label{2.14}
{}&D_l {\mathscr{M}}_{\mathcal{R}}\vec{f}(x_1,x_2)
\\
&=\sum\limits_{\mu=1}^m\frac{1}{(r_{1,1}+r_{1,2})^m}
\prod\limits_{j\neq \mu,1\le j \le m}^m
\int_{x_1-r_{1,1}}^{x_1+r_{1,2}}|f_j(y_1,x_2)|dy_1
\int_{x_1-r_{1,1}}^{x_1+r_{1,2}}D_l|f_\mu(y_1,x_2)|dy_1
\endaligned  
\end{equation}
\item[{\rm (iii)}]  %%%%%%%%%%%% (iii) %%%%%%%%%%%%%%%%%%%%%%%%%%%%%%%%%
For all $\vec r\in\mathcal{B}_1(\vec{f})(x_1,x_2)\cup
\mathcal{B}_3(\vec{f})(x_1,x_2)$ 
with $r_{1,1}=r_{1,2}=0,\, r_{2,1}+r_{2,2}>0$, it holds
\begin{equation}\aligned\label{2.15}
{}&D_l {\mathscr{M}}_{\mathcal{R}}\vec{f}(x_1,x_2)
\\
&=\sum\limits_{\mu=1}^m\frac{1}{(r_{2,1}+r_{2,2})^m}
\prod\limits_{j\neq \mu,1\le j \le m}^m
\int_{x_2-r_{2,1}}^{x_2+r_{2,2}}|f_j(x_1,y_2)|dy_2
\int_{x_2-r_{2,1}}^{x_2+r_{2,2}}D_l|f_\mu(x_1,y_2)|dy_2
\;\endaligned
\end{equation}
\item[{\rm (iv)}]   %%%%%%%%%%%% (iv) %%%%%%%%%%%%%%%%%%%%%%%%%%%%%%%
If $ \vec{0}\in\mathcal{B}_i(\vec{f})(x_1,x_2)$ for $i=1,2,3$, then,
\begin{equation}\label{2.16}
D_l {\mathscr{M}}_{\mathcal{R}}\vec{f}(x_1,x_2)
=D_l|f|(x_1,x_2).
\end{equation}
\end{enumerate}}
\end{lemma}

\begin{proof}
We may assume without loss of generality that all
$f_i\geq0$, since $|f_i|\in W^{1,p_i}(\mathbb{R}^2)$
if $f_i\in W^{1,p_i}(\mathbb{R}^2)$. Fix $\Lambda>0$
and $l\in\{1,2\}$. Invoking Corollary 2.4, for any
$i\in\{1,2,3\}$, we can choose a sequence
$\{s_{i,k}\}_{k=1}^\infty$, $s_{i,k}>0$ and
$s_{i,k}\rightarrow0$ such that $\lim_{k\rightarrow
\infty}\pi(\mathcal{B}_i(\vec{f})(x),\mathcal{B}_i
(\vec{f})(x+s_{i,k}\vec{e}_l))=0$ for a.e.
$x\in R_{\vec{\Lambda}}(\vec{0})$.
%%%%%%%%%%%%%%%%%%%%%%%%%%%%%%%%%%%%%%%%%%
Then, Step 1 in the proof of Theorem 1.1 yields that
${\mathscr{M}}_{\mathcal{R}}(\vec{f})\in W^{1,p}
(\mathbb R^2)$  and
$$
\|({\mathscr{M}}_{\mathcal{R}}(\vec{f}))_{s_{i,k}}^l
-D_l {\mathscr{M}}_{\mathcal{R}}(\vec{f})
\|_{L^p(\mathbb{R}^2)}\rightarrow0\ \ {\rm as}\ k\rightarrow\infty.
$$
We also see that
$$\|(f_\mu)_{s_{i,k}}^l-D_lf_\mu\|_{L^{p}(\mathbb{R}^2)}\rightarrow0
\ \ {\rm as}\ k\rightarrow\infty,
\quad
\|(f_\mu)_{\tau(s_{i,k})}^l-f_\mu\|_{L^{p}(\mathbb{R}^2)}\rightarrow0
\ \ {\rm as}\ k\rightarrow\infty,$$
$$\|\mathcal{M}_{\mathcal{R}}((f_\mu)_{s_{i,k}}^l-D_lf_\mu)\|_{L^{p}(\mathbb{R}^2)}\rightarrow0
\ \ {\rm as}\ k\rightarrow\infty,
\quad
\|\mathcal{M}_{\mathcal{R}}((f_\mu)_{\tau(s_{i,k})}^l-f_\mu)\|_{L^{p}(\mathbb{R}^2)}\rightarrow0
\ \ {\rm as}\ k\rightarrow\infty,$$
$$\|\widetilde{\mathcal{M}}_j((f_\mu)_{s_{i,k}}^l-D_lf_\mu)\|_{L^{p}(\mathbb{R}^2)}
\rightarrow0\ \ {\rm as}\ k\rightarrow\infty \quad (j=1,2),$$
$$\|\widetilde{\mathcal{M}}_j((f_\mu)_{\tau(s_{i,k})}^l-f_\mu)\|_{L^{p}(\mathbb{R}^2)}
\rightarrow0\ \ {\rm as}\ k\rightarrow\infty \quad (j=1,2),$$
where $\widetilde{\mathcal{M}}_j$ is the one dimensional
uncentered Hardy-Littlewood maximal operator with respect
to the variable $x_j$ $(j=1,2)$. Furthermore, there exists
a subsequence $\{h_{i,k}\}_{k=1}^\infty$ of
$\{s_{i,k}\}_{k=1}^\infty$ and a measurable set
$A_{i,1}\subset R_{\vec{\Lambda}}(\vec{0})$ such that
$|R_{\vec{\Lambda}}(\vec{0})\backslash A_{i,1}|=0$ and
\begin{enumerate}
\item[{(i)}] $(f_\mu)_{h_{i,k}}^l(x)\rightarrow D_lf_\mu(x)$,
$(f_\mu)_{\tau(h_{i,k})}^l(x)\rightarrow f_\mu(x)$,
$\mathcal{M}_{\mathcal{R}}((f_\mu)_{h_{i,k}}^l-D_lf_\mu)
\rightarrow0$, $\mathcal{M}_{\mathcal{R}}((f_\mu)_{\tau
(h_{i,k})}^l-f_\mu)\rightarrow0$,
$\widetilde{\mathcal{M}}_j((f_\mu)_{h_{i,k}}^l-D_lf_\mu)
(x)\rightarrow0$ $(j=1,2)$, $\widetilde{\mathcal{M}}_j
((f_\mu)_{\tau(h_{i,k})}^l-f_\mu)\rightarrow0$ $(j=1,2)$
and $({\mathscr{M}}_{\mathcal{R}}(\vec{f}))_{h_{i,k}}^l(x)
\rightarrow D_l {\mathscr{M}}_{\mathcal{R}}(\vec{f})(x)$
when $k\rightarrow\infty$ for any $x\in A_{i,1}$;

\item[{(ii)}] $\lim_{k\rightarrow\infty}\pi(\mathcal{B}_i
(\vec{f}))(x),\mathcal{B}_i(\vec{f})(x+h_{i,k}\vec{e}_l))=0$
for any $x\in A_{i,1}$.
\end{enumerate}
Let
\begin{enumerate}
\item[{}] $A_{i,2}:=\bigcap\limits_{k=1}^\infty\{x\in\mathbb{R}^2:
 {\mathscr{M}}_{\mathcal{R}}(\vec{f})(x+h_{i,k}\vec{e}_l)\geq
 u_{x+h_{i,k}\vec{e}_l,\vec{f}}(0,0,0,0)\},
$
\item[{}] $
A_{i,3}:=\{x\in\mathbb{R}^2: {\mathscr{M}}_{\mathcal{R}}(\vec{f})(x)=u_{x,\vec{f}}(0,0,0,0)
\ {\rm if}\ (0,0,0,0)\in\mathcal{B}_i(\vec{f}(x)\},
$
\item[{}] $
A_{i,4}:=\bigcap\limits_{k=1}^\infty\{x\in\mathbb{R}^2:
 {\mathscr{M}}_{\mathcal{R}}(\vec{f})(x+h_{i,k}\vec{e}_l)=u_{x+h_{i,k}\vec{e}_l,\vec{f}}(0,0,0,0)$\\

\qquad {\rm if}\ $(0,0,0,0)\in\mathcal{B}_i(\vec{f})(x+h_{i,k}\vec{e}_l)\},
$
\item[{}]
$A_{i,5}:=\bigcap\limits_{k=1}^\infty\{x\in\mathbb{R}^2:
 {\mathscr{M}}_{\mathcal{R}}(\vec{f})(x+h_{i,k}\vec{e}_l)=u_{x+h_{i,k}\vec{e}_l,\vec{f}}(r_{1,1},r_{1,2},0,0)$\\

 \qquad{\rm if}\ $(r_{1,1},r_{1,2},0,0)\in\mathcal{B}_1(\vec{f})(x+h_{i,k}\vec{e}_l)\ \ {\rm and}\ r_{1,1}+r_{1,2}>0\},$
\item[{}] $\aligned
A_{i,6}:&=\{x\in\mathbb{R}^2: {\mathscr{M}}_{\mathcal{R}}(\vec{f})(x)
=u_{x,\vec{f}}(r_{1,1},r_{1,2},0,0)
\ {\rm if}\ (r_{1,1},r_{1,2},0,0)\in\mathcal{B}_1(\vec{f})(x)
\\&\quad \ {\rm and}\ r_{1,1}+r_{1,2}>0\},\endaligned
$
\item[{}]$
A_{i,7}:=\bigcap\limits_{k=1}^\infty\{x\in\mathbb{R}^2:
 {\mathscr{M}}_{\mathcal{R}}(\vec{f})(x+h_{i,k}\vec{e}_l)
=u_{x+h_{i,k}\vec{e}_l,\vec{f}}(0,0,r_{2,1},r_{2,2})$\\

\qquad{\rm if}\ $(0,0,r_{2,1},r_{2,2})\in\mathcal{B}_1(\vec{f})(x+h_{i,k}\vec{e}_l)\ {\rm and}\ r_{2,1}+r_{2,2}>0\},$
\item[{}] $\aligned
A_{i,8}:&=\{x\in\mathbb{R}^2: {\mathscr{M}}_{\mathcal{R}}(\vec{f})(x)
=u_{x,\vec{f}}(0,0,r_{2,1},r_{2,2})\ {\rm if}\ (0,0,r_{2,1},r_{2,2})\in\mathcal{B}_1(\vec{f})(x)
\\&\quad\ \ {\rm and}\ r_{2,1}+r_{2,2}>0\},\endaligned
$
\item[{}] $A_{i,9}:=\bigcap\limits_{k=1}^\infty\{x\in\mathbb{R}^2:
 {\mathscr{M}}_{\mathcal{R}}(\vec{f})(x+h_{i,k}\vec{e}_l)
=u_{x+h_{i,k}\vec{e}_l,\vec{f}}(\vec r) \ {\rm if}\ \vec r\in\mathcal{B}_2(\vec{f})(x+h_{i,k}\vec{e}_l)\};$
\item[{}] $
A_{i,10}:=\{x\in\mathbb{R}^2: {\mathscr{M}}_{\mathcal{R}}(\vec{f})(x)
=u_{x,\vec{f}}(\vec r)\ \ {\rm if}\ \vec r\in\mathcal{B}_2(\vec{f})(x)\};
$
\item[{}] $
A_{i,11}:=\bigcap\limits_{k=1}^\infty\{x\in\mathbb{R}^2:
 {\mathscr{M}}_{\mathcal{R}}(\vec{f})(x+h_{i,k}\vec{e}_l)=u_{x+h_{i,k}\vec{e}_l,\vec{f}}(\vec r)
\ \ {\rm if}\ \vec r\in\mathcal{B}_3(\vec{f})(x+h_{i,k}\vec{e}_l)\};
$
\item[{}] $
A_{i,12}:=\{x\in\mathbb{R}^2: {\mathscr{M}}_{\mathcal{R}}(\vec{f})(x)
=u_{x,\vec{f}}(\vec r)\ \ {\rm if}\ \vec r\in\mathcal{B}_3(f)(x)\}.$
\end{enumerate}
Let $A_{i}=\bigcap_{k=1}^{12}A_{i,k}$ and $A=A_1\cup A_2
\cup A_3$. Note that $|R_{\vec{\Lambda}}(\vec{0})\setminus A|=0$.
Let $x=(x_1,x_2)\in A$ be a Lebesgue point of all $f_\mu$ and
$D_lf_\mu$ and $\vec{r}=(r_{1,1},r_{1,2},r_{2,1},r_{2,2})\in
\mathcal{B}_i(\vec{f})(x)$. There exists $\vec{r}_{k}^i=(r_{1,1,
i,k},r_{1,2,i,k},r_{2,1,i,k},r_{2,2,i,k})\in\mathcal{B}_i(\vec{f})
(x+h_{i,k}\vec{e}_l)$ such that $\lim_{k\rightarrow\infty}
(r_{1,1,i,k},r_{1,2,i,k},r_{2,1,i,k},r_{2,2,i,k})=(r_{1,1},
r_{1,2}, r_{2,1},r_{2,2})$. Furthermore, we assume that $x_1$
is a Lebesgue point of $D_lf_\mu(\cdot,x_2)$ for $i=2$, $x_2$
is a Lebesgue point of $D_lf_\mu(x_1,\cdot)$ for $i=3$, and
$\|D_lf_\mu(x_1,\cdot)\|_{L^{p_\mu}(\mathbb R)},\,\|D_lf_\mu
(\cdot,x_2)\|_{L^{p_\mu}(\mathbb R)}<\infty$.

%\noindent
%%%%%%%%%%%%  Case A %%%%%%%%%%%%%%%%%%%%%%%%%%%%%%%%%%%%%%%%%%%%%
{\bf Case A} ($r_{1,1}+r_{1,2}>0$ and $r_{2,1}+r_{2,2}>0$). In
this case $\vec r\in \mathcal B_1(\vec f)(x)$ and this happens
when $x\in A_1$. Without loss of generality we may assume that
all $r_{1,1,1,k}>0$, $r_{1,2,1,k}>0$, $r_{2,1,1,k}>0$ and
$r_{2,2,1,k}>0$. Denote  $[x_1-r_{1,1,1,k}, x_1+r_{1,2,1,k}]
\times [x_2-r_{2,1,1,k}, x_2+r_{2,2,1,k}]$ by $R_k$ and
$dy_1dy_2=d\vec{y}$. Then, noting $\vec r_k\in \mathcal B_1
(\vec f)(x+h_{1,k}\vec{e}_l)$ and using Lemma \ref{l2.2} (ix),
we have
\begin{equation}\label{2.17}
\begin{array}{ll}
{}&D_l {\mathscr{M}}_{\mathcal{R}}(\vec{f})(x)
=\displaystyle\lim\limits_{k\rightarrow\infty}
\frac{1}{h_{1,k}}( {\mathscr{M}}_{\mathcal{R}}(\vec{f})(x+h_{1,k}\vec{e}_l)
-{\mathscr{M}}_{\mathcal{R}}(\vec{f})(x))
\\
&\leq\displaystyle\lim\limits_{k\rightarrow\infty}
\frac{1}{h_{1,k}}(u_{x+h_{1,k}\vec{e}_l,\vec{f}}(\vec{r}_k^1)-u_{x,\vec{f}}(\vec{r}_k^1))\\
&=\displaystyle\lim\limits_{k\rightarrow\infty}\sum\limits_{\mu=1}^m\frac{1}{(r_{1,1,1,k}+r_{1,2,1,k})^m(r_{2,1,1,k}+r_{2,2,1,k})^m}\iint_{R_k}(f_\mu)_{h_{1,k}}^l(y_1,y_2)
dy_1dy_2\\
&\quad\times\displaystyle\prod\limits_{\nu=1}^{\mu-1}\iint_{R_k}(f_{\nu})_{\tau(h_{1,k})}^l(y_1,y_2)
dy_1dy_2\prod\limits_{w=\mu+1}^m\iint_{R_k}f_w(y_1,y_2)dy_1dy_2\\
&=\displaystyle\sum\limits_{\mu=1}^m\frac{1}{(r_{1,1}+r_{1,2})^m(r_{2,1}+r_{2,2})^m}
\int_{x_1-r_{1,1}}^{x_1+r_{1,2}}\int_{x_2-r_{2,1}}^{x_2+r_{2,2}}
D_lf_\mu(y_1,y_2)d\vec{y}
\\
&\quad\times\displaystyle\prod\limits_{\nu=1}^{\mu-1}
\int_{x_1-r_{1,1}}^{x_1+r_{1,2}}\int_{x_2-r_{2,1}}^{x_2+r_{2,2}}
f_{\nu}(y_1,y_2)d\vec{y}\prod\limits_{w=\mu+1}^m\int_{x_1-r_{1,1}}^{x_1+r_{1,2}}
\int_{x_2-r_{2,1}}^{x_2+r_{2,2}}f_w(y_1,y_2)d\vec{y}.
\end{array}\end{equation}
Here, we used the fact that $\lim_{k\rightarrow\infty}\vec{r}_k^1
=\vec{r}$ and $(f_\mu)_{\tau(h_{1,k})}^{l}\chi_{R_k}\rightarrow
f_\mu\chi_{[x_1-r_{1,1},x_1+r_{1,2}]\times[x_2-r_{2,1},x_2+
r_{2,2}]}$ and $(f_\mu)_{h_{1,k}}^{l}\chi_{R_k}\rightarrow
D_lf_\mu\chi_{[x_1-r_{1,1},x_1+r_{1,2}]\times[x_2-r_{2,1},
x_2+r_{2,2}]}$ in $L^1(\mathbb{R}^2)$ as $k\rightarrow\infty$.
Then

\begin{equation}\label{2.18}
\aligned
D_l {\mathscr{M}}_{\mathcal{R}}\vec{f}(x)
&
\leq\sum\limits_{\mu=1}^m\frac{1}{(r_{1,1}+r_{1,2})^m(r_{2,1}+r_{2,2})^m}
\int_{x_1-r_{1,1}}^{x_1+r_{1,2}}\int_{x_2-r_{2,1}}^{x_2+r_{2,2}}
D_lf_\mu(y_1,y_2)dy_1dy_2
\\&\quad\times
\prod\limits_{\nu\neq\mu,1\leq\nu\leq m}\int_{x_1-r_{1,1}}^{x_1+r_{1,2}}
\int_{x_2-r_{2,1}}^{x_2+r_{2,2}}f_{\nu}((y_1,y_2))dy_1dy_2.
\endaligned
\end{equation}
%%%%%%%%%%%%%%%%%%%%%%%%%%%%%%%%%%%%%%%%
On the other hand, using Lemma \ref{l2.2} (ix), we have
\begin{equation}\label{2.19}
\begin{array}{ll}
{}&D_l {\mathscr{M}}_{\mathcal{R}}(\vec{f})(x)
=\displaystyle\lim\limits_{k\rightarrow\infty}\frac{1}{h_{1,k}}
( {\mathscr{M}}_{\mathcal{R}}\vec{f}(x+h_{1,k}\vec{e}_l)
-{\mathscr{M}}_{\mathcal{R}}(\vec{f})(x))
\\
&\geq\displaystyle\lim\limits_{k\rightarrow\infty}
\frac{1}{h_{1,k}}(u_{x+h_{1,k}\vec{e}_l,\vec f}(\vec{r})-u_{x,\vec f}(\vec{r}))
\\
&=\displaystyle
\sum\limits_{\mu=1}^m\frac{1}{(r_{1,1}+r_{1,2})^m(r_{2,1}+r_{2,2})^m}
\int_{x_1-r_{1,1}}^{x_1+r_{1,2}}\int_{x_2-r_{2,1}}^{x_2+r_{2,2}}
D_lf_\mu(y_1,y_2)dy_1dy_2
\\
&\displaystyle\hspace{1cm}\times
\prod\limits_{1\le \nu\neq\mu\le m}\int_{x_1-r_{1,1}}^{x_1+r_{1,2}}
\int_{x_2-r_{2,1}}^{x_2+r_{2,2}}f_{\nu}((y_1,y_2))dy_1dy_2.
\end{array}\end{equation}
Combining \eqref{2.19} with \eqref{2.18} yields (\ref{2.13})
for a.e. $x\in R_{\vec{\Lambda}}(\vec{0})\cap A_1$.
%%%%%%%%%%%% end Case A %%%%%%%%%%%%%%%%%%%%%%%%%%%%%%%%%%%%%%%%%%%%%

%%%%%%%%%%%%%%%%%% Case B %%%%%%%%%%%%%%%%%%%%%%%%%%%%%%%%%%%%%
{\bf Case B} ($r_{1,1}+r_{1,2}>0$ and $r_{2,1}=r_{2,2}=0$).
We consider the following two cases.

(i) $(r_{1,1}, r_{1,2},0,0)\in \mathcal B_2(\vec{f})(x)$. This
happens in the case $x\in A_2$. Without loss of generality we
may assume that all $r_{1,1,2,k}, r_{1,2,2,k}>0$. We notice
that $r_{2,1,2,k}=r_{2,2,2,k}=0$ for all $k\geq1$. Then, noting
$\vec r_k\in \mathcal B_2(\vec f)(x+h_{2,k}\vec{e}_l)$ and using
Lemma \ref{l2.2} (vii), we have
\begin{equation}\label{2.20}
\begin{array}{ll}
{}&D_l {\mathscr{M}}_{\mathcal{R}}(\vec{f})(x)
=\displaystyle\lim\limits_{k\rightarrow\infty}
\frac{1}{h_{2,k}}( {\mathscr{M}}_{\mathcal{R}}(\vec{f})(x+h_{2,k}\vec{e}_l)
-{\mathscr{M}}_{\mathcal{R}}(\vec{f})(x))
\\
&\leq\displaystyle\lim\limits_{k\rightarrow\infty}
\frac{1}{h_{2,k}}(u_{x+h_{2,k}\vec{e}_l,\vec{f}}(r_{1,1,2,k},r_{1,2,2,k},0,0)
-u_{x,\vec{f}}(r_{1,1,2,k},r_{1,2,2,k},0,0))
\\
&\leq\displaystyle\lim\limits_{k\rightarrow\infty}\frac{1}{h_{2,k}}
\sum\limits_{\mu=1}^m
\frac{1}{(r_{1,1,2,k}+r_{1,2,2,k})^m}\int_{x_1-r_{1,1,2,k}}^{x_1+r_{1,2,2,k}}
(f_\mu)_{h_{2,k}}^{l}(y_1,x_2)dy_1
%\end{array}\end{equation}
%\begin{equation*}
%\begin{array}{ll}
%{}
\\
&
\quad \times\displaystyle\prod\limits_{\nu=1}^{\mu-1}
\int_{x_1-r_{1,1,2,k}}^{x_1+r_{1,2,2,k}}
(f_{\nu})_{\tau({h_{2,k}})}^l(y_1,x_2)
dy_1\prod\limits_{w=\mu+1}^m\int_{x_1-r_{1,1,2,k}}^{x_1+r_{1,2,2,k}}
f_{w}(y_1,x_2)
dy_1
\\
&\leq\displaystyle\sum\limits_{\mu=1}^m
\frac{1}{(r_{1,1}+r_{1,2})^m}\int_{x_1-r_{1,1}}^{x_1+r_{1,2}}
D_l f_\mu(y_1,x_2)dy_1
\prod\limits_{\nu\neq\mu,1\leq\nu\leq m}\int_{x_1-r_{1,1}}^{x_1+r_{1,2}}
f_{\nu}(y_1,x_2)dy_1.
\end{array}
\end{equation}
Here we used the fact that $\lim_{k\rightarrow\infty}r_{1,1,2,k}
=r_{1,1}$, $\lim_{k\rightarrow\infty}r_{1,2,2,k}=r_{1,2}$ and
$$
(f_\mu)_{h_{2,k}}^{l}(\cdot,x_2)\chi_{[x_1-r_{1,1,2,k},x_1+r_{1,2,2,k}]}
\rightarrow D_lf_\mu(\cdot,x_2)\chi_{[x_1-r_{1,1},x_1+r_{1,2}]}
\ {\rm in}\ L^1(\mathbb{R})\ {\rm as}\ k\rightarrow\infty.
$$
%%%%%%%%%%%%%%%%%%%%%%%%%%%%%
Further more, using Lemma \ref{l2.2} (vii), we have
\begin{equation}\label{2.21}
\begin{array}{ll}
&D_l {\mathscr{M}}_{\mathcal{R}}(\vec f)(x)\\
&\geq\displaystyle\lim\limits_{k\rightarrow\infty}
\frac{1}{h_{2,k}}(u_{x+h_{2,k}e_l,\vec{f}}(r_{1,1},r_{1,2},0,0)
-u_{x,\vec{f}}(r_{1,1},r_{1,2},0,0))
\\
&\geq\displaystyle\lim\limits_{k\rightarrow\infty}\sum\limits_{\mu=1}^m
\frac{1}{(r_{1,1,2,k}+r_{1,2,2,k})^m}\int_{x_1-r_{1,1,2,k}}^{x_1+r_{1,2,2,k}}
(f_\mu)_{h_{2,k}}^{l}(y_1,x_2)dy_1
\\&\hspace{0.5cm}\times\displaystyle
\prod\limits_{\nu=1}^{\mu-1}\int_{x_1-r_{1,1,2,k}}^{x_1+r_{1,2,2,k}}
(f_{\nu})_{\tau({h_{2,k}})}^l((y_1,x_2))
dy_1\prod\limits_{w=\mu+1}^m\int_{x_1-r_{1,1,2,k}}^{x_1+r_{1,2,2,k}}
f_{w}((y_1,x_2))dy_1
\end{array}\end{equation}
\begin{equation*}
\begin{array}{ll}
&\geq\displaystyle\sum\limits_{\mu=1}^m
\frac{1}{(r_{1,1}+r_{1,2})^m}\int_{x_1-r_{1,1}}^{x_1+r_{1,2}}
D_l f_\mu(y_1,x_2)dy_1
\prod\limits_{\nu\neq\mu,1\leq\nu\leq m}\int_{x_1-r_{1,1}}^{x_1+r_{1,2}}
f_{\nu}(y_1,x_2)dy_1.
\end{array}\end{equation*}
\eqref{2.21} together with \eqref{2.20} yields \eqref{2.14} 
for a.e. $x\in R_{\vec{\Lambda}}(\vec{0})\cap A_2$.

%%%%%%%%%%%%%%%%%% (ii) %%%%%%%%%%%%%%%%%%%%%%%%%%%%%%%%%%
(ii) $(r_{1,1},r_{1,2},0,0)\in \mathcal B_1(\vec{f})(x)$. 
This happens in the case $x\in A_1$. Assume that
$r_{1,1,1,k},r_{1,2,1,k}>0$. As in the case A, noting 
$x\in A_1\subset A_{1,5}$, we have
\begin{equation}\label{2.22}
\aligned
D_l {\mathscr{M}}_{\mathcal{R}}(\vec{f})(x)
&\leq \displaystyle\lim\limits_{k\rightarrow\infty}
\sum_{\mu=1}^m
\frac{1}{(r_{1,1,1,k}+r_{1,2,1,k})^m}
\int_{x-r_{1,1,1,k}}^{x+r_{1,2,1,k}}
(f_{\mu})_{h_{1,k}}^l(y_1,x_2)dy_1
\\
&\quad \times
\prod_{\nu=1}^{\mu-1}\int_{x-r_{1,1,1,k}}^{x+r_{1,2,1,k}}
(f_{\nu})_{\tau(h_{1,k)}}^l(y_1,x_2)dy_1
\prod_{w=\mu+1}^{m}\int_{x-r_{1,1,1,k}}^{x+r_{1,2,1,k}}f_{w}(y_1,x_2)dy_1.
\endaligned
\end{equation}
We claim that the limits of the right side will tend to
$$
\aligned
\sum_{\mu=1}^m\frac{1}{(r_{1,1}+r_{1,2})^m}
\int_{x_1-r_{1,1}}^{x_1+r_{1,2}}D_lf_{\mu}(y_1,x_2)dy_1
\prod_{\nu\neq\mu,1\leq\nu\leq m}\int_{x_1-r_{1,1}}^{x_1+r_{1,2}}
f_{\nu}(y_1,x_2)dy_1.
\endaligned
$$
To see this, we only consider the limit of the following parts,
since the same reasoning applies to the other terms.
$$
\frac{1}{r_{1,1,1,k}+r_{1,2,1,k}}
\int_{x-r_{1,1,1,k}}^{x+r_{1,2,1,k}}(f_{\mu})_{h_{1,k}}^l(y_1,x_2)dy_1.
$$
Now, we know from the property (i) for $x\in A_1$ that
\begin{equation}\label{2.23}
\begin{array}{ll}\aligned
{}&\displaystyle\lim\limits_{k\rightarrow\infty}\biggl|\displaystyle
\frac{1}{r_{1,1,1,k}+r_{1,2,1,k}}
\int_{x-r_{1,1,1,k}}^{x+r_{1,2,1,k}}
(f_\mu)_{h_{1,k}}^l(y_1,x_2)-D_l f_\mu(y_1,x_2))dy_1\biggr|
\\
&\le \lim\limits_{k\rightarrow\infty}\widetilde{\mathcal{M}}_1((f_\mu)_{h_{1,k}}^l-D_lf_\mu)(x_1,x_2)\\
&\quad+\displaystyle\lim\limits_{k\rightarrow\infty}\biggl|\displaystyle
\frac{1}{r_{1,1,1,k}+r_{1,2,1,k}}
\int_{x-r_{1,1,1,k}}^{x+r_{1,2,1,k}}D_lf_\mu(y_1,x_2)-D_l f_\mu(y_1,x_2))dy_1\biggr|\\
&=0.
\endaligned\end{array}
\end{equation}
We see moreover that
\begin{equation}\label{2.24}
\begin{array}{ll}
{}&\displaystyle\lim\limits_{k\rightarrow\infty}\biggl|\displaystyle
\Bigl(\dfrac{1}{r_{1,1,1,k}+r_{1,2,1,k}}
-\dfrac{1}{r_{1,1}+r_{1,2}}\Bigr)
%\\&\hspace{5.5cm}\times
\displaystyle\int_{x_1-r_{1,1,1,k}}^{x_1+r_{1,2,1,k}}
D_lf_\mu(y_1,x_2)dy_1\biggr|
\\
&\le \displaystyle\lim\limits_{k\rightarrow\infty}
(r_{1,1,1,k}+r_{1,2,1,k})\Bigl|\dfrac{1}{r_{1,1,1,k}+r_{1,2,1,k}}
-\dfrac{1}{r_{1,1}+r_{1,2}}\Bigr|\widetilde{\mathcal{M}}_1(D_l f_\mu)(x_1,x_2)
=0.
\end{array}
\end{equation}
Noting that $\|D_\ell f_\mu)(\cdot,x_2)\|_{L^p(\mathbb R))}<\infty$, we get
\begin{equation}\label{2.25}
\begin{array}{ll}
{}&\displaystyle\lim\limits_{k\rightarrow\infty}\biggl|%\displaystyle
\dfrac{1}{r_{1,1}+r_{1,2}}
\int_{x_1-r_{1,1,1,k}}^{x_1+r_{1,2,1,k}}D_lf_\mu(y_1,x_2)dy_1 -
\dfrac{1}{r_{1,1}+r_{1,2}}\int_{x_1-r_{1,1}}^{x_1+r_{1,2}}
D_lf_\mu(y_1,x_2)dy_1\biggr|
\\
&\le
\displaystyle\lim\limits_{k\rightarrow\infty}
\dfrac{(|r_{1,1,1,k}-r_{1,1}|+|r_{1,2,1,k}-r_{1,2}|)^{1/p'}}{r_{1,1}+r_{1,2}}
\\
&\quad\times\biggl(\displaystyle
\int_{x_1-\max\{r_{1,1},r_{1,1,1,k}\}}^{x_1-\min\{r_{1,1},r_{1,1,1,k}\}}
+ \int_{x_1+\min\{r_{1,2},r_{1,2,1,k}\}}^{x_1+\max\{r_{1,2},r_{1,2,1,k}\}}|D_lf_\mu(y_1,x_2))|^pdy_1
\biggr)^{1/p}
\\
&\le C\displaystyle\lim\limits_{k\rightarrow\infty}
\dfrac{(|r_{1,1,1,k}-r_{1,1}|+|r_{1,2,1,k}-r_{1,2}|)^{1/p'}}{r_{1,1}+r_{1,2}}
\|D_lf_\mu(\cdot,x_2)\|_{L^p(\mathbb R))}=0.
\end{array}
\end{equation}
From \eqref{2.22} to \eqref{2.25} , it follows that

\begin{equation}\label{2.26}\aligned
&\displaystyle\lim\limits_{k\rightarrow\infty}
\frac{1}{r_{1,1,1,k}+r_{1,2,1,k}}
\int_{x_1-r_{1,1,1,k}}^{x_1+r_{1,2,1,k}}
(f_\mu)_{h_{1,k}}^l(y_1,x_2)dy_1
= \displaystyle
\frac{1}{r_{1,1}+r_{1,2}}\int_{x_1-r_{1,1}}^{x_1+r_{1,2}}
D_lf_\mu(y_1,x_2)dy_1,
\endaligned
\end{equation}
and hence we verified the claim.

On the other hand, noting $x\in A_1\subset A_{1,6}$, by 
the same reasoning as in the case A, we get
\begin{equation}\label{2.27}
\aligned
D_l {\mathscr{M}}_{\mathcal{R}}(\vec{f})(x)
&\ge\displaystyle\sum_{\mu=1}^m\frac{1}{(r_{1,1}+r_{1,2})^m}
\int_{x_1-r_{1,1,}}^{x_1+r_{1,2}}
D_lf_{\mu}(y_1,x_2)dy_1\prod_{\nu=1}^{\mu-1}\int_{x_1-r_{1,1}}^{x_1+r_{1,2}}
f_{\nu}(y_1,x_2)dy_1
&\\
&\quad\times\prod_{w=\mu+1}^{m}\int_{x_1-r_{1,1}}^{x_1+r_{1,2}}
f_{w}(y_1,x_2)dy_1
\endaligned
\end{equation}
The above claim and \eqref{2.27} yield \eqref{2.14} for a.e.
$x\in R_{\vec{\Lambda}}(\vec{0})\cap A_1$.

%%%%%%%%%%% Case C %%%%%%%%%%%%%%%%%%%%%%%%%%%%%%%%%%%%%%%%%%%%%%%%%
{\bf Case C} ($r_{1,1}=r_{1,2}=0$ and $r_{2,1}+r_{2,2}>0$).
Similar argument as in Case B gives \eqref{2.15} for a.e.
$x\in R_{\vec{\Lambda}}(\vec{0})\cap (A_1\cup A_2)$.

%%%%%%%%%%%%%%% Case D %%%%%%%%%%%%%%%%%%%%%%%%%%%%%%%%%%%%%%%%%%%%%
{\bf Case D} ($\vec r=(0,0,0,0)$). We consider the following
three cases:

(i) Assume that $(0,0,0,0)\in\mathcal{B}_2(\vec{f})(x)$. Then
$x\in A_2$. The lower bound of $D_l {\mathscr{M}}_{\mathcal{R}}
(\vec{f})(x)$ follows from
\begin{equation}\label{2.28}
\aligned
&D_l {\mathscr{M}}_{\mathcal{R}}(\vec{f})(x)\\
&=\lim\limits_{k\rightarrow\infty}\frac{1}{h_{2,k}}
({\mathscr{M}}_{\mathcal{R}}(\vec{f})(x+h_{2,k}\vec{e}_l)
- {\mathscr{M}}_{\mathcal{R}}(\vec{f})(x))
\\&\geq\displaystyle\lim\limits_{k\rightarrow\infty}\frac{1}{h_{2,k}}\Big(\prod\limits_{\mu=1}^mf_\mu(x+h_{2,k}\vec{e}_l)
-\prod\limits_{\mu=1}^mf_\mu(x)\Big)\\
&\geq\displaystyle\sum\limits_{\mu=1}^m\lim\limits_{k\rightarrow\infty}\frac{1}{h_{2,k}}(f_\mu(x+h_{2,k}\vec{e}_l)-f_\mu(x))
\Big(\prod\limits_{\nu=1}^{\mu-1}f_\nu(x)\Big)\Big(\prod\limits_{j=\mu+1}^mf_j(x+h_{2,k}\vec{e}_l)\Big)\\
&\geq\displaystyle\sum\limits_{\mu=1}^mD_lf_\mu(x)\Big(\prod\limits_{i\neq\mu,1\leq i\leq m}f_i(x)\Big).
\endaligned\end{equation}
To get the upper bound of $D_l\mathscr{M}_{\mathcal{R}}
(\vec{f})(x)$, note that $\lim_{k\rightarrow\infty}r_{1,1,
2,k}=0$, $\lim_{k\rightarrow\infty}r_{1,2,2,k}=0$ and
$r_{2,1,2,k}=r_{2,2,2,k}=0$ for all $k\geq1$. If $r_{1,1,
k}+r_{1,2,k}=0$ for infinitely many $k$, then by Lemma 
\ref{l2.2} (iv). one obtains that 
\begin{equation}\label{2.29}
\aligned
D_l {\mathscr{M}}_{\mathcal{R}}(\vec{f})(x)
&=\lim\limits_{k\rightarrow\infty}\frac{1}{h_{2,k}}
({\mathscr{M}}_{\mathcal{R}}(\vec{f})(x+h_{2,k}\vec{e}_l)
-{\mathscr{M}}_{\mathcal{R}}(\vec{f})(x))
\\&\leq\displaystyle\lim\limits_{k\rightarrow\infty}\frac{1}{h_{2,k}}\Big(\prod\limits_{\mu=1}^mf_\mu(x+h_{2,k}\vec{e}_l)
-\prod\limits_{\mu=1}^mf_\mu(x)\Big)\\
&\leq\displaystyle\sum\limits_{\mu=1}^mD_lf_\mu(x)\Big(\prod\limits_{\nu\neq\mu,1\leq\nu\leq m}f_\nu(x)\Big)
.\endaligned\end{equation}
If there exists $k_0\in\mathbb{N}$ such that $r_{1,1,2,k}
+r_{1,2,2,k}>0$ when $k\geq k_0$. Then \eqref{2.20} gives 
that
\begin{equation*}
\begin{array}{ll}
D_l {\mathscr{M}}_{\mathcal{R}}(\vec{f})(x)\leq\displaystyle\sum\limits_{\mu=1}^m\lim\limits_{k\rightarrow\infty}\frac{1}{r_{1,1,2,k}+r_{1,2,2,k}}\int_{x_1-r_{1,1,2,k}}^{x_1+r_{1,2,2,k}}(f_\mu)_{h_{2,k}}^l(y_1,x_2)dy_1\displaystyle
\Big(\prod\limits_{\nu=1}^{\mu-1}\frac{1}{r_{1,1,2,k}+r_{1,2,2,k}}\end{array}\end{equation*}\begin{equation}\label{2.30}\begin{array}{ll}
{}&\quad\times\displaystyle\int_{x_1-r_{1,2,2,k}}^{x_1+r_{1,2,2,k}}f_\nu(y_1,x_2)dy_1\Big)
\Big(\prod\limits_{j=\nu+1}^{m}\frac{1}{r_{1,1,2,k}+r_{1,2,2,k}}\int_{x_1-r_{1,1,2,k}}^{x_1+r_{1,2,2,k}}(f_j)_{\tau(h_{2,k})}^l(y_1,x_2)dy_1\Big).
\end{array}\end{equation}
Since $x_1$ is a Lebesgue point for $D_lf_\mu(\cdot,x_2)$,
we have
\begin{equation}\label{2.31}
\begin{array}{ll}
{}&\displaystyle\lim\limits_{k\rightarrow\infty}\Big|\frac{1}{r_{1,1,2,k}+r_{1,2,2,k}}\int_{x_1-r_{1,1,2,k}}^{x_1+r_{1,2,2,k}}(f_\mu)_{h_{2,k}}^l(y_1,x_2)dy_1-D_lf_\mu(x_1,x_2)\Big|\\
&\leq\displaystyle\lim\limits_{k\rightarrow\infty}\frac{1}{r_{1,1,2,k}+r_{1,2,2,k}}\int_{x_1-r_{1,1,2,k}}^{x_1+r_{1,2,2,k}}|(f_\mu)_{h_{2,k}}^l(y_1,x_2)-D_lf_\mu(y_1,x_2))|dy_1\\
&\leq\displaystyle\lim\limits_{k\rightarrow\infty}\widetilde{\mathcal{M}}_1((f_\mu)_{h_{2,k}}^l-D_lf_\mu)(x)\\
&\quad+\displaystyle\lim\limits_{k\rightarrow\infty}\frac{1}{r_{1,1,2,k}+r_{1,2,2,k}}\int_{x_1-r_{1,1,2,k}}^{x_1+r_{1,2,2,k}}|D_lf_\mu(y_1,x_2)-D_lf_\mu(y_1,x_2))|dy_1=0.
\end{array}\end{equation}
Similarly, it holds that $\lim_{k\rightarrow\infty}
\frac{1}{r_{1,1,2,k}+r_{1,2,2,k}}\int_{x_1-r_{1,1,2,k}}^{x_1
+r_{1,2,2,k}}(f_\mu)_{\tau(h_{2,k})}^l(y_1,x_2)dy_1=f_\mu
(x_1,x_2)$.

We get from \eqref{2.30} and \eqref{2.31} that
\begin{equation}\label{2.32}
\begin{array}{ll}
D_l {\mathscr{M}}_{\mathcal{R}}(\vec{f})(x)\leq \sum\limits_{\mu=1}^m D_lf_\mu(x_1,x_2)\Big(\prod\limits_{1\leq \nu\neq\mu\leq m}f_{\nu}(x_1,x_2)\Big)\end{array}\end{equation}
\eqref{2.32} together with \eqref{2.28}-\eqref{2.29} yields 
\eqref{2.16} in the case $\vec{0}\in\mathcal{B}_2(\vec{f})
(x)$ for a.e. $x\in R_{\vec{\Lambda}}(\vec{0})$.

(ii) Assume that $(0,0,0,0)\in\mathcal{B}_3(\vec{f})(x)$. We 
can get \eqref{2.16} for almost $x\in R_{\vec{\Lambda}}
(\vec{0})$ similarly.

(iii) Assume that $(0,0,0,0)\in\mathcal{B}_1(\vec{f})(x)$. In the
case $x\in A_1$. Note that
\begin{equation}\label{2.33}
\begin{array}{ll}
D_l {\mathscr{M}}_{\mathcal{R}}(\vec{f})(x)&=\displaystyle\lim\limits_{k\rightarrow\infty}\frac{1}{h_{1,k}}( {\mathscr{M}}_{\mathcal{R}}(\vec{f})(x+h_{1,k}\vec{e}_l)
-{\mathscr{M}}_{\mathcal{R}}(\vec{f})(x))\\
&\geq\displaystyle\sum\limits_{\mu=1}^mD_lf_\mu(x)\Big(\prod\limits_{\nu\neq\mu,1\leq\nu\leq m}f_{\nu}(x)\Big).
\end{array}\end{equation}
Below we estimate the upper bound of
$D_l{\mathscr{M}}_{\mathcal{R}}(\vec{f})(x)$. We consider
the following four cases:

{\rm (a)} If $(r_{1,1,k},r_{1,2,k},r_{2,1,k},r_{2,2,k})=
(0,0,0,0)$ for infinitely many $k$, then
$$\aligned D_l {\mathscr{M}}_{\mathcal{R}}(\vec{f})(x)&=\lim\limits_{k\rightarrow\infty}\frac{1}{h_{1,k}}({\mathscr{M}}_{\mathcal{R}}(\vec{f})(x+h_{1,k}\vec{e}_l)
-{\mathscr{M}}_{\mathcal{R}}(\vec{f})(x))\\
&\leq\displaystyle\sum\limits_{\mu=1}^m\lim\limits_{k\rightarrow\infty}\frac{1}{h_{1,k}}(f_\mu(x+h_{1,k}\vec{e}_l)-f_\mu(x))
\Big(\prod\limits_{\nu=1}^{\mu-1}f_\nu(x)\Big)\Big(\prod\limits_{j=\mu+1}^{m}f_j(x+h_{1,k}\vec{e}_l)\Big)\\
&\leq\displaystyle\sum\limits_{\mu=1}^mD_lf_\mu(x_1,x_2)\Big(\prod\limits_{i\neq\mu,1\leq i\leq m}f_i(x_1,x_2)\Big)
.\endaligned$$
This leads to the desired results.

{\rm (b)} Denote $[x_1-r_{1,1,1,k},x_1+r_{1,2,1,k}]\times
[x_2-r_{2,1,1,k}, x_2+r_{2,2,2,k}]$ by $R_k$. If there exists
$k_0\in\mathbb{N}$ such that $r_{1,1,1,k}+r_{1,2,1,k}>0$ and
$r_{2,1,1,k}+r_{2,2,1,k}>0$ when $k\geq k_0$.  Then 
\eqref{2.17} gives that
\begin{equation}\label{2.34}
\aligned
{}&D_l {\mathscr{M}}_{\mathcal{R}}(\vec{f})(x)
\leq\displaystyle\sum\limits_{\mu=1}^m\lim\limits_{k\rightarrow\infty}\frac{1}{(r_{1,1,1,k}+r_{1,2,1,k})(r_{2,1,1,k}+r_{2,2,1,k})}\iint_{R_k}(f_\mu)_{h_{1,k}}^l(y_1,y_2)dy_1dy_2\\
&\quad\times\displaystyle\Big(\prod\limits_{\nu=1}^{\mu-1}\frac{1}{(r_{1,1,1,k}+r_{1,2,1,k})(r_{2,1,1,k}+r_{2,2,1,k})}\iint_{R_k}f_\nu(y_1,y_2)dy_1dy_2\Big)\endaligned\end{equation}
\begin{equation*}
\aligned\quad\times\displaystyle\Big(\prod\limits_{j=\mu+1}^{m}\frac{1}{(r_{1,1,1,k}+r_{1,2,1,k})(r_{2,1,1,k}+r_{2,2,1,k})}\iint_{R_k}(f_j)_{\tau(h_{1,k})}^l(y_1,y_2)dy_1dy_2.
\endaligned\end{equation*}
Since $(x_1,x_2)$ is a Lebesgue point for $D_lf_\mu$, then
\begin{equation}\label{2.35}
\aligned
{}&\displaystyle\lim\limits_{k\rightarrow\infty}\Big|\frac{1}{(r_{1,1,1,k}+r_{1,2,1,k})(r_{2,1,1,k}+r_{2,2,1,k})}\iint_{R_k}(f_\mu)_{h_{1,k}}^l(y_1,y_2)dy_1dy_2-D_lf_\mu(x_1,x_2)\Big|\\
&\leq\lim\limits_{k\rightarrow\infty}\mathcal{M}_{\mathcal{R}}((f_\mu)_{h_{1,k}}^l-D_lf_\mu)(x_1,x_2)\\
&\quad+\displaystyle\lim\limits_{k\rightarrow\infty}\Big|\frac{1}{(r_{1,1,1,k}+r_{1,2,1,k})(r_{2,1,1,k}+r_{2,2,1,k})}\iint_{R_k}D_lf_\mu(y_1,y_2)-D_lf_\mu(x_1,x_2)dy_1dy_2\Big|\\
&=0.
\endaligned\end{equation}
Similarly, we have
\begin{equation}\label{2.36}
\lim\limits_{k\rightarrow\infty}\Big|\frac{1}{(r_{1,1,1,k}+r_{1,1,2,k})(r_{1,2,1,k}+r_{1,2,2,k})}\iint_{R_k}(f_\mu)_{\tau(h_{1,k})}^l(y_1,y_2)dy_1dy_2-f_\mu(x_1,x_2)\Big|=0.
\end{equation}
\eqref{2.34} together with \eqref{2.35}-\eqref{2.36} yields 
the desired estimate.

{\rm (c)} If there exists $k_0\in\mathbb{N}$ such that
$r_{1,1,1,k}+r_{1,2,1,k}>0$ when $k\geq k_0$ and $r_{2,1,1,k}
=r_{2,2,1,k}=0$ for infinitely many $k$. Then we may have
$$D_l{\mathscr{M}}_{\mathcal{R}}(\vec{f})(x)\leq\sum\limits_{\mu=1}^mD_lf_\mu(x_1,x_2)\Big(\prod\limits_{\nu\neq\mu,1\leq\nu\leq m}f_{\nu}(x_1,x_2)\Big).$$
This shows the desired upper bounds.

{\rm (d)} If there exists $k_0\in\mathbb{N}$ such that
$r_{2,1,1,k}+r_{2,2,1,k}>0$ when $k\geq k_0$ and $r_{1,1,1,k}
=r_{1,2,1,k}=0$ for infinitely many $k$, we can get the
upper bounds by the arguments similar to those used
in the case (c).

\eqref{2.33} together with (a)-(d) yields \eqref{2.16} for 
almost every $x\in R_{\vec{\Lambda}}(\vec{0})$. Since 
$\Lambda$ is arbitrary. This proves Lemma \ref{l2.5}.
\end{proof}

%%%%%%%%%%%%%%% Proof of Theorem 1.1 %%%%%%%%%%%%%%%%%%%%%%%%%%%%%%
\subsection{Proof of Theorem 1.1}
%%%%%%%%%%%%%%%%%%%%%%%%%%%%%%%%%%%%%%%%%%%%%%

We divide the proof into three steps:

%%%%%%%%%%% Step 1 %%%%%%%%%%%%%%%%%%%%%%%%%%%%%%%%%%%%%%%%%%%%%%%%
{\it Step 1: The boundedness part}. Let $1<p_1,\ldots,
p_m,p<\infty$ and $1/p=\sum_{i=1}^m1/p_i$.  Let
$\vec{f}=(f_1,\ldots,f_m)$ with each $f_i\in
W^{1,p_i}(\mathbb{R}^d)$. For a function $u$ and
$y\in\mathbb{R}^d$ we define $u_h(x)=u(x+h)$. According
to \cite[Section 7.11]{GT} we know that $u\in W^{1,p}
(\mathbb{R}^d)$ for $1<p<\infty$ if and only if
$u\in L^p(\mathbb{R}^d)$ and $\lim\sup_{h\rightarrow0}
\|u_h-u\|_{L^p(\mathbb{R}^d)}/|h|<\infty$. Therefore,
we have
\begin{equation}\label{2.37}
\limsup\limits_{h\rightarrow0}
\frac{\|(f_i)_h-f_i\|_{L^p(\mathbb{R}^d)}}{|h|}<\infty, \quad
\ \ \forall 1\leq i\leq m.
\end{equation}
On the other hand, for any fixed $h\in\mathbb{R}^d$ and
$x\in\mathbb{R}^d$, we have
\begin{equation}\label{2.38}
\begin{array}{ll}
|(\mathscr{M}_{\mathcal{R}}(\vec{f}))_h(x)
-\mathscr{M}_{\mathcal{R}}(\vec{f})(x)|
&\leq\displaystyle\sup\limits_{\substack{R \ni x \\ R\in\mathcal{R}}}
\frac{1}{|R|^m}\Big|\prod\limits_{i=1}^m\int_{R}|f_i(y+h)|dy
-\prod\limits_{i=1}^m\int_{R}|f_i(y)|dy\Big|
\\
&\leq\displaystyle\sum\limits_{i=1}^m
\sup\limits_{\substack{R \ni x \\ R\in\mathcal{R}}}\frac{1}{|R|^m}
\int_{R}|f_i(y+h)-f_i(y)|dy
\\
&\quad\times\displaystyle\Big(\prod\limits_{\mu=1}^{i-1}
\int_{R}|f_\mu(y)|dy\Big)\Big(\prod\limits_{\nu=\mu+1}^m
\int_{R}|f_i(y+h)|dy\Big)
\\
&\leq\displaystyle\sum\limits_{i=1}^m
\mathscr{M}_{\mathcal{R}}(\vec{f}_h^i)(x),
\end{array}
\end{equation}
where $\vec{f}_h^i=(f_1,\ldots,f_{i-1},(f_i)_h-f_i,
(f_{i+1})_h,\ldots,(f_m)_h)$.
%%%% 2.42 %%%%%%%%%%%%%%%%%%%%%%%%%%%%%%%%
\eqref{2.38} together with \eqref{1.2} yields that
\begin{equation}\label{2.39}
\begin{array}{ll}
&\|(\mathscr{M}_{\mathcal{R}}(\vec{f}))_h
-\mathscr{M}_{\mathcal{R}}(\vec{f})\|_{L^p(\mathbb{R}^d)}
\leq\displaystyle\sum\limits_{i=1}^m
\|\mathscr{M}_{\mathcal{R}}(\vec{f}_h^i)\|_{L^p(\mathbb{R}^d)}
\\
&\lesssim_{m,d,p_1,\ldots,p_m}\displaystyle\sum\limits_{i=1}^m
\|(f_i)_h-f_i\|_{L^{p_i}(\mathbb{R}^d)}
\prod\limits_{\mu\neq i,1\leq\mu\leq m}\|f_\mu\|_{L^{p_\mu}(\mathbb{R}^d)}.
\end{array}
\end{equation}
We get from \eqref{2.39} and \eqref{2.37} that 
$\limsup_{h\rightarrow0}\frac{\|(\mathscr{M}_{\mathcal{R}}
(\vec{f}))_h-\mathscr{M}_{\mathcal{R}}(\vec{f})\|_{L^p
(\mathbb{R}^d)}}{|h|}<\infty$. This together with the fact 
that $\mathscr{M}_{\mathcal{R}}(\vec{f})\in L^p(\mathbb{R}^d)$ 
yields that $\mathscr{M}_{\mathcal{R}}(\vec{f})\in W^{1,p}
(\mathbb{R}^d)$.
%%%%%%%%%%%%%%%%%%%%%%%% end Step 1 %%%%%%%%%%%%%%%%%%%%%%%%%%%%%%%%%
%%%%%%%%%%% Step 2 %%%%%%%%%%%%%%%%%%%%%%%%%%%%%%%%%%%%%%%%%%%%%%%%
%%%%%%%%%% Pointwise estimate for $\mathscr{M}_{\mathcal{R}} %%%%%%%%%%%%%
%%%%%%%%%%%%%%%%%%%%%%%%%%%%%%%%%%%%%%%%%%%%%%%%

{\it Step 2: Pointwise estimate for $\mathscr{M}_{\mathcal{R}}
(\vec{f})$.}
Let $s_k$ ($k=1,2,\ldots$) be an enumeration of positive rational
numbers. We can write
$$
\mathscr{M}_{\mathcal{R}}(\vec{f})(x)
=\sup\limits_{{\vec{r}\in
(\{s_k\}_{k=1}^\infty)^{2d}}}\frac{1}{|R_{\vec{r}}(x)|^{m}}
\prod\limits_{i=1}^m\int_{R_{\vec{r}}(x)}|f_i(y)|dy,
$$
where $\vec{r}=(r_1^-,\dots,r_d^-;r_1^+,\dots,r_d^+)$ and $R_{\vec{r}}(x)
=(x_1-r_1^-,\,x_1+r_1^+)\times\cdots\times(x_d-r_d^-,\,x_d+r_d^+)$. 
Fixing $k\geq1$, we let
$E_k=\{\vec{r}=(r_1^-,\dots,r_d^-;r_1^+,\dots,r_d^+)\in\mathbb{R}_{+}^{2d};
r_i^-,r_i^+\in\{s_1,\ldots,s_k\},\, i=1,2,\ldots,d\}$. For
$k\in\{1,2,\ldots\}$, we define the operator $T_k$ by
$$
T_k(\vec{f})(x)=\max\limits_{{\vec{r}\in E_k}}
\frac{1}{|R_{\vec{r}}(x)|^m}
\prod\limits_{i=1}^m\int_{R_{\vec{r}}(x)}|f_i(y)|dy .
$$
For any $h\in\mathbb{R}^d$, we can write
$$
\begin{array}{ll}
|T_k(\vec{f})(x+h)-T_k(\vec{f})(x)|
&\leq\displaystyle\sum\limits_{i=1}^m
\max\limits_{{\vec{r}\in{E}_k}}
\frac{1}{|R_{\vec{r}}(x)|^m}\int_{R_{\vec{r}}(x)}|f_i(y+h)-f_i(y)|dy
\\
&\quad\times\displaystyle\Big(\prod\limits_{\mu=1}^{i-1}
\int_{R_{\vec{r}}(x)}|f_\mu(y)|dy\Big)\Big(\prod\limits_{\nu=\mu+1}^m
\int_{R_{\vec{r}}(x)}|f_i(y+h)|dy\Big).
\end{array}
$$This yields that
\begin{equation}\label{2.40}
|D_l(T_k(\vec{f}))(x)|\leq\sum\limits_{i=1}^mT_k(\vec{f}_i^l)(x)
\leq\sum\limits_{i=1}^m\mathscr{M}_{\mathcal{R}}(\vec{f}_i^l)(x), \quad \hbox{for\ a.e.}\ x\in\mathbb{R}^d.
\end{equation}
Here $\vec{f}_i^l=(f_1,\ldots,f_{i-1},D_lf_i,f_{i+1},\ldots,f_m)$.
For all $k\geq1$, by \eqref{2.40} and \eqref{1.2}, it holds that
$$
\begin{array}{ll}
\|T_k(\vec{f})\|_{1,p}&\leq\displaystyle\|T_k(\vec{f})\|_{L^p(\mathbb{R}^d)}
+\sum\limits_{l=1}^d\|D_lT_k(\vec{f})\|_{L^p(\mathbb{R}^d)}
\\
&\leq\displaystyle\|\mathscr{M}_{\mathcal{R}}(\vec{f})\|_{L^p(\mathbb{R}^d)}
+\sum\limits_{l=1}^d\sum\limits_{i=1}^m
\|\mathscr{M}_{\mathcal{R}}(\vec{f}_i^l)
\|_{L^p(\mathbb{R}^d)}\\&\lesssim_{m,p_1,\ldots,p_m}
\prod\limits_{i=1}^m\|f_i\|_{1,p_i}.
\end{array}
$$
This yields that $\{T_k(\vec{f})\}_k$ is a bounded sequence
in $W^{1,p}(\mathbb{R}^d)$ which converges to
$\mathscr{M}_{\mathcal{R}}(\vec{f})$ pointwisely. The weak
compactness of Sobolev spaces implies that
$\{D_l(T_k(\vec{f}))\}_k$ converges to
$D_l(\mathscr{M}_{\mathcal{R}}(\vec{f}))$ weakly in
$L^p(\mathbb{R}^d)$. This together with \eqref{2.40} implies 
that
$$
|D_l\mathscr{M}_{\mathcal{R}}(\vec{f})(x)|
\leq\sum\limits_{i=1}^m\mathscr{M}_{\mathcal{R}}(\vec{f}_i^l)(x), \quad \hbox{for\ a.e.}\ x\in\mathbb{R}^d.
$$
Combining this with \eqref{1.2} yields that
$$
\begin{array}{ll}
\|\nabla\mathscr{M}_{\mathcal{R}}(\vec{f})\|_{L^p(\mathbb{R}^d)}
&\leq\displaystyle\sum\limits_{l=1}^d
\|D_l\mathscr{M}_{\mathcal{R}}(\vec{f})\|_{L^p(\mathbb{R}^d)}
\leq
\sum\limits_{l=1}^d\sum\limits_{i=1}^m
\|\mathscr{M}_{\mathcal{R}}(\vec{f}_i^l)\|_{L^p(\mathbb{R}^d)}
\\
&\lesssim_{m,d,p_1,\ldots,p_m}\displaystyle\sum\limits_{i=1}^m
\sum\limits_{l=1}^d\|D_lf_i\|_{L^{p_i}(\mathbb{R}^d)}
\prod\limits_{j\neq i,1\leq j\leq m}\|f_j\|_{L^{p_j}(\mathbb{R}^d)}.
\end{array}
$$
Therefore, it holds that
\begin{equation}\label{2.41}
\|\mathscr{M}_{\mathcal{R}}(\vec{f})\|_{1,p}
=\|\mathscr{M}_{\mathcal{R}}(\vec{f})\|_{L^p(\mathbb{R}^d)}
+\|\nabla\mathscr{M}_{\mathcal{R}}(\vec{f})\|_{L^p(\mathbb{R}^d)}
\leq C_{m,d,p_1,\ldots,p_m}\prod\limits_{i=1}^m\|f_i\|_{1,p_i}.
\end{equation}

{\it Step 3: The continuity part.} For convenience, we only
prove the case $d=2$ and the case $d>2$ is analogous and
more complex, we leave the details to the interested reader.
Let $\vec{f}=(f_1,\ldots,f_m)$ with each $f_i\in W^{1,p_i}
(\mathbb{R}^2)$ for $1<p_i<\infty$. Let $\vec{f}_j=(f_{1,j},
\ldots,f_{m,j})$ such that $f_{i,j}\rightarrow f_i$ in
$W^{1,p_i}(\mathbb{R}^2)$ when $j\rightarrow\infty$. Let
$1<p<\infty$ and $1/p=\sum_{i=1}^m1/p_i$. It follows from
\eqref{2.6} that $\|\mathscr{M}_{\mathcal{R}}(\vec{f}_j)
-\mathscr{M}_{\mathcal{R}}(\vec{f})\|_{L^p(\mathbb{R}^2)}
\rightarrow0$ when $j\rightarrow\infty$. Thus, it suffices
to show that, for any $l=1,2,\ldots,d$, it holds that
\begin{equation}\label{2.42}
\|D_l\mathscr{M}_{\mathcal{R}}(\vec{f}_j)
-D_l\mathscr{M}_{\mathcal{R}}(\vec{f})\|_{L^p(\mathbb{R}^2)}\rightarrow 0
\ \ \quad\quad {\rm when}\ j\rightarrow\infty.
\end{equation}
Without loss of generality we may assume that all
$f_{i,j}\geq0$ and $f_i\geq0$.

Given $\epsilon>0$ and $l=1,2$, letting $\vec{f}_l^i=(f_1,
\ldots,f_{i-1},D_lf_i,f_{i+1},\ldots,f_m)$, there exists
$\Lambda>0$ such that $\sum_{i=1}^m\|\mathscr{M}_{\mathcal{R}}
(\vec{f}_l^i)\|_{p,B_1}<\epsilon$ with $B_1=\mathbb{R}^2
\setminus R_{\vec{\Lambda}}(\vec{0})$. Here $\vec{\Lambda}
=(\Lambda,\Lambda)$. By the absolute continuity, there
exists $\eta>0$ such that
$\sum_{i=1}^m\|\mathscr{M}_{\mathcal{R}}(\vec{f}_l^i)
\|_{p,A}<\epsilon$ whenever $A$ is a measurable subset
of $R_{\vec{\Lambda}}(\vec{0})$ such that $|A|<\eta$.
As we already observed, for a.e. $x\in\mathbb{R}^2$, we
notice that:

(i) $u_{x,\vec{f}_l^i}$ is continuous on
$\overline{\mathbb{R}}_{+}^4$ and
$
\lim\limits_{(r_{1,1},r_{1,2},r_{2,1},r_{2,2})
\in\overline{\mathbb{R}}_{+}^4\atop
r_{1,1}+r_{1,2}+r_{1,2}+r_{2,2}\rightarrow
\infty}u_{x,\vec{f}_l^i}(r_{1,1},r_{1,2},r_{2,1},r_{2,2})=0;
$

(ii) $u_{x,\vec{f}_l^i}(r_{1,1},r_{1,2},0,0)$ is continuous
on $\overline{\mathbb{R}}_{+}^2$ and
$
\lim\limits_{(r_{1,1},r_{1,2})\in\overline{\mathbb{R}}_{+}^2
\atop r_{1,1}+r_{1,2}\rightarrow\infty}u_{x,\vec{f}_l^i}
(r_{1,1},r_{1,2},0,0)=0;
$

(iii) $u_{x,\vec{f}_l^i}(0,0,r_{2,1},r_{2,2})$ is continuous
on $\overline{\mathbb{R}}_{+}^2$ and
$
\lim\limits_{(r_{2,1},r_{2,2})\in\overline{\mathbb{R}}_{+}^2
\atop r_{2,1}+r_{2,2}\rightarrow\infty}u_{x,\vec{f}_l^i}
(0,0,r_{2,1},r_{2,2})=0.
$

Then, it follows that for a.e. $x\in\mathbb{R}^2$, the
function $\sum_{i=1}^mu_{x,\vec{f}_l^i}(\cdot,\cdot,
\cdot,\cdot)$ is uniformly continuous on
$\overline{\mathbb{R}}_{+}^4$; the function
$\sum_{i=1}^mu_{x,\vec{f}_l^i}(\cdot,\cdot,0,0)$
is uniformly continuous on $\overline{\mathbb{R}}_{+}^2$;
the function $\sum_{i=1}^mu_{x,\vec{f}_l^i}(0,0,\cdot,
\cdot)$ is uniformly continuous on
$\overline{\mathbb{R}}_{+}^2$. Hence, we can find
$\delta_x>0$ such that

(iv) If $\ |\vec{r_1}-\vec{r_2}|<\delta_x$, then
$
\big|\sum\limits_{i=1}^mu_{x,\vec{f}_l^i}(\vec{r_1})
-\sum\limits_{i=1}^mu_{x,\vec{f}_l^i}(\vec{r_2})\big|
<|R_{\vec{\Lambda}}(\vec{0})|^{-1/p}\epsilon;
$

(v) If $|r_{1,1,1}-r_{2,1,1}|+|r_{1,1,2}-r_{2,1,2}|<\delta_x$,
then
$$
\big|\sum\limits_{i=1}^mu_{x,\vec{f}_l^i}(r_{1,1,1},r_{1,1,2},0,0)
-\sum\limits_{i=1}^mu_{x,\vec{f}_l^i}(r_{2,1,1},r_{2,1,2},0,0)\big|
<|R_{\vec{\Lambda}}(\vec{0})|^{-1/p}\epsilon;
$$

(vi) If $|r_{1,2,1}-r_{2,2,1}|+|r_{1,2,2}-r_{2,2,2}|<\delta_x$,
then
$$
\big|\sum\limits_{i=1}^mu_{x,\vec{f}_l^i}(0,0,r_{1,2,1},r_{1,2,2})
-\sum\limits_{i=1}^mu_{x,\vec{f}_l^i}(0,0,r_{2,2,1},r_{2,2,2})\big|
<|R_{\vec{\Lambda}}(\vec{0})|^{-1/p}\epsilon.
$$

Now we can write
$$
R_{\vec{\Lambda}}(\vec{0})
=\Big(\bigcup\limits_{i=1}^\infty\Big\{x\in R_{\vec{\Lambda}}(\vec{0}):
\ \delta_x>\frac{1}{i}\Big\}\Big)\bigcup\mathcal{N},
$$
where $|\mathcal{N}|=0$. It follows that there exists
$\delta>0$ such that
\begin{equation}\label{2.43}
\begin{array}{ll}
{}&|\{x\in R_{\vec{\Lambda}}(\vec{0}):
|\sum_{i=1}^mu_{x,\vec{f}_l^i}(\vec{r}_1)-\sum_{i=1}^mu_{x,\vec{f}_l^i}
(\vec{r}_2)|\geq|R_{\vec{\Lambda}}(\vec{0})|^{-1/p}\epsilon
\ {\rm for\ some}\ \vec{r_1},\,\vec{r_2}
\\& \quad \quad{\rm with}\ |\vec{r_1}-\vec{r_2}|<\delta\}|
=:|B_{2,1}|<\frac{\eta}{2};
\end{array}
\end{equation}
\begin{equation}\label{2.44}
\begin{array}{ll}{}
&|\{x\in R_{\vec{\Lambda}}(\vec{0}):
|\sum_{i=1}^mu_{x,\vec{f}_l^i}(r_{1,1,1},r_{1,1,2},0,0)
-\sum_{i=1}^mu_{x,\vec{f}_l^i}(r_{2,1,1},r_{2,1,2},0,0)|
\geq|R_{\vec{\Lambda}}(\vec{0})|^{-1/p}\epsilon
\\
&\quad\ {\rm for\ some}\ r_{1,1,1},\,r_{1,1,2},\,r_{2,1,1},\,r_{2,1,2}
\ {\rm with}\ |r_{1,1,1}-r_{2,1,1}|+|r_{1,1,2}-r_{2,1,2}|<\delta\}|
\\
&=:|B_{2,2}|<\frac{\eta}{2};
\end{array}
\end{equation}
\begin{equation}\label{2.45}
\begin{array}{ll}{}
&|\{x\in R_{\vec{\Lambda}}(\vec{0}):
|\sum_{i=1}^mu_{x,\vec{f}_l^i}(0,0,r_{1,2,1},r_{1,2,2})
-\sum_{i=1}^mu_{x,\vec{f}_l^i}(0,0,r_{2,2,1},r_{2,2,2})|
\geq|R_{\vec{\Lambda}}(\vec{0})|^{-1/p}\epsilon
\\
&\quad\ {\rm for\ some}\ r_{1,2,1},\,r_{1,2,2},\,r_{2,2,1},\,r_{2,2,2}
\ {\rm with}\ |r_{1,2,1}-r_{2,2,1}|+|r_{1,2,2}-r_{2,2,2}|
<\delta\}|\\&=:|B_{2,3}|<\frac{\eta}{2}.
\end{array}
\end{equation}
Applying Lemma \ref{l2.3}, there exists $j_1\in\mathbb{N}$
such that for $i=1,2,3$
\begin{equation}\label{2.46}
|\{x\in R_{\vec{\Lambda}}(\vec{0});\mathcal{B}_i(\vec{f}_j)(x)
\nsubseteq\mathcal{B}_i(\vec{f})(x)_{(\delta)}\}|=:|B^{i,j}|<\frac{\eta}{2}
\ \ {\rm when}\ j\geq j_1.
\end{equation}
Let
$
\vec{f}_l^{i,j}=(f_{1,j},\ldots,f_{i-1,j},D_lf_{i,j},f_{i+1,j},
\ldots,f_{m,j}).
$
Fix $i=1,2,3$. Invoking Lemma \ref{l2.5}, for a.e. $x\in\mathbb{R}^2$
, $j\geq j_1$, and for any $\vec{r_1}\in\mathcal{B}_i(\vec{f}_j)(x)$
and $\vec{r_2}\in\mathcal{B}_i(\vec{f})(x)$ with $i=1,2,3$, we have
\begin{equation}\label{2.47}
\begin{array}{ll}
&\Big|D_l {\mathscr{M}}_{\mathcal{R}}(\vec{f}_j)(x)
-D_l{\mathscr{M}}_{\mathcal{R}}(\vec{f})(x)\Big|
\\
&=\displaystyle\Big|\sum\limits_{i=1}^mu_{x,\vec{f}_l^{i,j}}(\vec{r_1})
-\sum\limits_{i=1}^mu_{x,\vec{f}_l^i}(\vec{r_2})\Big|
\\
&\leq\displaystyle\sum\limits_{i=1}^m|u_{x,\vec{f}_l^{i,j}}(\vec{r_1})
-u_{x,\vec{f}_l^i}(\vec{r_1})|
+\Big|\sum\limits_{i=1}^mu_{x,\vec{f}_l^i}(\vec{r_1})
-\sum\limits_{i=1}^mu_{x,\vec{f}_l^i}(\vec{r_2})\Big|.
\end{array}
\end{equation}
If $x\notin B_1\cup B_{2,i}\cup B^{i,j}$, we choose
$\vec{r_1}\in\mathcal{B}_i(\vec{f}_j)(x)$ and $\vec{r_2}
\in\mathcal{B}_i(\vec{f})(x)$ such that $|\vec{r_1}-
\vec{r_2}|<\delta$ and
\begin{equation}\label{2.48}
\Big|\sum\limits_{i=1}^mu_{x,\vec{f}_l^i}(\vec{r_1})
-\sum\limits_{i=1}^mu_{x,\vec{f}_l^i}(\vec{r_2})\Big|
<|R_{\vec{\Lambda}}(\vec{0})|^{-1/p}\epsilon.
\end{equation}
On the other hand, for any $\vec{r_1}\in\mathcal{B}_i
(\vec{f}_j)(x)$ and $\vec{r_2}\in\mathcal{B}_i(\vec{f})(x)$,
one may obtain that
\begin{equation}\label{2.49}
\Big|\sum\limits_{i=1}^mu_{x,\vec{f}_l^i}(\vec{r_1})
-\sum\limits_{i=1}^mu_{x,\vec{f}_l^i}(\vec{r_2})\Big|
\leq 2\sum\limits_{i=1}^m{\mathscr{M}}_{\mathcal{R}}(\vec{f}_l^i)(x), \quad \hbox{for a.e.}\  x\in\mathbb{R}^2.
\end{equation}
To get the estimate of $\sum_{i=1}^m|u_{x,\vec{f}_l^{i,j}}
(\vec{r_1})-u_{x,\vec{f}_l^i}(\vec{r_1})|$, we consider 
the following cases:

{\bf Case 1.} For simplicity, we denote $\iint_{{R_0}}=\int_{x_1
-r_{1,1}}^{x_1+r_{1,2}}\int_{x_2-r_{2,1}}^{x_2+r_{2,2}}$. If
$\vec{r_1}=(r_{1,1},r_{1,2},r_{2,1},r_{2,2})\in
\overline{\mathbb{R}}_{+}^4$ with $r_{1,1}+r_{1,2}>0$ and
$r_{2,1}+r_{2,2}>0$. Then
\begin{eqnarray*}
&&|u_{x,\vec{f}_l^{i,j}}(\vec{r_1})-u_{x,\vec{f}_l^i}(\vec{r_1})|\\
&&=\displaystyle{\prod\limits_{w=1}^2(r_{w,1}+r_{w,2})^{-m}}\Big|\Big(
\prod\limits_{\mu=1}^{i-1}\iint_{{R_0}}f_{\mu,j}(y_1,y_2)dy_1dy_2\Big)
\Big(\iint_{{R_0}}D_lf_{i,j}(y_1,y_2)dy_1dy_2\Big)
\\&&\quad\times\Big(\prod\limits_{\nu=i+1}^{m}\iint_{{R_0}}f_{\nu,j}(y_1,y_2)dy_1dy_2\Big)
\\
&&\quad-
\Big(\prod\limits_{\mu=1}^{i-1}
\iint_{{R_0}}
f_{\mu}(y_1,y_2)dy_1dy_2\Big)
\Big(\iint_{{R_0}}
D_lf_{i}(y_1,y_2)dy_1dy_2\Big)\displaystyle\Big(\prod\limits_{\nu=i+1}^{m}
\iint_{{R_0}}f_{\nu}(y_1,y_2)dy_1dy_2\Big)\Big|
\\
&&\leq\displaystyle\sum\limits_{\mu=1}^{i-1}{\prod\limits_{w=1}^2
(r_{w,1}+r_{w,2})^{-m}}\Big(\prod\limits_{\ell=1}^{\mu-1}
\iint_{R_0}f_{\ell}(y_1,y_2)dy_1y_2\Big)
\iint_{R_0}|f_{\mu,j}-f_{\mu}|(y_1,y_2)dy_1dy_2\\
&&\quad\times\displaystyle\Big(
\prod\limits_{\kappa=\mu+1}^{i-1}\iint{R_0}f_{\kappa,j}(y_1,y_2)dy_1dy_2\Big)
\iint_{R_0}D_lf_{i,j}(y_1,y_2)dy_1dy_2
\displaystyle\\&&\quad \times\Big(\prod\limits_{\tau=i+1}^m\iint{R_0}f_{\tau,j}(y_1,y_2)dy_1dy_2\Big)
+\sum\limits_{\nu=i+1}^m{\prod\limits_{w=1}^2
(r_{w,1}+r_{w,2})^{-m}}\Big(\prod\limits_{\ell=1}^{i-1}
\iint_{R_0}f_\ell(y_1,y_2)dy_1dy_2\Big)\\
&&\quad\times\iint_{R_0}
D_lf_i(y_1,y_2)dy_1dy_2\Big(\prod\limits_{\kappa=i+1}^{\nu-1}
\iint_{R_0}
f_\kappa(y_1,y_2)dy_1dy_2\Big)\iint_{R_0}|f_{\nu,j}-f_\nu|(y_1,y_2)dy_1dy_2
\end{eqnarray*}
\begin{eqnarray*}
&&\quad\times\Big(\prod\limits_{\tau=\nu+1}^m
\iint_{R_0}
f_{\tau,j}(y_1,y_2)dy_1dy_2\Big)
+{\prod\limits_{w=1}^2(r_{w,1}+r_{w,2})^{-m}}\Big(\prod\limits_{\kappa=1}^{i-1}\iint{R_0}f_\kappa(y_1,y_2)dy_1dy_2\Big)
\\
&&\quad\times\iint_{R_0} |D_lf_{i,j}-D_lf_i|(y_1,y_2)dy_1dy_2\Big(\prod\limits_{\tau=i+1}^{m}\iint{R_0}f_{\tau,j}(y_1,y_2)dy_1dy_2\Big)
\\
&&\leq\sum\limits_{\mu=1}^{i-1}\mathscr{M}_{\mathcal{R}}(\vec{F}_{\mu,j}^l)(x)
+\sum\limits_{\nu=i+1}^m\mathscr{M}_{\mathcal{R}}(\vec{G}_{\nu,j}^l)(x)
+\mathscr{M}_{\mathcal{R}}(\vec{H}_{i,j}^l)(x)
=:\mathcal{G}_{i,j}^l(x),
\end{eqnarray*}
where $\vec{F}_{\mu,j}^l=(f_1,\ldots,f_{\mu-1},f_{\mu,j}-f_\mu,
f_{\mu+1,j},\ldots,f_{i-1,j},D_lf_{i,j},f_{i+1,j},\ldots,f_{m,j}),
$ and $\vec{G}_{\nu,j}^l, \vec{H}_{i,j}^l$ are defined by
$\vec{G}_{\nu,j}^l=(f_1,\ldots,f_{i-1},D_lf_i,f_{i+1},\ldots,
f_{\nu-1},f_{\nu,j}-f_\nu,f_{\nu+1,j},\ldots,f_{m,j})$ and
$\vec{H}_{i,j}^l=(f_1,\ldots,f_{i-1},D_lf_{i,j}-D_lf_i,f_{i+1,j},
\ldots,f_{m,j})$.

{\bf Case 2.} If $\vec{r_1}=(0,0,0,0)$, then
$$\aligned
|u_{x,\vec{f}_l^{i,j}}(\vec{r_1})-u_{x,\vec{f}_l^i}(\vec{r_1})|
&\leq\displaystyle\sum\limits_{\mu=1}^{i-1}\Big(\prod\limits_{\ell=1}^{\mu-1}
f_{\ell}(x)\Big)(f_{\mu,j}-f_{\mu})(x)
\Big(\prod\limits_{\kappa=\mu+1}^{i-1}f_{\kappa,j}(x)\Big)D_lf_{i,j}(x)
\Big(\prod\limits_{\tau=i+1}^mf_{\tau,j}(x)\Big)
\\
&\quad\displaystyle+\sum\limits_{\nu=i+1}^m\Big(\prod\limits_{\ell=1}^{i-1}
f_\ell(x)\Big)D_lf_i(x)\Big(\prod\limits_{\kappa=i+1}^{\nu-1}f_\kappa(x)
\Big)(f_{\nu,j}-f_\nu)(x)
\Big(\prod\limits_{\tau=\nu+1}^mf_{\tau,j}(x)\Big)\\
&\quad\displaystyle+\Big(\prod\limits_{\kappa=1}^{i-1}f_\kappa(x)\Big)
(D_lf_{i,j}-D_lf_i)(x)\Big(\prod\limits_{\tau=i+1}^{m}f_{\tau,j}(x)\Big).
\endaligned
$$

{\bf Case 3.} If $\vec{r_1}=(0,0,r_{2,1},r_{2,2})\in
\overline{\mathbb{R}}_{+}^4$ for $r_{2,1}+r_{2,2}>0$, then
$$\begin{array}{ll}
&|u_{x,\vec{f}_l^{i,j}}(\vec{r_1})-u_{x,\vec{f}_l^i}(\vec{r_1})|\\
&=\displaystyle\frac{1}{(r_{2,1}+r_{2,2})^m}\Big|\Big(\prod\limits_{\mu=1}^{i-1}\int_{x_2-r_{2,1}}^{x_2+r_{2,2}}f_{\mu,j}(x_1,y_2)dy_2\Big)\\
&\quad\times\displaystyle\Big(\int_{x_2-r_{2,1}}^{x_2+r_{2,2}}D_lf_{i,j}(x_1,y_2)dy_2\Big)
\Big(\prod\limits_{\nu=i+1}^{m}\int_{x_2-r_{2,1}}^{x_2+r_{2,2}}f_{\nu,j}(x_1,y_2)dy_2\Big)\\
&\quad\displaystyle-\Big(\prod\limits_{\mu=1}^{i-1}\int_{x_2-r_{2,1}}^{x_2+r_{2,2}}f_{\mu}(x_1,y_2)dy_2\Big)\\
&\quad\times\displaystyle\Big(\int_{x_2-r_{2,1}}^{x_2+r_{2,2}}D_lf_{i}(x_1,y_2)dy_2\Big)
\Big(\prod\limits_{\nu=i+1}^{m}\int_{x_2-r_{2,1}}^{x_2+r_{2,2}}f_{\nu}(x_1,y_2)dy_2\Big)\Big|\\
&\leq\displaystyle\sum\limits_{\mu=1}^{i-1}\frac{1}{(r_{2,1}+r_{2,2})^m}\Big(\prod\limits_{\ell=1}^{\mu-1}\int_{x_2-r_{2,1}}^{x_2+r_{2,2}}f_{\ell}(x_1,y_2)dy_2\Big)
\int_{x_2-r_{2,1}}^{x_2+r_{2,2}}|f_{\mu,j}-f_{\mu})|(x_1,y_2)dy_2\\
&\quad\times\displaystyle\Big(\prod\limits_{\kappa=\mu+1}^{i-1}\int_{x_2-r_{2,1}}^{x_2+r_{2,2}}f_{\kappa,j}(x_1,y_2)dy_2\Big)
\int_{x_2-r_{2,1}}^{x_2+r_{2,2}}D_lf_{i,j}(x_1,y_2)dy_2\\
&\quad\times\displaystyle\Big(\prod\limits_{\tau=i+1}^m\int_{x_2-r_{2,1}}^{x_2+r_{2,2}}f_{\tau,j}(x_1,y_2)dy_2\Big)\\
&\quad\displaystyle+\sum\limits_{\nu=i+1}^m\frac{1}{(r_{2,1}+r_{2,2})^m}\Big(\prod\limits_{\ell=1}^{i-1}\int_{x_2-r_{2,1}}^{x_2+r_{2,2}}f_\ell(x_1,y_2)dy_2\Big)\int_{x_2-r_{2,1}}^{x_2+r_{2,2}}D_lf_i(x_1,y_2)dy_2
\end{array}$$
$$\begin{array}{ll}
&\quad\times\displaystyle\Big(\prod\limits_{\kappa=i+1}^{\nu-1}\int_{x_2-r_{2,1}}^{x_2+r_{2,2}}f_\kappa(x_1,y_2)dy_2\Big)\int_{x_2-r_{2,1}}^{x_2+r_{2,2}}|f_{\nu,j}-f_\nu|(x_1,y_2)dy_2\\
&\quad\times\displaystyle\Big(\prod\limits_{\tau=\nu+1}^m\int_{x_2-r_{2,1}}^{x_2+r_{2,2}}f_{\tau,j}(x_1,y_2)dy_2\Big)\\&\quad\displaystyle+\frac{1}{(r_{2,1}+r_{2,2})^m}\Big(\prod\limits_{\kappa=1}^{i-1}\int_{x_2-r_{2,1}}^{x_2+r_{2,2}}f_\kappa(x_1,y_2)dy_2\Big)\int_{x_2-r_{2,1}}^{x_2+r_{2,2}}|D_lf_{i,j}-D_lf_i|(x_1,y_2)dy_2\\
&\quad\times\displaystyle\Big(\prod\limits_{\tau=i+1}^{m}\int_{x_2-r_{2,1}}^{x_2+r_{2,2}}f_{\tau,j}(x_1,y_2)dy_2\Big).
\end{array}$$

{\bf Case 4.} If $\vec{r_1}=(r_{1,1},r_{1,2},0,0)\in
\overline{\mathbb{R}}_{+}^4$ with $r_{1,1}+r_{1,2}>0$.
Then, similarly as in Case 3, we can obtain
\begin{eqnarray*}
&&|u_{x,\vec{f}_l^{i,j}}(\vec{r_1})-u_{x,\vec{f}_l^i}(\vec{r_1})|\\
&&\leq\displaystyle\sum\limits_{\mu=1}^{i-1}\frac{1}{(r_{1,1}+r_{1,2})^m}\Big(\prod\limits_{\ell=1}^{\mu-1}\int_{x_1-r_{1,1}}^{x_1+r_{1,2}}f_{\ell}(y_1,x_2)dy_1\Big)
\int_{x_1-r_{1,1}}^{x_1+r_{1,2}}|f_{\mu,j}-f_{\mu})|(y_1,x_2)dy_1\\
&&\quad\times\displaystyle\Big(\prod\limits_{\kappa=\mu+1}^{i-1}\int_{x_1-r_{1,1}}^{x_1+r_{1,2}}f_{\kappa,j}(y_1,x_2)dy_1\Big)
\int_{x_1-r_{1,1}}^{x_1+r_{1,2}}D_lf_{i,j}(y_1,x_2)dy_1\\
&&\quad\times\displaystyle\Big(\prod\limits_{\tau=i+1}^m\int_{x_1-r_{1,1}}^{x_1+r_{1,2}}f_{\tau,j}(y_1,x_2)dy_1\Big)\end{eqnarray*}\begin{eqnarray*}
&&\quad\displaystyle+\sum\limits_{\nu=i+1}^m\frac{1}{(r_{1,1}+r_{1,2})^m}\Big(\prod\limits_{\ell=1}^{i-1}\int_{x_1-r_{1,1}}^{x_1+r_{1,2}}f_\ell(y_1,x_2)dy_1\Big)\int_{x_1-r_{1,1}}^{x_1+r_{1,2}}D_lf_i(y_1,x_2)dy_1\\&&\quad\times\displaystyle\Big(\prod\limits_{\kappa=i+1}^{\nu-1}\int_{x_1-r_{1,1}}^{x_1+r_{1,2}}f_\kappa(y_1,x_2)dy_1\Big)\int_{x_1-r_{1,1}}^{x_1+r_{1,2}}|f_{\nu,j}-f_\nu|(y_1,x_2)dy_1\\
&&\quad\times\displaystyle\Big(\prod\limits_{\tau=\nu+1}^m\int_{x_1-r_{1,1}}^{x_1+r_{1,2}}f_{\tau,j}(y_1,x_2)dy_1\Big)\\
&&\quad\displaystyle+\frac{1}{(r_{1,1}+r_{1,2})^m}\Big(\prod\limits_{\kappa=1}^{i-1}\int_{x_1-r_{1,1}}^{x_1+r_{1,2}}f_\kappa(y_1,x_2)dy_1\Big)\int_{x_1-r_{1,1}}^{x_1+r_{1,2}}|D_lf_{i,j}-D_lf_i|(y_1,x_2)dy_1\\
&&\quad\times\displaystyle\Big(\prod\limits_{\tau=i+1}^{m}\int_{x_1-r_{1,1}}^{x_1+r_{1,2}}f_{\tau,j}(y_1,x_2)dy_1\Big).
\end{eqnarray*}
Together with the above cases, we obtain
\begin{equation}\label{2.50}
\sum\limits_{i=1}^m|u_{x,\vec{f}_l^{i,j}}(\vec{r_1})-u_{x,\vec{f}_l^i}(\vec{r_1})|
\leq\sum\limits_{i=1}^m\mathcal{G}_l^{i,j}(x)=:\mathcal{G}_l^j(x), \quad\hbox{for any } \vec{r_1}\in[0,\infty)^4.
\end{equation}
Note that
$$
\lim\limits_{j\rightarrow\infty}\|\mathcal{G}_l^{i,j}\|_{L^p(\mathbb{R}^d)}=0.
$$
It follows that there exists $j_2\in\mathbb{N}$ such that
\begin{equation}\label{2.51}
\|\mathcal{G}_{l}^j\|_{L^p(\mathbb{R}^2)}<\epsilon, \quad\ \ \forall j\geq j_2.
\end{equation}
Observe from \eqref{2.43}-\eqref{2.46} that $|B_{2,i}\cup B^{i,j}|
<\eta$ for all $j\geq j_1$ and $i=1,2,3$. These facts together with
(2.47)-(2.51) imply that
$$
\begin{array}{ll}
{}&\|D_l {\mathscr{M}}_{\mathcal{R}}(\vec{f}_j)
-D_l {\mathscr{M}}_{\mathcal{R}}(\vec{f})\|_{L^p(\mathbb{R}^2)}
\\
&\leq\displaystyle\|\mathcal{G}_l^{j}\|_{L^p(\mathbb{R}^2)}
+
\Big\|2 \sum\limits_{i=1}^m{\mathscr{M}}_{\mathcal{R}}(\vec{f}_l^i)\Big\|_{p,B_1}
+\Big\|2 \sum\limits_{i=1}^m{\mathscr{M}}_{\mathcal{R}}(\vec{f}_l^i)
\Big\|_{p,B_{2,1}\cup B^{1,j}}
+\Big\|2 \sum\limits_{i=1}^m{\mathscr{M}}_{\mathcal{R}}(\vec{f}_l^i)
\Big\|_{p,B_{2,2}\cup B^{2,j}}
\\
&\quad
+\displaystyle\Big\|2 \sum\limits_{i=1}^m{\mathscr{M}}_{\mathcal{R}}(\vec{f}_l^i)
\Big\|_{p,B_{2,3}\cup B^{3,j}}+\bigl\||R_{\vec{\Lambda}}(\vec{0})|^{-1/p}\epsilon
\bigr\|_{p,(B_1\cup(\cup_{i=1}^{3}B_{2,i}\cup B^{i,j}))^c}
\leq10\epsilon
\end{array}
$$
for all $j\geq\max\{j_1,j_2\}$, which leads to \eqref{2.42}. $\hfill\Box$
%%%%%%%%%%%%%%%%%%%%%%%%%%%%%%%%%%%%%%%%%%%%%%%%%%%%
%%%%%%%%%%%%%%% end Proof of Theorem 1.1 %%%%%%%%%%%%%%%%%%%%%%%%%%%%%%%%%%
%%%%%%%%%%%%%%%%  %%%%%%%%%%%%%%%%%%%%%%%%%%%%%%%
%%%%%%%%%%%%%%% Properties on Besov and Triebel-Lizorkin spaces %%%%%%%%%%%%%
%%%%%%%%%%%%%%%%%%%%%%%%%%%%%%%%%%%%%%%%%%%%%%%%%%%%%%
\section {Properties on Besov and Triebel-Lizorkin spaces} \label{S3}
%%%%%%%%%%%%%%%%%%%%%%%%%%%%%%%%%%%%%%%%%%%%%%%%%%%%%%%%%%%%%%%%%%%%
This section will be devoted to give the proofs of Theorems \ref{thm2}
and \ref{thm3}. In what follows, we let $\Delta_\zeta f$ denote
the difference of $f$, i.e. $\Delta_\zeta f(x)=f(x+\zeta)
-f(x)$ for all $x,\,\zeta\in\mathbb{R}^d$. We also let
$\mathfrak{R}_d=\{\zeta\in\mathbb{R}^d;1/2<|\zeta|\leq1\}$.

To prove Theorems \ref{thm2} and \ref{thm3}, we need the
following characterizations of homogeneous Triebel-Lizorkin
spaces $\dot{F}_s^{p,q}(\mathbb{R}^d)$ and homogeneous Besov
spaces $\dot{B}_s^{p,q}(\mathbb{R}^d)$.

%%%%%%%%%% Lemma 3.1 %%%%%%%%%%%%%%%%%%%%%%%%%%%%%%%%
\begin{lemma}\label{l3.1}
(\cite{Ya}). {\rm (i)} Let $0<s<1$, $1<p<\infty$,
$1<q\leq\infty$ and $1\leq r<\min(p,q)$. Then
$$
\|f\|_{\dot{F}_s^{p,q}(\mathbb{R}^d)}
\sim\Big\|\Big(\sum\limits_{k\in\mathbb{Z}}2^{ksq}
\Big(\int_{\mathfrak{R}_d}|\Delta_{2^{-k}\zeta}f|^rd\zeta\Big)^{q/r}
\Big)^{{1}/{q}}\Big\|_{L^p(\mathbb{R}^d)};
$$

{\rm (ii)} Let $0<s<1$, $1\leq p<\infty$, $1\leq q\leq\infty$
and $1\leq r\leq p$. Then
\begin{equation}\label{3.1}
\|f\|_{\dot{B}_s^{p,q}(\mathbb{R}^d)}\sim\Big(\sum\limits_{k\in\mathbb{Z}}
2^{ksq}\Big\|\Big(\int_{\mathfrak{R}_d}
|\Delta_{2^{-k}\zeta}f|^rd\zeta\Big)^{1/r}\Big\|_{L^p(\mathbb{R}^d)}^q
\Big)^{1/q}.
\end{equation}
\end{lemma}
%%%%%%%%%% end Lemma 3.1 %%%%%%%%%%%%%%%%%%%%%%%%%%%%%%%%

%%%%%%%%%%%%%%%%%%%%%%%%%%%%%%%%%%%%%%%%%%%%%%%%%%%%%%%%%%%%%%%%
%%%%%%%%%%%%% Proof of Theorem 1.2 %%%%%%%%%%%%%%%%%%%%%%%%%%%
\quad\hspace{-20pt}{\it Proof of Theorem \ref{thm2}.}
Note that $\|f\|_{B_s^{p,q}(\mathbb{R}^d)}\thicksim\|f
\|_{\dot{B}_s^{p,q}(\mathbb{R}^d)}+\|f\|_{L^p(\mathbb{R}^d)}$ 
for $s>0$ and $1<p,q<\infty$. For a measurable function
$g:\mathbb{R}^d\times\mathbb{Z}\times\mathfrak{R}_d
\rightarrow\mathbb{R}$, we define
$$
\|g\|_{p,q}:=\Big(\sum\limits_{k\in\mathbb{Z}}2^{ksq}
\Big(\int_{\mathfrak{R}_d}
\int_{\mathbb{R}^d}|g(x,k,\zeta)|^pdxd\zeta\Big)^{q/p}\Big)^{1/q}.
$$
Using \eqref{3.1} with $r=p$ and Fubini's theorem, we have
\begin{equation}\label{3.2}
\|f\|_{\dot{B}_s^{p,q}(\mathbb{R}^d)}\thicksim\|\Delta_{2^{-k}\zeta}f\|_{p,q}.
\end{equation}
Let $0<s<1$ and $1<p_1,\ldots,p_m,p,q<\infty$ with
$1/p=\sum_{i=1}^m1/p_i$. Let $\vec{f}=(f_1,\ldots,f_m)$
with each $f_j\in B_s^{p_j,q}(\mathbb{R}^d)$. Fix
$\zeta\in\mathbb{R}^d$, it is clear that
$$
\mathscr{M}_{\mathcal{R}}(\vec{f})(x+\zeta)=\sup\limits_{R\ni x+\zeta\atop
R\in\mathcal{R}}\frac{1}{|R|^{m}}\prod\limits_{i=1}^m\int_{R}|f_i(y)|dy
=\sup\limits_{R\ni x\atop R\in\mathcal{R}}\frac{1}{|R|^{m}}
\prod\limits_{i=1}^m\int_{R}|f_i(y+\zeta)|dy.
$$
One can easily check that
\begin{equation}\label{3.3}
|\Delta_{2^{-k}\zeta}(\mathscr{M}_{\mathcal{R}}(\vec{f}))(x)|
\le \displaystyle\sum\limits_{l=1}^m\mathscr{M}_{\mathcal{R}}(\vec{f}_l^{k,\zeta})(x),
\end{equation}
where $\vec{f}_l^{k,\zeta}=(f_1,\ldots,f_{l-1},
\Delta_{2^{k}\zeta}f_l,f_{l+1}^{k,\zeta},\ldots,f_m^{k,
\zeta})$ and $f_j^{k,\zeta}(x)=f_j(x+2^{-k}\zeta)$ for
all $l+1\leq j\leq m$. Then we get from 
\eqref{3.2}-\eqref{3.3} and Minkowski's inequality that
\begin{equation}\label{3.4}
\begin{array}{ll}
&\|\mathscr{M}_{\mathcal{R}}(\vec{f})\|_{\dot{B}_s^{p,q}(\mathbb{R}^d)}
\\
&\lesssim\displaystyle\Big(\sum\limits_{k\in\mathbb{Z}}2^{ksq}
\Big(\int_{\mathfrak{R}_d}\int_{\mathbb{R}^d}
|\Delta_{2^{-k}\zeta}\mathscr{M}_{\mathcal{R}}(\vec{f})(x)|^pdxd\zeta
\Big)^{q/p}\Big)^{1/q}
\\
&\lesssim\displaystyle\sum\limits_{l=1}^m\Big(\sum\limits_{k\in\mathbb{Z}}
2^{ksq}\Big\|\|\mathscr{M}_{\mathcal{R}}(\vec{f}_l^{k,\zeta})\|_{L^p(\mathbb{R}^d)}
\Big\|_{L^p(\mathfrak{R}_d)}^q\Big)^{1/q}
\\
&\lesssim\displaystyle\sum\limits_{l=1}^m\Big(\sum\limits_{k\in\mathbb{Z}}
2^{ksq}\Big\|\prod\limits_{i=1}^{l-1}\|f_i\|_{L^{p_i}(\mathbb{R}^d)}
\|\Delta_{2^{-k}\zeta}f_l\|_{L^{p_l}(\mathbb{R}^d)}
\prod\limits_{j=l+1}^m\|f_j^{k,\zeta}\|_{L^{p_j}(\mathbb{R}^d)}
\Big\|_{L^p(\mathfrak{R}_d)}^q\Big)^{1/q}
\\
&\lesssim\displaystyle\sum\limits_{l=1}^m\prod\limits_{i\neq l,1\leq i\leq m}
\|f_i\|_{L^{p_i}(\mathbb{R}^d)}\Big(\sum\limits_{k\in\mathbb{Z}}
2^{ksq}\Big\|\|\Delta_{2^{-k}\zeta}f_l\|_{L^{p_l}(\mathbb{R}^d)}
\Big\|_{L^{p_l}(\mathfrak{R}_d)}^q\Big)^{1/q}
\\
&\lesssim\displaystyle\sum\limits_{l=1}^m\prod\limits_{i\neq l,1\leq i\leq m}
\|f_i\|_{L^{p_i}(\mathbb{R}^d)}\|f_l\|_{\dot{B}_s^{p_l,q}(\mathbb{R}^d)}.
\end{array}
\end{equation}
\eqref{3.4} together with \eqref{1.2} implies that
$$
\|\mathscr{M}_{\mathcal{R}}(\vec{f})\|_{B_s^{p,q}(\mathbb{R}^d)}
\leq C\prod\limits_{i=1}^m\|f_i\|_{B_s^{p_i,q}(\mathbb{R}^d)}.
$$
This completes the proof of the boundedness part.
%%%%%%%%%%%%%%% end proof of the boundedness part %%%%%%%%%%%%%%%%
%%%%%%%%%%%%%%%%%%%%%%%%%%%%%%%%%%%%%%%%%%%%%%%%%%%%%%%%%

We now prove the continuity part. Let $\vec{f}_j=(f_{1,j},
\ldots,f_{m,j})$ and $f_{i,j}\rightarrow f_i$ in
$B_s^{p_i,q}(\mathbb{R}^d)$ as $j\rightarrow\infty$. It is
known that $f_{i,j}\rightarrow f_i$ in $\dot{B}_s^{p_i,q}
(\mathbb{R}^d)$ and in $L^{p_i}(\mathbb{R}^d)$ as
$j\rightarrow\infty$. One can check that
\begin{equation}\label{3.5}
|\mathscr{M}_{\mathcal{R}}(\vec{f}_j)-\mathscr{M}_{\mathcal{R}}(\vec{f})|
\leq\sum\limits_{l=1}^m\mathscr{M}_{\mathcal{R}}(\vec{f}^l).
\end{equation}
Here $\vec{f}^l=(f_1,\ldots,f_{l-1},f_{l,j}-f_l,f_{l+1,j},
\ldots,f_{m,j})$. It follows from \eqref{3.5} that
$\mathscr{M}_{\mathcal{R}}(\vec{f}_j)\rightarrow
\mathscr{M}_{\mathcal{R}}(\vec{f})$ in $L^p(\mathbb{R}^d)$
as $j\rightarrow\infty$. Therefore, it suffices to show
that $\mathscr{M}_{\mathcal{R}}(\vec{f}_j)\rightarrow
\mathscr{M}_{\mathcal{R}}(\vec{f})$ in $\dot{B}_s^{p,q}
(\mathbb{R}^d)$ as $j\rightarrow\infty$.
We will prove this claim by contradiction.

Without loss of generality, we may assume that there
exists $c>0$ such that
$$
\|\mathscr{M}_{\mathcal{R}}(\vec{f}_j)
-\mathscr{M}_{\mathcal{R}}(\vec{f})\|_{\dot{B}_s^{p,q}(\mathbb{R}^d)}>c, \quad \hbox{for\ every }j.
$$
It is obvious that $\|\Delta_{2^{-k}\zeta}
(\mathscr{M}_{\mathcal{R}}(\vec{f}_j)
-\mathscr{M}_{\mathcal{R}}(\vec{f}))\|_{L^p(\mathbb{R}^d)}
\rightarrow0$ as $j\rightarrow\infty$ for every $(k,\zeta)
\in\mathbb{Z}\times\mathfrak{R}_d$. By \eqref{3.3}, for 
every $(x,k,\zeta)\in\mathbb{R}^d\times\mathbb{Z}\times
\mathfrak{R}_d$, we have
\begin{equation}\label{3.6}
\begin{array}{ll}
&|\Delta_{2^{-k}\zeta}(\mathscr{M}_{\mathcal{R}}(\vec{f}_j)-\mathscr{M}_{\mathcal{R}}(\vec{f}))(x)|\\
&\leq|\Delta_{2^{-k}\zeta}(\mathscr{M}_{\mathcal{R}}(\vec{f}_j))(x)|+|\Delta_{2^{-k}\zeta}(\mathscr{M}_{\mathcal{R}}(\vec{f}))(x)|\\
&\leq\displaystyle\sum\limits_{l=1}^m|\mathscr{M}_{\mathcal{R}}(\vec{f}_{l,j}^{k,\zeta})(x)-\mathscr{M}_{\mathcal{R}}(\vec{f}_{l}^{k,\zeta})(x)|+2\sum\limits_{l=1}^m\mathscr{M}_{\mathcal{R}}(\vec{f}_{l}^{k,\zeta})(x).
\end{array}\end{equation}
Here $\vec{f}_{l}^{k,\zeta}$ is given as in \eqref{3.3} and
$\vec{f}_{l,j}^{k,\zeta}=(f_{1,j},\ldots,f_{l-1,j},
\Delta_{2^{-k}\zeta}f_{l,j},f_{l+1,j}^{k,\zeta},\ldots,
f_{m,j}^{k,\zeta})$ with $f_{i,j}^{k,\zeta}(x)=f_{i,j}(x+
2^{-k}\zeta)$ for all $l+1\leq i\leq m$. From the 
third inequality to the last one in \eqref{3.4}, we obtain
\begin{equation}\label{3.7}
\begin{array}{ll}
\displaystyle\Big\|\sum\limits_{l=1}^m\mathscr{M}_{\mathcal{R}}(\vec{f}_{l}^{k,\zeta})\Big\|_{p,q}&\lesssim\displaystyle\Big(\sum\limits_{k\in\mathbb{Z}}2^{ksq}\Big(\int_{\mathfrak{R}_d}\int_{\mathbb{R}^d}
\Big|\sum\limits_{l=1}^m\mathscr{M}_{\mathcal{R}}(\vec{f}_{l}^{k,\zeta})(x)\Big|^pdxd\zeta\Big)^{q/p}\Big)^{1/q}
\\
&\lesssim\displaystyle\sum\limits_{l=1}^m\Big(\sum\limits_{k\in\mathbb{Z}}2^{ksq}\Big\|\|\mathscr{M}_{\mathcal{R}}(\vec{f}_l^{k,\zeta})\|_{L^p(\mathbb{R}^d)}\Big\|_{L^p(\mathfrak{R}_d)}^q\Big)^{1/q}\\
&\lesssim\displaystyle\sum\limits_{l=1}^m\prod\limits_{i\neq l,1\leq i\leq m}\|f_i\|_{L^{p_i}(\mathbb{R}^d)}\|f_l\|_{\dot{B}_s^{p_l,q}(\mathbb{R}^d)}.
\end{array}
\end{equation}
One can also verify that
\begin{equation}\label{3.8}
|\mathscr{M}_{\mathcal{R}}(\vec{f}_{l,j}^{k,\zeta})-\mathscr{M}_{\mathcal{R}}(\vec{f}_{l}^{k,\zeta})|\leq\sum\limits_{\mu=1}^{l-1}\mathscr{M}_{\mathcal{R}}(\vec{I_{\mu,j}^{k,\zeta}})
+\sum\limits_{\nu=l+1}^{m}\mathscr{M}_{\mathcal{R}}(\vec{J_{\nu,j}^{k,\zeta}})+\mathscr{M}_{\mathcal{R}}(\vec{K_{i,j}^{k,\zeta}}),\end{equation}
where $$\vec{I_{\mu,j}^{k,\zeta}}=(f_1,\ldots,f_{\mu-1},f_{\mu,j}-f_{\mu},
f_{\mu+1,j},\ldots,f_{l-1,j},\Delta_{2^{-k}\zeta}f_{l,j},f_{l+1,j}^{k,\zeta},\ldots,f_{m,j}^{k,\zeta}),$$
$$\vec{J_{\nu,j}^{k,\zeta}}=(f_1,\ldots,f_{l-1},\Delta_{2^{k}\zeta}f_l,f_{l+1}^{k,\zeta},\ldots,f_{\nu-1}^{k,\zeta},f_{\nu,j}^{k,\zeta}-f_{\nu}^{k,\zeta},f_{\nu+1,j}^{k,\zeta},\ldots,f_{m,j}^{k,\zeta}),$$
$$\vec{K_{i,j}^{k,\zeta}}=(f_1,\ldots,f_{l-1},\Delta_{2^{-k}\zeta}(f_{l,j}-f_l),f_{l+1,j}^{k,\zeta},\ldots,f_{m,j}^{k,\zeta}).$$
By \eqref{3.7} and \eqref{3.8}, one can deduce that
$$\Big\|\sum\limits_{l=1}^m|\mathscr{M}_{\mathcal{R}}(\vec{f}_{l,j}^{k,\zeta})-\mathscr{M}_{\mathcal{R}}(\vec{f}_{l}^{k,\zeta})|\Big\|_{p,q}\rightarrow0\ \ {\rm as}\ j\rightarrow\infty.$$
Thus, we can extract a subsequence such that $\sum_{j=1}^\infty
\|\sum_{l=1}^m|\mathscr{M}_{\mathcal{R}}(\vec{f}_{l,j}^{k,\zeta})
-\mathscr{M}_{\mathcal{R}}(\vec{f}_{l}^{k,\zeta})|\|_{p,q}<\infty$. 

Let
$$H(x,k,\zeta)=\sum_{j=1}^\infty\Big|\sum\limits_{l=1}^m\mathscr{M}_{\mathcal{R}}(\vec{f}_{l,j}^{k,\zeta})(x)
-\mathscr{M}_{\mathcal{R}}(\vec{f}_{l}^{k,\zeta})(x)\Big|+2\sum\limits_{l=1}^m\mathscr{M}_{\mathcal{R}}(\vec{f}_{l}^{k,\zeta})(x).$$
It is easily to check that $\|H\|_{p,q}<\infty$. By 
\eqref{3.6}, we get
\begin{equation}\label{3.9}
|\Delta_{2^{-k}\zeta}(\mathscr{M}_{\mathcal{R}}(\vec{f}_j)-\mathscr{M}_{\mathcal{R}}(\vec{f}))(x)|\leq H(x,k,\zeta)\ \ {\rm for\ a.e.}\ (x,k,\zeta)\in\mathbb{R}^d\times\mathbb{Z}\times\mathfrak{R}_d.\end{equation}
Since $\|H\|_{p,q}<\infty$, we have $\int_{\mathbb{R}^d}
|H(x,k,\zeta)^pdx<\infty$ for a.e. $(k,\zeta)\in\mathbb{Z}
\times\mathfrak{R}_d$. By \eqref{3.9} and the dominated 
convergence theorem, for a.e. $(k,\zeta)\in\mathbb{Z}
\times\mathfrak{R}_d$, it holds that
\begin{equation}\label{3.10}
\lim\limits_{j\rightarrow\infty}\int_{\mathbb{R}^d}|\Delta_{2^{-k}\zeta}(\mathscr{M}_{\mathcal{R}}(\vec{f}_j)-\mathscr{M}_{\mathcal{R}}(\vec{f}))(x)|^pdx=0.\end{equation}
Using \eqref{3.9} and the fact $\|H\|_{p,q}<\infty$ again, 
we have
\begin{equation}\label{3.11}
\int_{\mathbb{R}^d}|\Delta_{2^{-k}\zeta}(\mathscr{M}_{\mathcal{R}}(\vec{f}_j)-\mathscr{M}_{\mathcal{R}}(\vec{f}))(x)|^pdx\leq\int_{\mathbb{R}^d}H(x,k,\zeta)^pdx, \quad\hbox{for a.e }(k,\zeta)\in\mathbb{Z}\times\mathfrak{R}_d \end{equation}
and
\begin{equation}\label{3.12}
\int_{\mathfrak{R}_d}\int_{\mathbb{R}^d}H(x,k,\zeta)^pdxd\zeta<\infty\quad\quad\hbox{for every}\  k \in\mathbb{Z}.\end{equation}
It follows from \eqref{3.10}-\eqref{3.12} and
the dominated convergence theorem that
\begin{equation}\label{3.13}
\lim\limits_{j\rightarrow\infty}\Big(\int_{\mathfrak{R}_d}\int_{\mathbb{R}^d}|\Delta_{2^{-k}\zeta}(\mathscr{M}_{\mathcal{R}}(\vec{f}_j)-\mathscr{M}_{\mathcal{R}}(\vec{f}))(x)|^pdxd\zeta\Big)^{1/p}=0\end{equation}
For every $k\in\mathbb{Z}$, by \eqref{3.9} and the fact
$\|H\|_{p,q}<\infty$ again, we have
\begin{equation}\label{3.14}
\Big(\int_{\mathfrak{R}_d}\int_{\mathbb{R}^d}|\Delta_{2^{-k}\zeta}(\mathscr{M}_{\mathcal{R}}(\vec{f}_j)-\mathscr{M}_{\mathcal{R}}(\vec{f}))(x)|^pdxd\zeta\Big)^{1/p}
\leq\Big(\int_{\mathfrak{R}_d}\int_{\mathbb{R}^d}H(x,k,\zeta)^pdxd\zeta\Big)^{1/p}\end{equation} 
and
\begin{equation}\label{3.15}
\Big(\sum_{k\in\mathbb{Z}}2^{ksq}\Big(\int_{\mathfrak{R}_d}\int_{\mathbb{R}^d}H(x,k,\zeta)^pdxd\zeta\Big)^{q/p}\Big)^{1/q}<\infty.\end{equation}
Using \eqref{3.14}-\eqref{3.15} and the dominated convergence 
theorem again, one may obtain
$$\aligned\|\Delta_{2^{-k}\zeta}&(\mathscr{M}_{\mathcal{R}}(\vec{f}_j)-\mathscr{M}_{\mathcal{R}}(\vec{f}))\|_{p,q}\\
&=\Big(\sum_{k\in\mathbb{Z}}2^{ksq}\Big(\int_{\mathfrak{R}_d}\int_{\mathbb{R}^d}|\Delta_{2^{-k}\zeta}(\mathscr{M}_{\mathcal{R}}(\vec{f}_j)-\mathscr{M}_{\mathcal{R}}(\vec{f}))(x)|^pdxd\zeta\Big)^{q/p}\Big)^{1/q}\rightarrow 0\ \ {\rm as}\ j\rightarrow\infty.\endaligned$$
By \eqref{3.2}, this yields that $\|\mathscr{M}_{\mathcal{R}}(\vec{f}_j)-\mathscr{M}_{\mathcal{R}}(\vec{f})\|_{\dot{B}_s^{p,q}(\mathbb{R}^d)}
\rightarrow 0$ as $j\rightarrow\infty$, which gives
a contradiction. The proof of Theorem \ref{thm2} is
finished. $\hfill\Box$

\quad\hspace{-20pt}{\bf Proof of Theorem 1.3.} Given an
operator $T$ acting on functions in $\mathbb{R}$, we
denote by $T^j$, $j=1,2,\ldots,d$, the operator defined
on functions in $\mathbb{R}^d$ by letting $T$ act on
the $j$-th variable while keeping the remaining variables
fixed, namely
$$
T^jf(x)=T(f(x_1,x_2,\ldots,x_{j-1},\cdot,x_{j+1},\ldots,x_d))(x_j)\
\text{ for }x\in\mathbb{R}^d.
$$
We also define the operator $\mathcal{T}$ by
$
\mathcal{T}f(x)=T^1\circ T^2\circ\ldots\circ T^df(x).
$ We need the following lemma.
\begin{lemma}\label{l32}
If $T$ is bounded on $L^p(\mathbb{R},\ell^q(L^r
(\mathfrak{R}_d)))$ for some $1<p,q,r<\infty$, then the
operator $\mathcal{T}$ is bounded on $L^p(\mathbb{R}^d,
\ell^q(L^r(\mathfrak{R}_d)))$.
\end{lemma}

\begin{proof} 
For all $j=1,\ldots,d$, we shall prove the following 
inequality
\begin{equation}\label{3.16}
\Big\|\Big(\sum\limits_{i\in\mathbb{Z}}
\|T^jf_{i,\zeta}\|_{L^r(\mathfrak{R}_d)}^q\Big)^{1/q}\Big\|_{L^p(\mathbb{R}^d)}
\leq\|T\| \Big\|\Big(\sum\limits_{i\in\mathbb{Z}}
\|f_{i,\zeta}\|_{L^r(\mathfrak{R}_d)}^q\Big)^{1/q}\Big\|_{L^p(\mathbb{R}^d)}.
\end{equation}
Here $\|T\|$ represents the operator norm of $T$ on 
$L^p(\mathbb{R},\ell^q(L^r(\mathfrak{R}_d)))$. We only 
prove \eqref{3.16} for $j=1$ and the other cases are
analogous. We may write
$$
\begin{array}{ll}
&\displaystyle\Big\|\Big(\sum\limits_{i\in\mathbb{Z}}
\|T^1f_{i,\zeta}\|_{L^r(\mathfrak{R}_d)}^q\Big)^{1/q}
\Big\|_{L^p(\mathbb{R}^d)}^p
\\
&=\displaystyle\int_{\mathbb{R}^d}\Big(
\sum\limits_{i\in\mathbb{Z}}\|T^1f_{i,\zeta}\|_{L^r(\mathfrak{R}_d)}^q
\Big)^{p/q}dx
\\
&=\displaystyle\int_{\mathbb{R}^{d-1}}\Big(
\int_{\mathbb{R}}\Big(\sum\limits_{i\in\mathbb{Z}}\Big(
\int_{\mathfrak{R}_d}|T(f_{i,\zeta}(\cdot,x_2,\ldots,x_d))(x_1)|^r
d\zeta\Big)^{q/r}\Big)^{p/q}dx_1\Big)dx_2\ldots dx_d
\\
&\leq\displaystyle\|T\|^p\int_{\mathbb{R}^{d-1}}\Big(
\int_{\mathbb{R}}\Big(\sum\limits_{i\in\mathbb{Z}}
\|f_{i,\zeta}(x_1,x_2,\ldots,x_d)\|_{L^r(\mathfrak{R}_d)}^q\Big)^{p/q}dx_1\Big)dx_2\ldots dx_d
\\
&=\displaystyle\|T\|^p\Big\|\Big(\sum\limits_{i\in\mathbb{Z}}
\|f_{i,\zeta}\|_{L^r(\mathfrak{R}_d)}^q\Big)^{1/q}
\Big\|_{L^p(\mathbb{R}^d)}^p,
\end{array}
$$
which leads to \eqref{3.16} for $j=1$. \eqref{3.16} together 
with the definition of $\mathcal{T}$ yields that
$$
\Big\|\Big(\sum\limits_{i\in\mathbb{Z}}
\|\mathcal{T}f_{i,\zeta}\|_{L^r(\mathfrak{R}_d)}^q\Big)^{1/q}
\Big\|_{L^p(\mathbb{R}^d)}\leq\|T\|^d \Big\|\Big(\sum\limits_{i\in\mathbb{Z}}
\|f_{i,\zeta}\|_{L^r(\mathfrak{R}_d)}^q\Big)^{1/q}\Big\|_{L^p(\mathbb{R}^d)}.
$$
This proves Lemma \ref{l32}.
\end{proof}

The following vector-valued inequalities of the one dimensional
uncentered Hardy-Littlewood maximal function will be very useful
in the proof of Theorem \ref{thm3}.
\begin{lemma}\label{l3.3}
(\cite{Ya}). For any $1<p,\,q,\,r<\infty$, it holds that
$$
\Big\|\Big(\sum\limits_{j\in\mathbb{Z}}
\|\mathcal{M}f_{j,\zeta}\|_{L^r(\mathfrak{R}_d)}^q\Big)^{1/q}
\Big\|_{L^p(\mathbb{R})}
\lesssim_{p,q,r}\Big\|\Big(\sum\limits_{j\in\mathbb{Z}}\|f_{j,\zeta}
\|_{L^r(\mathfrak{R}_d)}^q\Big)^{1/q}\Big\|_{L^p(\mathbb{R})}.
$$
\end{lemma}

Applying Lemmas \ref{l32} and \ref{l3.3}, we can get the following

\begin{lemma}\label{l3.4}
For any $1<p,\,q,\,r<\infty$, it holds that
$$
\Big\|\Big(\sum\limits_{j\in\mathbb{Z}}\|\mathcal{M}_{\mathcal{R}}
f_{j,\zeta}\|_{L^r(\mathfrak{R}_d)}^q\Big)^{1/q}\Big\|_{L^p(\mathbb{R}^d)}
\lesssim_{p,q,r}\Big\|\Big(\sum\limits_{j\in\mathbb{Z}}\|f_{j,\zeta}
\|_{L^r(\mathfrak{R}_d)}^q\Big)^{1/q}\Big\|_{L^p(\mathbb{R}^d)}.
$$
\end{lemma}
\begin{proof}
For $j=1,\ldots,d$, we define the operator $M^j$ by
$$
M^jf(x_1,x_2,\ldots,x_d)=\sup\limits_{a<x_j<b}\frac{1}{b-a}
\int_a^b|f(x_1,\ldots,x_{j-1},y,x_{j+1},\ldots,x_d)|dy.
$$
One can easily check that
\begin{equation}\label{3.17}
M^jf(x)=\mathcal{M}(f(x_1,x_2,\ldots,x_{j-1},\cdot,x_{j+1},\ldots,x_d)(x_j),\end{equation}
\begin{equation}\label{3.18}
\mathcal{M}_{\mathcal{R}}f(x)\leq M^1\circ M^2\circ\cdots\circ M^df(x).\end{equation}
Using \eqref{3.17}-\eqref{3.18} and Lemmas \ref{l32}-\ref{l3.3}, 
 for all $1<p,\,q,\,r<\infty$, we can get
$$\Big\|\Big(\sum\limits_{j\in\mathbb{Z}}\|\mathcal{M}_{\mathcal{R}}f_{j,\zeta}\|_{L^r(\mathfrak{R}_d)}^q\Big)^{1/q}\Big\|_{L^p(\mathbb{R}^d)}
\lesssim_{p,q,r}\Big\|\Big(\sum\limits_{j\in\mathbb{Z}}\|f_{j,\zeta}\|_{L^r(\mathfrak{R}_d)}^q\Big)^{1/q}\Big\|_{L^p(\mathbb{R}^d)}$$
Then Lemma \ref{l3.4} is proved.
\end{proof}
\quad\hspace{-20pt}{\it Proof of Theorem \ref{thm3}.}
Let $0<s<1$ and $1<p_1,\ldots,p_m,p,q<\infty$ with $1/p=
\sum_{i=1}^m1/p_i$. Let $\vec{f}=(f_1,\ldots,f_m)$ with
each $f_j\in F_s^{p_j,q}(\mathbb{R}^d)$. One can easily
check that \eqref{3.3} also holds. We get from \eqref{3.3} 
that
\begin{equation}\label{3.19}
\begin{array}{ll}
&|\Delta_{2^{-k}\zeta}(\mathscr{M}_{\mathcal{R}}(\vec{f}))(x)|\\
&\leq\displaystyle\sum\limits_{l=1}^m\mathcal{M}_{\mathcal{R}}(\Delta_{2^{-k}\zeta}f_l)(x)\prod\limits_{\mu=1}^{l-1}\mathcal{M}_{\mathcal{R}}f_\mu(x)\prod\limits_{\nu=l+1}^m\mathcal{M}_{\mathcal{R}}(f_\nu^{k,\zeta})(x)\\
&=\displaystyle\sum\limits_{l=1}^m\mathcal{M}_{\mathcal{R}}(\Delta_{2^{-k}\zeta}f_l)(x)\prod\limits_{\mu=1}^{l-1}\mathcal{M}_{\mathcal{R}}f_\mu(x)
\prod\limits_{\nu=l+1}^m\mathcal{M}_{\mathcal{R}}(\Delta_{2^{-k}\zeta}f_\nu+f_\nu)(x)\\
&\le\displaystyle\sum\limits_{\emptyset \ne\tau\subset \tau_m}
\prod\limits_{\mu\in\tau}\mathcal{M}_{\mathcal{R}}(\Delta_{2^{-k}\zeta}f_\mu)(x)
\prod\limits_{\nu\in\tau'}\mathcal{M}_{\mathcal{R}}f_\nu(x),
\end{array}\end{equation}
where $\tau_m=\{1,2,\dots,m\}$ and $\tau'=\tau_m\setminus
\tau$ for $\tau\subset\tau_m$.

Thus, Lemma \ref{l3.1} (i), \eqref{3.19} and the Minkowski
inequality yield that
\begin{equation}\label{3.20}
\begin{array}{ll}
&\|\mathscr{M}_{\mathcal{R}}(\vec{f})\|_{\dot{F}_s^{p,q}(\mathbb{R}^d)}
\\ %%%%%%%%%%%% 1 %%%%%%%%%%%%%%%%%%%%%
&\lesssim\displaystyle\Big\|\Big(\sum\limits_{k\in\mathbb{Z}}2^{ksq}
\Big(\int_{\mathfrak{R}_d}|\Delta_{2^{-k}\zeta}(\mathscr{M}_{\mathcal{R}}
(\vec{f}))|d\zeta\Big)^{{q}}
\Big)^{{1}/{q}}\Big\|_{L^p(\mathbb{R}^d)}
\\ %%%%%%%%%%%%%% 2 %%%%%%%%%%%%%%%%%%%%%%%%%%%%
&\lesssim\displaystyle\sum\limits_{\emptyset \ne\tau\subset \tau_m}
\Big\|\Big(\sum\limits_{k\in\mathbb{Z}}2^{ksq}
\Big(\int_{\mathfrak{R}_d}
\prod\limits_{\mu\in\tau}\mathcal{M}_{\mathcal{R}}(\Delta_{2^{-k}\zeta}f_\mu)
\prod\limits_{\nu\in\tau'}\mathcal{M}_{\mathcal{R}}f_\nu d\zeta\Big)^{{q}}
\Big)^{{1}/{q}}\Big\|_{L^p(\mathbb{R}^d)}.
\end{array}
\end{equation}
We shall prove the following estimate.
\begin{equation}\label{3.21}
\begin{array}{ll}
&\displaystyle\Big\|\Big(\sum\limits_{k\in\mathbb{Z}}2^{ksq}
\Big(\int_{\mathfrak{R}_d}
\prod\limits_{\mu\in\tau}\mathcal{M}_{\mathcal{R}}(\Delta_{2^{-k}\zeta}f_\mu)
\prod\limits_{\nu\in\tau'}\mathcal{M}_{\mathcal{R}}f_{\nu}d\zeta\Big)^{{q}}
\Big)^{{1}/{q}}\Big\|_{L^p(\mathbb{R}^d)}
\\
&=\displaystyle\Big\|\Big(\sum\limits_{k\in\mathbb{Z}}2^{ksq}
\Big(\int_{\mathfrak{R}_d}
\prod\limits_{\mu\in\tau}\mathcal{M}_{\mathcal{R}}(\Delta_{2^{-k}\zeta}f_\mu)
d\zeta\Big)^{{q}}\Big)^{{1}/{q}}
\prod\limits_{\nu\in\tau'}\mathcal{M}_{\mathcal{R}}f_\nu
\Big\|_{L^p(\mathbb{R}^d)}
\\
&\lesssim\displaystyle
\prod\limits_{\mu\in\tau}\|f_\mu\|_{{F}_{s}^{p_\mu,q}(\mathbb{R}^d)}
\prod\limits_{\nu\in\tau'}\|f_\nu\|_{{F}_{s}^{p_\nu,q}(\mathbb{R}^d)}
\end{array}
\end{equation}
Let $1/p_{\tau}=\sum_{\mu\in\tau}1/p_\mu$. Then, using
H\"older's inequality and the $L^p$ bounds for
$\mathcal{M}_{\mathcal{R}}$ we have
\begin{align*}
&\Big\|\Big(\sum\limits_{k\in\mathbb{Z}}2^{ksq}
\Big(\int_{\mathfrak{R}_d}
\prod\limits_{\mu\in\tau}\mathcal{M}_{\mathcal{R}}(\Delta_{2^{-k}\zeta}
f_\mu)d\zeta\Big)^{{q}}\Big)^{{1}/{q}}
\prod\limits_{\nu\in\tau'}\mathcal{M}_{\mathcal{R}}
f_\nu\Big\|_{L^p(\mathbb{R}^d)}
\\ %%%%%%%%%%%%%% 1 %%%%%%%%%%%%%%%%%%%%%%%%
&\le
\Big\|\Big(\sum\limits_{k\in\mathbb{Z}}2^{ksq}
\Big(\int_{\mathfrak{R}_d}
\prod\limits_{\mu\in\tau}\mathcal{M}_{\mathcal{R}}(\Delta_{2^{-k}\zeta}f_\mu)
d\zeta\Big)^{{q}}\Big)^{{1}/{q}}\Big\|_{L^{p_\tau}(\mathbb{R}^d)}
\Big\|
\prod\limits_{\nu\in\tau'}\mathcal{M}_{\mathcal{R}}
f_\nu\Big\|_{L^{p_{\tau'}}(\mathbb{R}^d)}
\\ %%%%%%%%%%%%%%%% 2 %%%%%%%%%%%%%%%%%%%%%%%%
&\le
\Big\|\Big(\sum\limits_{k\in\mathbb{Z}}2^{ksq}\prod\limits_{\mu\in\tau}
\Bigl\|\mathcal{M}_{\mathcal{R}}(\Delta_{2^{-k}\zeta}f_\mu)
\Bigr\|_{L^{p_\mu /p_\tau}({\mathfrak{R}_d})}^{{q}}\Big)^{{1}/{q}}
\Big\|_{L^{p_\tau}(\mathbb{R}^d)}\prod\limits_{\nu\in\tau'}\Big\|
\mathcal{M}_{\mathcal{R}}f_\nu\Big\|_{L^{p_\nu}(\mathbb{R}^d)}
\\&\le
\Big\|\prod\limits_{\mu\in\tau}\Big(\sum\limits_{k\in\mathbb{Z}}2^{ksq}
\Bigl\|\mathcal{M}_{\mathcal{R}}(\Delta_{2^{-k}\zeta}f_\mu)
\Bigr\|_{L^{p_\mu/p_\tau}({\mathfrak{R}_d})}^{{p_\mu q/p_\tau}}
\Big)^{{p_\tau}/{p_\mu q}}\Big\|_{L^{p_\tau}(\mathbb{R}^d)}
\prod\limits_{\nu\in\tau'}\Big\|f_\nu\Big\|_{L^{p_\nu}(\mathbb{R}^d)}\\&\le\prod\limits_{i\in\tau}\Big\|\Big(\sum\limits_{k\in\mathbb{Z}}2^{ksq}
\Bigl\|\mathcal{M}_{\mathcal{R}}(\Delta_{2^{-k}\zeta}f_\mu)
\Bigr\|_{L^{p_\mu/p_\tau}({\mathfrak{R}_d})}^{{p_\mu q/p_\tau}}
\Big)^{{p_\tau}/{p_\mu q}}\Big\|_{L^{p_\mu}(\mathbb{R}^d)}
\prod\limits_{\nu\in\tau'}\Big\|f_\nu\Big\|_{L^{p_\nu}(\mathbb{R}^d)}
\end{align*}
\begin{align*}
{}&=
\prod\limits_{\mu\in\tau}\Big\|
\Big(\sum\limits_{k\in\mathbb{Z}}2^{k(p_\tau s/p_\mu)(p_\mu q/p_\tau)}
\Bigl\|\mathcal{M}_{\mathcal{R}}(\Delta_{2^{-k}\zeta}f_\mu)
\Bigr\|_{L^{p_\mu/p_\tau}({\mathfrak{R}_d})}^{{p_\mu q/p_\tau}}
\Big)^{{p_\tau}/{p_\mu q}}\Big\|_{L^{p_\mu}(\mathbb{R}^d)}
\prod\limits_{\nu\in\tau'}\|f_\nu\|_{L^{p_\nu}(\mathbb{R}^d)}\\&\lesssim
\prod\limits_{\mu\in\tau}\Big\|
\Big(\sum\limits_{k\in\mathbb{Z}}2^{k(p_\tau s/p_\mu)(p_\mu q/p_\tau)}
\|\Delta_{2^{-k}\zeta}f_\mu
\|_{L^{p_\mu/p_\tau}({\mathfrak{R}_d})}^{{p_\mu q/p_\tau}}
\Big)^{{p_\tau}/{p_\mu q}}\Big\|_{L^{p_\mu}(\mathbb{R}^d)}
\prod\limits_{\nu\in\tau'}\|f_\nu\|_{L^{p_\nu}(\mathbb{R}^d)}\\&\lesssim
\prod\limits_{\mu\in\tau}
\|f_\mu\|_{\dot{F}_{p_\tau s/p_\mu}^{p_\mu,p_\mu q/p_\tau}(\mathbb{R}^d)}
\prod\limits_{\nu\in\tau'}\|f_i\|_{L^{p_\nu}(\mathbb{R}^d)}
\\ %%%%%%%%%%%%%%%%%%%% 8 %%%%%%%%%%%%%%%%%%%%%%
&\le
\prod\limits_{\mu\in\tau}
\|f_\mu\|_{{F}_{p_\tau s/p_\mu}^{p_\mu,p_\mu q/p_\tau}(\mathbb{R}^d)}
\prod\limits_{\nu\in\tau'}\|f_\nu\|_{L^{p_\nu}(\mathbb{R}^d)}
\\ %%%%%%%%%%%%%%%% 9 %%%%%%%%%%%%%%%%%%%%%%%%%%%
&\le
\prod\limits_{\mu\in\tau}
\|f_\mu\|_{{F}_{s}^{p_\mu,q}(\mathbb{R}^d)}
\prod\limits_{\nu\in\tau'}\|f_\nu\|_{{F}_{s}^{p_\nu,q}(\mathbb{R}^d)}.
\end{align*}
In the last estimate, we have used $p_\mu>p_\tau$ and the
inclusion property of Triebel-Lizorkin spaces. In the 6th
estimate, we used Lemma \ref{l3.4}. Thus, \eqref{3.21} 
holds. It follows from \eqref{3.20}-\eqref{3.21} that
\begin{equation}\label{3.22}
\|\mathscr{M}_{\mathcal{R}}(\vec{f})\|_{F_s^{p,q}(\mathbb{R}^d)}
\leq C\prod\limits_{i=1}^m\|f_i\|_{F_s^{p_i,q}(\mathbb{R}^d)}.\end{equation}
This completes the proof of the boundedness part.

Below we prove the continuity part. Let $f_{i,j}\rightarrow
f_i$ in $F_s^{p_i,q}(\mathbb{R}^d)$ as $j\rightarrow\infty$.
It is known that that $f_{i,j}\rightarrow f_i$ in
$\dot{F}_s^{p_i,q}(\mathbb{R}^d)$ and in $L^{p_i}
(\mathbb{R}^d)$ as $j\rightarrow\infty$. By \eqref{3.5}, it 
follows that $\mathscr{M}_{\mathcal{R}}(\vec{f_j})\rightarrow
\mathscr{M}_{\mathcal{R}}(\vec{f})$ in $L^p(\mathbb{R}^d)$
as $j\rightarrow\infty$. Therefore, it suffices to show that
$\mathscr{M}_{\mathcal{R}}(\vec{f_j})\rightarrow
\mathscr{M}_{\mathcal{R}}(\vec{f})$ in $\dot{F}_s^{p,q}
(\mathbb{R}^d)$ as $j\rightarrow\infty$. Again, we will prove 
this claim by contradiction. Without loss of generality we may
assume that, for every $j$, there exists $c>0$ such that
$$\|\mathscr{M}_{\mathcal{R}}(\vec{f_j})-\mathscr{M}_{\mathcal{R}}(\vec{f})\|_{\dot{F}_s^{p,q}(\mathbb{R}^d)}>c.$$
For a measurable function $g:\mathbb{R}^d\times\mathbb{Z}
\times\mathfrak{R}_d\rightarrow\mathbb{R}$, we define
$$
\|g\|_{E_{p,q}^{s}}:=\Big(\int_{\mathbb{R}^d}\Big(\sum\limits_{k\in\mathbb{Z}}
2^{ks q}\Big(\int_{\mathfrak{R}_d}|g(x,k,\zeta)|d\zeta\Big)^{q}\Big)^{p/q}
dx\Big)^{1/p}.
$$
By Lemma \ref{l3.1}, we see that if $1\leq r<\min(p,q)$, then
$\|f\|_{\dot{F}_s^{p,q}(\mathbb{R}^d)}\thicksim\|
\Delta_{2^{-k}\zeta}f\|_{E_{p,q}^{s}}$.
By \eqref{3.6} and \eqref{3.8}, we get
\begin{equation}\label{3.23}
\begin{array}{ll}
&|\Delta_{2^{-k}\zeta}(\mathscr{M}_{\mathcal{R}}(\vec{f_j})
-\mathscr{M}_{\mathcal{R}}(\vec{f}))|
\\
&\leq\displaystyle\sum\limits_{l=1}^m\Big(\sum\limits_{\mu=1}^{l-1}
\mathscr{M}_{\mathcal{R}}(\vec{I_{\mu,j}^{k,\zeta}})
+\sum\limits_{\nu=l+1}^{m}\mathscr{M}_{\mathcal{R}}(\vec{J_{\nu,j}^{k,\zeta}})
+\mathscr{M}_{\mathcal{R}}(\vec{K_{i,j}^{k,\zeta}})\Big)+2\sum\limits_{l=1}^m
\mathscr{M}_{\mathcal{R}}(\vec{f}_l^{k,\zeta}),
\end{array}
\end{equation}
where $\vec{f}_l^{k,\zeta}$ is given as in \eqref{3.19} and
$\vec{I_{\mu,j}^{k,\zeta}}$, $\vec{J_{\nu,j}^{k,\zeta}}$
and $\vec{K_{i,j}^{k,\zeta}}$ are given as in \eqref{3.8}.

Notice that
$$
\begin{array}{ll}
\mathscr{M}_{\mathcal{R}}(\vec{I_{\mu,j}^{k,\zeta}})
&\leq\displaystyle\prod\limits_{i=1}^{\mu-1}\mathcal{M}_{\mathcal{R}}f_i
\mathcal{M}_{\mathcal{R}}(f_{\mu,j}-f_\mu)
\prod\limits_{\ell=\mu+1}^{l-1}\mathcal{M}_{\mathcal{R}}f_{\ell,j}
\mathcal{M}_{\mathcal{R}}(\Delta_{2^{-k}\zeta}f_{l,j})\prod\limits_{w=l+1}^m
\mathcal{M}_{\mathcal{R}}f_{w,j}^{k,\zeta}\\&=\displaystyle\prod\limits_{i=1}^{\mu-1}\mathcal{M}_{\mathcal{R}}f_i
\mathcal{M}_{\mathcal{R}}(f_{\mu,j}-f_\mu)\prod\limits_{\ell=\mu+1}^{l-1}
\mathcal{M}_{\mathcal{R}}f_{\ell,j}\mathcal{M}_{\mathcal{R}}(\Delta_{2^{-k}\zeta}f_{l,j})\\
&\quad\times\displaystyle\prod\limits_{w=l+1}^m
\mathcal{M}_{\mathcal{R}}(\Delta_{2^{-k}\zeta}f_{w,j}+f_{w,j})
\end{array}
$$
$$
\begin{array}{ll}
&\le\displaystyle\prod\limits_{i=1}^{\mu-1}\mathcal{M}_{\mathcal{R}}f_i
\mathcal{M}_{\mathcal{R}}(f_{\mu,j}-f_\mu)
\prod\limits_{\ell=\mu+1}^{l-1}\mathcal{M}_{\mathcal{R}}f_{\ell,j}
\mathcal{M}_{\mathcal{R}}(\Delta_{2^{-k}\zeta}f_{l,j})
\\
&\quad\times\displaystyle\prod\limits_{w=l+1}^m(\mathcal{M}_{\mathcal{R}}
(\Delta_{2^{-k}\zeta}f_{w,j})+\mathcal{M}_{\mathcal{R}}f_{w,j}).
\end{array}
$$
This together with the arguments similar to those used in
deriving \eqref{3.21} yields that
\begin{equation}\label{3.24}
\|\mathscr{M}_{\mathcal{R}}(\vec{I_{\mu,j}^{k,\zeta}})\|_{E_{p,q}^{s}}\lesssim\|f_{\mu,j}-f_\mu\|_{F_s^{p_\mu,q}(\mathbb{R}^d)}
\prod\limits_{i=1}^{\mu-1}\|f_i\|_{F_s^{p_i,q}(\mathbb{R}^d)}\prod\limits_{w=\mu+1}^m\|f_{w,j}\|_{F_s^{p_w,q}(\mathbb{R}^d)}.
\end{equation}
Similarly, we can conclude that
\begin{equation}\label{3.25}
\|\mathscr{M}_{\mathcal{R}}(\vec{J_{\nu,j}^{k,\zeta}})\|_{E_{p,q}^{s}}\lesssim\|f_{\nu,j}-f_\nu\|_{F_s^{p_\nu,q}(\mathbb{R}^d)}
\prod\limits_{i=1}^{\nu-1}\|f_i\|_{F_s^{p_i,q}(\mathbb{R}^d)}\prod\limits_{w=\nu+1}^m\|f_{w,j}\|_{F_s^{p_w,q}(\mathbb{R}^d)};
\end{equation}
\begin{equation}\label{3.26}
\|\mathscr{M}_{\mathcal{R}}(\vec{K_{i,j}^{k,\zeta}})\|_{E_{p,q}^{s}}\lesssim\|f_{l,j}-f_l\|_{F_s^{p_l,q}(\mathbb{R}^d)}
\prod\limits_{\ell=1}^{l-1}\|f_\ell\|_{F_s^{p_\ell,q}(\mathbb{R}^d)}\prod\limits_{w=l+1}^m\|f_{w,j}\|_{F_s^{p_w,q}(\mathbb{R}^d)};
\end{equation}
\begin{equation}\label{3.27}
\|\mathscr{M}_{\mathcal{R}}(\vec{f}_l^{k,\zeta})\|_{E_{p,q}^{s}}\lesssim\prod\limits_{\ell=1}^{m}\|f_\ell\|_{F_s^{p_\ell,q}(\mathbb{R}^d)}.
\end{equation}
It follows from \eqref{3.24}-\eqref{3.27} that
$$\Big\|\sum\limits_{\mu=1}^{l-1}\mathscr{M}_{\mathcal{R}}(\vec{I_{\mu,j}^{k,\zeta}})
+\sum\limits_{\nu=l+1}^{m}\mathscr{M}_{\mathcal{R}}(\vec{J_{\nu,j}^{k,\zeta}})+\mathscr{M}_{\mathcal{R}}(\vec{K_{i,j}^{k,\zeta}})\Big\|_{E_{p,q}^{s}}\rightarrow0\ \ {\rm as}\ j\rightarrow\infty.$$
Therefore, one can extract a subsequence, we still denote
it by $j$, such that
\begin{equation}\label{3.28}
\sum_{j=1}^\infty\Big\|\sum\limits_{\mu=1}^{l-1}\mathscr{M}_{\mathcal{R}}(\vec{I_{\mu,j}^{k,\zeta}})
+\sum\limits_{\nu=l+1}^{m}\mathscr{M}_{\mathcal{R}}(\vec{J_{\nu,j}^{k,\zeta}})+\mathscr{M}_{\mathcal{R}}(\vec{K_{i,j}^{k,\zeta}})\Big\|_{E_{p,q}^{s}}<\infty.
\end{equation}
Let
$$\begin{array}{ll}
G(x,k,\zeta)&=\displaystyle\sum\limits_{l=1}^m\sum_{j=1}^\infty
\sum\limits_{\mu=1}^{l-1}\mathscr{M}_{\mathcal{R}}(\vec{I_{\mu,j}^{k,\zeta}})(x)
+\sum\limits_{\nu=l+1}^{m}\mathscr{M}_{\mathcal{R}}(\vec{J_{\nu,j}^{k,\zeta}})(x)\\
&\quad+\displaystyle\mathscr{M}_{\mathcal{R}}(\vec{K_{i,j}^{k,\zeta}})(x)+2\sum\limits_{l=1}^m\mathscr{M}_{\mathcal{R}}(\vec{f}_l^{k,\zeta})(x).
\end{array}$$
We get from \eqref{3.27} and \eqref{3.28} that 
$\|G\|_{E_{p,q}^{s}}<\infty$. Furthermore by \eqref{3.23}, 
one obtains that
\begin{equation}\label{3.29}
|\Delta_{2^{-k}\zeta}(\mathscr{M}_{\mathcal{R}}(\vec{f_j})- \mathscr{M}_{\mathcal{R}}(\vec{f}))(x)|\leq G(x,k,\zeta)
\ \ {\rm for\ every}\ (x,k,\zeta)\in\mathbb{R}^d\times\mathbb{Z}
\times\mathfrak{R}_d.\end{equation}
\eqref{3.29} together with the dominated convergence theorem
leads to
\begin{equation}\label{3.30}
\int_{\mathfrak{R}_d}|\Delta_{2^{-k}\zeta}(
\mathscr{M}_{\mathcal{R}}(\vec{f_j})- \mathscr{M}_{\mathcal{R}}(\vec{f}))(x)|d\zeta\rightarrow0\ \ {\rm as}\ j\rightarrow\infty\ \ {\rm for\ every}\ (x,k,\zeta)\in\mathbb{R}^d\times\mathbb{Z}
\times\mathfrak{R}_d.\end{equation}
Since it holds that $\|G\|_{E_{p,q}^{s}}<\infty$, we
immediately deduce that
\begin{equation}\label{3.31}
\Big(\sum_{k\in\mathbb{Z}}2^{ksq}\Big(\int_{\mathfrak{R}_d}G(x,k,\zeta)d\zeta\Big)^{q}\Big)^{1/q}<\infty, \quad\hbox{for a.e. }x\in\mathbb{R}^d.\end{equation}
Using \eqref{3.29}, we obtain
\begin{equation}\label{3.32}
\int_{\mathfrak{R}_d}|\Delta_{2^{-k}\zeta}(\mathscr{M}_{\mathcal{R}}(\vec{f_j})-\mathscr{M}_{\mathcal{R}}(\vec{f}))(x)|d\zeta
\leq\int_{\mathfrak{R}_d}G(x,k,\zeta)d\zeta, \quad\hbox{for a.e. } x\in\mathbb{R}^d \  \hbox{and }k\in\mathbb{Z}.\end{equation}
\eqref{3.30}-\eqref{3.32} and the dominated convergence 
theorem give
\begin{equation}\label{3.33}
\Big(\sum_{k\in\mathbb{Z}}2^{ksq}\Big(\int_{\mathfrak{R}_d}|\Delta_{2^{-k}\zeta}(\mathscr{M}_{\mathcal{R}}(\vec{f_j})-\mathscr{M}_{\mathcal{R}}(\vec{f}))(x)|d\zeta\Big)^{q}\Big)^{1/q}\rightarrow 0\ \ {\rm as}\ j\rightarrow\infty, \ \ \hbox{for a.e. } x\in\mathbb{R}^d\end{equation}
By \eqref{3.29} again, for a.e. $x\in\mathbb{R}^d$, it is 
true that
\begin{equation}\label{3.34}
\Big(\sum_{k\in\mathbb{Z}}2^{ksq}\Big(\int_{\mathfrak{R}_d}|\Delta_{2^{-k}\zeta}(\mathscr{M}_{\mathcal{R}}(\vec{f_j})-\mathscr{M}_{\mathcal{R}}(\vec{f}))(x)|d\zeta\Big)^{q}\Big)^{1/q}
\leq\Big(\sum_{k\in\mathbb{Z}}2^{ksq}\Big(\int_{\mathfrak{R}_d}G(x,k,\zeta)d\zeta\Big)^{q}\Big)^{1/q}\end{equation}
It follows from \eqref{3.33}-\eqref{3.34}, $\|G\|_{E_{p,q}^{s}}
<\infty$ and the dominated convergence theorem that
$$\lim\limits_{j\rightarrow\infty}\|\Delta_{2^{-k}\zeta}(\mathscr{M}_{\mathcal{R}}(\vec{f_j})-\mathscr{M}_{\mathcal{R}}(\vec{f}))\|_{E_{p,q}^{s}}=0,$$
which yields $\|\mathscr{M}_{\mathcal{R}}(\vec{f_j})-
\mathscr{M}_{\mathcal{R}}(\vec{f})\|_{\dot{F}_s^{p,q}
(\mathbb{R}^d)}\rightarrow 0$ as $j\rightarrow\infty$
and leads to a contradiction. $\hfill\Box$

\section{ Property of $p$-quasicontinuity}\label{S4}
\begin{proof}
We will divide the proof of Theorem \ref{thm5} into three
steps.

{\it Step 1: A weak type inequality for the Sobolev capacity.}
Let us begin with a capacity inequality that can be used in
studying the pointwise behaviour of Sobolev functions by the
standard methods (see \cite{FZ}). Let $\vec{f}=(f_1,\ldots,
f_m)$ with each $f_i\in W^{1,p_i}(\mathbb{R}^d)$ for
$1<p_i<\infty$. Let $1<p<\infty$ and $1/p=\sum_{i=1}^m1/p_i
$. For $\lambda>0$, we set
$$
O_\lambda=\{x\in\mathbb{R}^d;\mathscr{M}_{\mathcal{R}}(\vec{f})(x)>\lambda\}.
$$
Note that $O_\lambda$ is an open set. We get from Theorem
\ref{thm1} that
\begin{equation}\label{4.1}
\begin{array}{ll}
C_p(O_\lambda)^{1/p}&\leq\displaystyle\frac{1}{\lambda}\Big(
\int_{\mathbb{R}^d}(|\mathscr{M}_{\mathcal{R}}(\vec{f})(x)|^p
+|\nabla\mathscr{M}_{\mathcal{R}}(\vec{f})(x)|^p)dx\Big)^{1/p}
\\
&\leq\displaystyle\frac{1}{\lambda}\|\mathscr{M}_{\mathcal{R}}
(\vec{f})\|_{1,p}
\lesssim_{m,d,p_1,\ldots,p_m}\prod\limits_{i=1}^m
\frac{\|f_i\|_{1,p_i}}{\lambda}.
\end{array}\end{equation}

{\it Step 2: The continuity of $\mathscr{M}_{\mathcal{R}}(\vec{f})$.}
To prove the $p$-quasicontinuity of $\mathscr{M}_{\mathcal{R}}(\vec{f})$,
we first prove that $\mathscr{M}_{\mathcal{R}}(\vec{f})\in\mathcal{C}
(\mathbb{R}^d)$ if $\vec{f}=(f_1,\ldots,f_m)$ with each
$f_i\in\mathcal{C}_0^\infty(\mathbb{R}^d)$. We can write
$$
\mathscr{M}_{\mathcal{R}}(\vec{f})(x)
=\sup\limits_{\vec{r}\in\mathbb{R}_{+}^{2d}}
\prod_{i=1}^{m}\frac{1}{|E_{\vec{r}}(x)|}\int_{E_{\vec{r}}(x)}|f_i(y)|dy,
$$
where $\vec r=(r^-_1,\dots,r^-_{d}; r^+_1,\dots,r^+_{d})$
and $E_{\vec{r}}(x)=(x-r^-_1,x+r^+_1)\times\cdots\times
(x-r^-_{d},x+r^+_{d})$. For fixed $x,\,h\in\mathbb{R}^d$,
we have
$$
\begin{array}{ll}
&|\mathscr{M}_{\mathcal{R}}(\vec{f})(x+h)
-\mathscr{M}_{\mathcal{R}}(\vec{f})(x)|
\\
&\leq\displaystyle
\sum\limits_{i=1}^m\sup\limits_{\vec{r}\in\mathbb{R}_{+}^{2d}}
\frac{1}{|E_{\vec{r}}(x)|^{m}}\int_{E_{\vec{r}}(x)}|f_i(y+h)-f_i(y)|dy
\\
&\quad\times\displaystyle\Big(\prod\limits_{\mu=1}^{i-1}
\int_{E_{\vec{r}}(x)}|f_\mu(y)|dy\Big)\Big(\prod\limits_{\nu=i+1}^m
\int_{E_{\vec{r}}(x+h)}|f_\nu(y)|dy\Big).
\end{array}
$$
For fixed $\vec{r}\in\mathbb{R}_{+}^{2d}$ and $i=1,\ldots,m$,
by H\"{o}lder's inequality, we obtain
$$
\begin{array}{ll}
{}&\displaystyle\frac{1}{|E_{\vec{r}}(x)|^{m}}\int_{E_{\vec{r}}(x)}
|f_i(y+h)-f_i(y)|dy\displaystyle\Big(\prod\limits_{\mu=1}^{i-1}
\int_{E_{\vec{r}}(x)}|f_\mu(y)|dy\Big)
\Big(\prod\limits_{\nu=i+1}^m\int_{E_{\vec{r}}(x+h)}|f_\nu(y)|dy\Big)
\\
&\leq\displaystyle2|E_{\vec{r}}(x)|^{-1/p}\prod\limits_{i=1}^m
\|f_i\|_{L^{p_i}(\mathbb{R}^d)}.
\end{array}
$$
It follows that given $\epsilon>0$, there exists a constant
$0<\delta_\epsilon<+\infty$ such that
$$
\frac{1}{|E_{\vec{r}}(x)|^{m}}\int_{E_{\vec{r}}(x)}|f_i(y+h)-f_i(y)|dy
\Big(\prod\limits_{\mu=1}^{i-1}\int_{E_{\vec{r}}(x)}|f_\mu(y)|dy\Big)
\Big(\prod\limits_{\nu=i+1}^m\int_{E_{\vec{r}}(x+h)}|f_\nu(y)|dy\Big)<\epsilon,
$$
when $|E_{\vec{r}}(x)|>\delta_\epsilon$. On the other hand, for any
$x,\,h\in\mathbb{R}^d$ and $\vec{r}\in
\mathbb{R}_{+}^{2d}$ with $|E_{\vec{r}}(x)|\leq
\delta_\epsilon$, by the mean value theorem for differentials, we have
$$
\frac{1}{|E_{\vec{r}}(x)|}\int_{E_{\vec{r}}(x)}|f_i(y+h)-f_i(y)|dy
\leq C(f_i)|h|
$$
and there exists $M_i>0$ such that $|f_i(x)|\leq M_i$
for all $x\in\mathbb{R}^d$ and $i=1,\ldots,m$. Then
we have
$$
\begin{array}{ll}
&\displaystyle\frac{1}{|E_{\vec{r}}(x)|^{m}}\int_{E_{\vec{r}}(x)}
|f_i(y+h)-f_i(y)|dy\displaystyle\Big(\prod\limits_{\mu=1}^{i-1}
\int_{E_{\vec{r}}(x)}|f_\mu(y)|dy\Big)
\Big(\prod\limits_{\nu=i+1}^m\int_{E_{\vec{r}}(x+h)}|f_\nu(y)|dy\Big)
\\
&\leq\displaystyle C(f_i)\prod\limits_{\mu\neq i,1\leq\mu\leq m}M_\mu|h|.
\end{array}
$$
Therefore, for the above $\epsilon>0$ and fixed
$x\in\mathbb{R}^d$, there exists $\gamma=\gamma
(\epsilon)>0$, if $|h|<\gamma$, then
$$
|\mathscr{M}_{\mathcal{R}}(\vec{f})(x+h)-\mathscr{M}_{\mathcal{R}}(\vec{f})(x)|
\leq C(\vec{f})\epsilon.
$$
Thus, it holds that $\mathscr{M}_{\mathcal{R}}(\vec{f})
\in\mathcal{C}(\mathbb{R}^d)$.

{\it Step 3: The $p$-quasicontinuity of $\mathscr{M}_{
\mathcal{R}}(\vec{f})$.} Suppose that $f_i\in W^{1,p_i}
(\mathbb{R}^d)$, we can choose a sequence of functions
$\{f_{i,k}\}_{k\geq1}\subset\mathcal{C}_0^\infty
(\mathbb{R}^d)$ such that $f_{i,k}\rightarrow f_i$ in
$W^{1,p_i}(\mathbb{R}^d)$. This yields that there
exists a large $K_0\in\mathbb{N}$ such that
\begin{equation}\label{4.2}
\|f_{i,k}-f_i\|_{1,p_i}
\leq 2^{-2k}, \ \ \forall k\geq K_0\ {\rm and}\ 1\leq i\leq m.
\end{equation}
Fix $k\geq K_0$. Let $\vec{f}_k=(f_{1,k},\ldots,f_{m,k})$
and
$$
E_k=\{x\in\mathbb{R}^d:|\mathscr{M}_{\mathcal{R}}(\vec{f}_k)(x)
-\mathscr{M}_{\mathcal{R}}(\vec{f})(x)|>2^{-k}\}.
$$
By \eqref{2.6}, we have
\begin{equation}\label{4.3}
|\mathscr{M}_{\mathcal{R}}(\vec{f}_k)(x)-\mathscr{M}_{\mathcal{R}}(\vec{f})(x)|
\leq\sum\limits_{l=1}^m\mathscr{M}_{\mathcal{R}}(\vec{F}_k^l)(x),
\end{equation}
where $\vec{F}_k^l(x)=(f_1,\ldots,f_{l-1},f_{l,k}-f_l,
f_{l+1,k},\ldots,f_{m,k})$. Then, by \eqref{4.1}-\eqref{4.3}, 
we have
\begin{equation}\label{4.4}
\begin{array}{ll}
(C_p(E_k))^{1/p}&\displaystyle\lesssim_{m,d,p_1,\ldots,p_m}2^k
\sum\limits_{l=1}^m\prod\limits_{\mu=1}^{l-1}\|f_\mu\|_{1,p_\mu}
\|f_{l,k}-f_l\|_{1,p_l}\prod\limits_{\nu=l+1}^m\|f_{\nu,k}\|_{1,p_\nu}
\\
&\lesssim_{m,d,p_1,\ldots,p_m}2^{-k}.
\end{array}\end{equation}
Let $G_k=\bigcup_{i=k}^\infty E_i$ with $k\geq K_0$. Then
by subadditivity and \eqref{4.4}, it holds that
$$
C_p(G_k)\leq\sum\limits_{i=k}^\infty C_p(E_i)\lesssim_{m,d,p_1,\ldots,p_m}
\sum\limits_{i=k}^\infty 2^{-ip}\lesssim_{m,d,p_1,\ldots,p_m}
\frac{2^{(1-k)p}}{2^p-1},\ \ \forall k\geq K_0,
$$
which leads to $\lim_{k\rightarrow\infty}C_p(G_k)=0$. On
the other hand, for $x\in\mathbb{R}^d\setminus G_k$,
\begin{equation}\label{4.5}
|\mathscr{M}_{\mathcal{R}}(\vec{f}_{k})(x)
-\mathscr{M}_{\mathcal{R}}(\vec{f})(x)|\leq2^{-k}\ \ \forall k\geq K_0.
\end{equation}
This implies that $\{\mathscr{M}_{\mathcal{R}}(\vec{f}_k)\}$
converges to $\mathscr{M}_{\mathcal{R}}(\vec{f})$ uniformly
in $\mathbb{R}^d\setminus G_k$. By Step 2, we see that
$\mathscr{M}_{\mathcal{R}}(\vec{f}_k)\in\mathcal{C}
(\mathbb{R}^d)$. It follows that $\mathscr{M}_{\mathcal{R}}
(\vec{f})$ is continuous in $\mathbb{R}^d\setminus G_k$.
We notice that $\mathscr{M}_{\mathcal{R}}(\vec{f}_{K_0})(x)
<\infty$ for all $x\in\mathbb{R}^d$. This together with
\eqref{4.5} implies that $\mathscr{M}_{\mathcal{R}}(\vec{f})$
is finite in $\mathbb{R}^d\setminus G_k$. Hence,
$\mathscr{M}_{\mathcal{R}}(\vec{f})$ is
$q$-quasicontinuous.
\end{proof}
\section{Approximate differentiability of $\mathscr{M}_{\mathcal{R}}$} \label{S5}

This section is devoted to proving Theorem \ref{thm6}. Let
us recall some definitions and present some useful lemmas.

Let $f$ be a real-valued function defined on a set
$E\subset\mathbb{R}^d$. We say that $f$ is approximately
differentiable at $x_0\in E$ if there is a vector
$L=(L_1,L_2,\ldots,L_d)\in\mathbb{R}^d$ such that for
any $\epsilon>0$ the set
$$
A_\epsilon=\Big\{x\in\mathbb{R}^d:\frac{|f(x)-f(x_0)-L(x-x_0)|}{|x-x_0|}
<\epsilon\Big\}
$$
has $x_0$ as a density point. If this is the case, then
$x_0$ is a density point of $E$ and $L$ is uniquely
determined. The vector $L$ is  called the approximate
differential of $f$ at $x_0$ and is denoted by
$\nabla f(x_0)$. Note that every function $f\in W^{1,1}
(\mathbb{R}^d)$ is approximately differentiable a.e.
It was pointed out in \cite{HM} that $Mf$ is approximately
differentiable a.e. under the assumption that
$f\in W^{1,1}(\mathbb{R}^d)$. However, it is unknown
that whether $f\in W^{1,1}(\mathbb{R}^d)$ implies the
weak differentiability of $Mf$ when $d\geq2$. The
relationship between approximate differentiability and
weak differentiability is still not clear.

To prove Theorem \ref{thm6}, we need the following lemma,
which provides several characterizations of a.e. approximate
differentiability of a function.
%%%%%%%%%%% Lemma 6.1 %%%%%%%%%%%%%%%%%%%%%%%%%%%%%%%%%%%%%%%%%%%%
\begin{lemma}
{\rm (\cite{Wh})} \label{l5.1}Let $f:E\rightarrow\mathbb{R}$
be measurable, $E\subset\mathbb{R}^d$. Then the following
conditions are equivalent:
\begin{enumerate}
\item[{\rm (i)}] $f$ is approximately differentiable a.e.

\item[{\rm (ii)}] For any $\epsilon>0$, there is a closed set
$F\subset E$ and a locally Lipschitz function $g:\mathbb{R}^d
\rightarrow\mathbb{R}$ such that $f=g$ on $x\in F$ and
$|E\setminus F|<\epsilon$.

\item[{\rm (iii)}] For any $\epsilon>0$, there is a closed
set $F\subset E$ and a function $g\in\mathcal{C}^1
(\mathbb{R}^d)$ such that $f=g$ on $x\in F$ and
$|E\setminus F|<\epsilon$.\end{enumerate}
\end{lemma}
%%%%%%%%%%%  end Lemma 6.1 %%%%%%%%%%%%%%%%%%%%%%%%%%%%%%%%%%%%%%%%%%%%
%%%%%%%%%%% Lemma 6.2 %%%%%%%%%%%%%%%%%%%%%%%%%%%%%%%%%%%%%%%%%%%%
\begin{lemma}\label{l5.2}
Let $\vec{f}=(f_1,\ldots,f_m)$ with each $f_j\in L^1(\mathbb{R}^d)$.
Let $\vec{\varepsilon}=(\varepsilon_1,\ldots,\varepsilon_d)$ with
each $\varepsilon_i>0$. The truncated multilinear strong maximal
operator $\mathscr{M}_{\mathcal{R}}^{\vec{\varepsilon}}$ is defined
by
$$\mathscr{M}_{\mathcal{R}}^{\vec{\varepsilon}}(\vec{f})(x)
=\sup\limits_{(r^-_1,\dots,r^-_{d};r^+_1,\dots,r^+_{d})\in\mathbb{R}_{+}^{2d}
\atop
r_i^{+}+r_i^{-}\geq\varepsilon_i,\,i=1,2,\ldots,d}
\prod_{i=1}^{m}\frac{1}{|E_{\vec{r}}(x)|}\int_{E_{\vec{r}}(x)}|f_i(y)|dy,$$
where $x=(x_1,\ldots,x_d)$, $\vec r=(r^-_1,\dots,r^-_{d};
r^+_1,\dots,r^+_{d})$ and $E_{\vec{r}}(x)=(x_1-r^-_1,x_1+r^+_1)
\times\cdots\times(x_d-r^-_{d},x_d+r^+_{d})$. Then
$\mathscr{M}_{\mathcal{R}}^{\vec{\varepsilon}}(\vec{f})$
is Lipschitz continuous for every $\vec{\varepsilon}\in
\mathbb{R}_{+}^{d}$.
\end{lemma}
\begin{proof} Fix $\vec{\varepsilon}=(\varepsilon_1,\ldots,
\varepsilon_d)\in\mathbb{R}_{+}^d$. We set $\varepsilon_0=
\min_{1\leq i\leq d}\varepsilon_i$. Fix $\vec r=(r^-_1,
\dots,r^-_{d}; r^+_1,\dots,r^+_{d})\in\mathbb{R}_{+}^{2d}$
with $r_i^{+}+r_i^{-}\geq\varepsilon_i$ for $1\leq i\leq d$.
It is obvious that $r_i^{+}+r_i^{-}\geq\varepsilon_0$ for
all $1\leq i\leq d$.

Notice that for any $r\geq a\geq0$, $b\geq0$ and $\delta\geq1$,
it is true that
\begin{equation}\label{5.1}
\Big(\frac{r}{r+b}\Big)^\delta\geq\Big(\frac{a}{a+b}\Big)^\delta\geq1
-\delta\frac{b}{a}.
\end{equation}
Let $x=(x_1,\ldots,x_d)\in\mathbb{R}^d$ and $y=(y_1,\ldots,
y_d)\in\mathbb{R}^d$.
We note that
\begin{equation}\label{5.2}
E_{\vec{r}}(x)\subset E_{\vec{r}'}(y),
\end{equation}
where $\vec{r}'=(r_1^{-}+|y_1-x_1|,\ldots,r_d^{-}+|y_d-x_d|;
r_1^{+}+|y_1-x_1|,\ldots,r_d^{+}+|y_d-x_d|)$. \eqref{5.2} 
gives that
\begin{equation}\label{5.3}
\Big(\frac{|E_{\vec{r}}(x)|}{|E_{\vec{r}'}(y)|}\Big)^m
=\Big(\prod\limits_{j=1}^d\frac{r_j^{-}+r_j^{+}}{r_j^{-}+r_j^{+}+|y_j-x_j|}
\Big)^m
\geq\Big(\frac{\varepsilon_0}{\varepsilon_0+|x-y|}\Big)^{md}
\geq 1-\frac{md}{\varepsilon_0}|x-y|.
\end{equation}
We get from \eqref{5.2} and \eqref{5.3} that
\begin{equation}\label{5.4}
\begin{array}{ll}
\mathscr{M}_{\mathcal{R}}^{\vec{\varepsilon}}(\vec{f})(y)&\geq\displaystyle
\prod_{i=1}^{m}\frac{1}{|E_{\vec{r}'}(y)|}\int_{E_{\vec{r}}(y)}|f_i(z)|dz
\\
&\geq\displaystyle\prod_{i=1}^{m}\frac{|E_{\vec{r}}(x)|}{|E_{\vec{r}'}(y)|}
\frac{1}{|E_{\vec{r}}(x)|}\int_{E_{\vec{r}}(x)}|f_i(z)|dz
\\
&\geq\displaystyle\Big(1-\frac{md}{\varepsilon_0}|x-y|\Big)\prod_{i=1}^{m}
\frac{1}{|E_{\vec{r}}(x)|}\int_{E_{\vec{r}}(x)}|f_i(z)|dz.
\end{array}
\end{equation}
Taking the supremum over $\vec r=(r^-_1,\dots,r^-_{d};
r^+_1,\dots,r^+_{d})\in\mathbb{R}_{+}^{2d}$ with
$r_i^{+}+r_i^{-}\geq\varepsilon_i$ for $1\leq i\leq d$,
we get from \eqref{5.4} that
$$
\mathscr{M}_{\mathcal{R}}^{\vec{\varepsilon}}(\vec{f})(y)
\geq\Big(1-\frac{md}{\varepsilon_0}|x-y|\Big)
\mathscr{M}_{\mathcal{R}}^{\vec{\varepsilon}}(\vec{f})(x).
$$
It follows that
\begin{equation}\label{5.5}
\mathscr{M}_{\mathcal{R}}^{\vec{\varepsilon}}(\vec{f})(x)
-\mathscr{M}_{\mathcal{R}}^{\vec{\varepsilon}}(\vec{f})(y)
\leq\frac{md}{\varepsilon_0}|x-y|\mathscr{M}_{\mathcal{R}}^{\vec{\varepsilon}}
(\vec{f})(x).
\end{equation}
Similarly, we can get
\begin{equation}\label{5.6}
\mathscr{M}_{\mathcal{R}}^{\vec{\varepsilon}}(\vec{f})(y)
-\mathscr{M}_{\mathcal{R}}^{\vec{\varepsilon}}(\vec{f})(x)
\leq\frac{md}{\varepsilon_0}|x-y|\mathscr{M}_{\mathcal{R}}^{\vec{\varepsilon}}
(\vec{f})(y).
\end{equation}
Thus, \eqref{5.5} and \eqref{5.6} imply that
$$
|\mathscr{M}_{\mathcal{R}}^{\vec{\varepsilon}}(\vec{f})(x)
-\mathscr{M}_{\mathcal{R}}^{\vec{\varepsilon}}(\vec{f})(y)|
\leq \displaystyle\frac{md}{\varepsilon_0}
|x-y|(\mathscr{M}_{\mathcal{R}}^{\vec{\varepsilon}}(\vec{f})(x)
+\mathscr{M}_{\mathcal{R}}^{\vec{\varepsilon}}(\vec{f})(y))
\leq\displaystyle\frac{2md}{\varepsilon_0^{md+1}}
\prod\limits_{j=1}^m\|f_j\|_{L^1(\mathbb{R}^d)}|x-y|.
$$
This proves Lemma \ref{l5.2}.
\end{proof}

\quad\hspace{-20pt}{\it Proof of Theorem \ref{thm6}.}
Let $Z_j$ be the set of all Lebesgue points of $f_j$ and
$u_{x,\vec{f}}(r)$ defined as in Section 2. We set
$E_0=\mathbb{R}^d\setminus(\bigcap_{j=1}^mZ_j)$. Let
$x\in\bigcap_{j=1}^mZ_j$ such that $\mathscr{M}_{\mathcal{R}}
(\vec{f})(x)>u_{x,\vec{f}}(\vec{0})$ with
$\vec{0}\in\mathbb{R}^{2d}$. Since $f_j\in L^1(\mathbb{R}^d)$
and $\mathscr{M}_{\mathcal{R}}(\vec{f})(x)>0$, there
exists a sequence $\{\vec{r}_k\}_{k\geq1}$ with
$\vec{r}_k=(r_{1,k}^{-},\ldots,r_{d,k}^{-};r_{1,k}^{+},
\ldots,r_{d,k}^{+})\in\overline{\mathbb{R}}_{+}^{2d}$,
all $r_{i,k}^{-}+r_{i,k}^{+}$ are bounded such that
$$
\lim\limits_{k\rightarrow\infty}u_{x,\vec{f}}(\vec{r}_k)
=\mathscr{M}_{\mathcal{R}}(\vec{f})(x).
$$
Hence there exists a subsequence $\{\vec{r}_k'\}_{k\geq1}
\subset\{\vec{r}_k\}_{k\geq1}$ and $\vec{r}=(r_{1}^{-},\ldots,
r_d^{-};r_1^{+},\ldots,r_d^{+})\in\overline
{\mathbb{R}}_{+}^{2d}$ with $r_{i}^{-}+r_{i}^{+}>0$ for
all $1\leq i\leq d$ such that $\lim_{k\rightarrow\infty}
\vec{r}_k'=\vec{r}$. It follows that
$\mathscr{M}_{\mathcal{R}}(\vec{f})(x)=u_{x,\vec{f}}(\vec{r}).$
This, of course, yields that
$$\mathbb{R}^d=E_0\cup\{x\in\mathbb{R}^d:\,\mathscr{M}_{\mathcal{R}}
(\vec{f})(x)=u_{x,\vec{f}}(\vec{0})\}\cup E,
$$
where $E=\bigcup_{k_1=1}^\infty\ldots\bigcup_{k_d=1}^\infty
E_{k_1,\ldots,k_d}$ and $E_{k_1,\ldots,k_d}=\{x\in\mathbb{R}^d:\
\mathscr{M}_{\mathcal{R}}(\vec{f})(x)=
\mathscr{M}_{\mathcal{R}}^{1/k_1,\ldots,1/k_d}(\vec{f})(x)\}$.
By Lemma \ref{l5.1}, $\prod_{j=1}^m|f_j|$ is approximately
differentiable a.e. Then $\mathscr{M}_{\mathcal{R}}(\vec{f})$
is approximately differentiable a.e. in the set
$\{x\in\mathbb{R}^2:\,\mathscr{M}_{\mathcal{R}}(\vec{f})(x)
=u_{x,\vec{f}}(\vec{0})\}$. By Lemma \ref{l5.2} we have that
$\mathscr{M}_{\mathcal{R}}^{1/k_1,\ldots,1/k_d}(\vec{f})$
is Lipschitz continuous for any $k_i\geq1$ and $1\leq i\leq d$.
Then, for any $k_i\geq1$ and $1\leq i\leq d$, the function
$\mathscr{M}_{\mathcal{R}}^{1/k_1,\ldots,1/k_d}(\vec{f})$
is approximately differentiable a.e.. It follows that
$\mathscr{M}_{\mathcal{R}}(\vec{f})\chi_{E}$ is approximately
differentiable a.e. Note that $|E_0|=0$. Therefore,
$\mathscr{M}_{\mathcal{R}}(\vec{f})$ is approximately
differentiable a.e. This completes the proof of Theorem \ref{thm6}. $\hfill\Box$

%%%%%%% end Section 6 Approximate differentiability %%%%%%%%%%%%%%%%%%%%%%
%%%%%%%%%%%%%%%%%%%%%%%%%%%%%%%%%%%%%%%%%%%%%%%%%%%%%%%%%%%%%%%%%%%%%%
%%%%%%%%% Section 7 discrete strong maximal functions %%%%%%%%%%%%%%
%%%%%%%%%%%%%%%%%%%%%%%%%%%%%%%%%%%%%%%%%%%%%%%%%%%%%%%%%%%%%%%%%%%%%%%
\section{ Properties of discrete strong maximal
functions}\label{S6}
This section is devoted to proving Theorem \ref{thm7}. For
$a\in\mathbb{R}$ and $r>0$, we define
$$g(a;r)=|\{k\in\mathbb{Z};|k-a|<r\}|.$$
If $a\in\mathbb{Z}$, then $g(a;r)\geq\chi_{(0,1]}(r)
+(2[r-1]+1)\chi_{(1,\infty)}$, where $[x]=\max\{k\in
\mathbb{Z};k\leq x\}$. If $a\in\mathbb{R}\backslash
\mathbb{Z}$, then there exists an integer $n_0\in
\mathbb{Z}$ such that $|n_0-a|\leq1/2$ and
$$\{k\in\mathbb{Z};|k-n_0|<r-1/2\}\subset\{k\in\mathbb{Z};|k-a|<r\}.$$
It follows that $g(a;r)\geq\chi_{(\frac{1}{2},\frac{3}{2}]}
+(2[r-3/2]+1)\chi_{(\frac{3}{2},\infty)}(r)$ for $r>1/2$.
Specially, if there exists an integer $n_0$ such that
$|n_0-a|<r$, then
\begin{equation}\label{6.1}
g(a;r)\geq F(r):=\chi_{(0,\frac{3}{2}]}(r)+(2[r-3/2]+1)\chi_{(\frac{3}{2},
\infty)}(r)\ \ \forall r>0\ {\rm and}\ a\in\mathbb{R}.
\end{equation}
For $\vec{r}=(r_1,\ldots,r_d)\in\mathbb{R}_{+}^d$ and
$\vec{x}=(x_1,\ldots,x_d)\in\mathbb{R}^d$, it is easy
to see that $N(R_{\vec{r}}(\vec{x}))=\prod_{i=1}^d
g(x_i;r_i)$. Furthermore, if there exists $\vec{n}\in
R_{\vec{r}}(\vec{x})\cap\mathbb{Z}^d$, then by
(\ref{6.1}), it holds that
\begin{equation}\label{6.2}
N(R_{\vec{r}}(\vec{x}))\geq\prod\limits_{i=1}^dF(r_i).
\end{equation}

We now divide the proof of Theorem \ref{thm7} into two
parts.

%%%%%%%%%%%%%%%%%%%%%%%%%%%%%%%%%%%%%%5
\subsection{\bf The boundedness part}
Without loss of generality we may assume all $f_j\geq0$
since $\|\nabla|f|\|_{\ell^1(\mathbb{Z}^d)}\leq
\|\nabla f\|_{\ell^1(\mathbb{Z}^d)}$. For all $1\leq l\leq d$,
it suffices to show that
\begin{equation}\label{6.3}
\|D_l\mathbb{M}_{\mathcal{R}}(\vec{f})\|_{\ell^1(\mathbb{Z}^d)}
\lesssim_{d,m}\sum\limits_{i=1}^m\|D_lf_i\|_{\ell^1(\mathbb{Z}^d)}
\prod\limits_{j\neq i,1\leq j\leq m}\|f_j\|_{\ell^1(\mathbb{Z}^d)}
\end{equation}
We only prove \eqref{6.3} for $l=d$, since the other
cases are analogous. In what follows, we set $\vec{n}
=(n',n_d)\in\mathbb{Z}^d$ with $n'=(n_1,\ldots,n_{d-1})
\in\mathbb{Z}^{d-1}$. For each $n'\in\mathbb{Z}^{d-1}$, 
let
$$
X_{n'}^{+}=\{n_d\in\mathbb{Z}:\mathbb{M}_{\mathcal{R}}(\vec{f})(n',n_d+1)
\leq\mathbb{M}_{\mathcal{R}}(\vec{f})(n',n_d)\},
$$
$$
X_{n'}^{-}=\{n_d\in\mathbb{Z}:\mathbb{M}_{\mathcal{R}}(\vec{f})(n',n_d+1)
>\mathbb{M}_{\mathcal{R}}(\vec{f})(n',n_d)\}.
$$
Then we can write
$$\begin{array}{ll}
\|D_l\mathbb{M}_{\mathcal{R}}(\vec{f})\|_{\ell^1(\mathbb{Z}^d)}
&=\displaystyle\sum\limits_{n'\in\mathbb{Z}^{d-1}}
\sum\limits_{n_d\in X_{n'}^{+}}(\mathbb{M}_{\mathcal{R}}(\vec{f})(n',n_d)
-\mathbb{M}_{\mathcal{R}}(\vec{f})(n',n_d+1))
\\
&\quad+\displaystyle\sum\limits_{n'\in\mathbb{Z}^{d-1}}
\sum\limits_{n_d\in X_{n'}^{-}}(\mathbb{M}_{\mathcal{R}}(\vec{f})(n',n_d+1)
-\mathbb{M}_{\mathcal{R}}(\vec{f})(n',n_d)).
\end{array}
$$
Therefore, to prove \eqref{6.3} with $l=d$, it suffices 
to show that
\begin{equation}\label{6.4}
\sum\limits_{n'\in\mathbb{Z}^{d-1}}\sum\limits_{n_d\in X_{n'}^{+}}
(\mathbb{M}_{\mathcal{R}}(\vec{f})(n',n_d)-\mathbb{M}_{\mathcal{R}}
(\vec{f})(n',n_d+1))
\lesssim_{d}\sum\limits_{i=1}^m\|D_lf_i\|_{\ell^1(\mathbb{Z}^d)}
\prod\limits_{j\neq i,1\leq j\leq m}\|f_j\|_{\ell^1(\mathbb{Z}^d)};
\end{equation}
\begin{equation}\label{6.5}
\sum\limits_{n'\in\mathbb{Z}^{d-1}}\sum\limits_{n_d\in X_{n'}^{-}}
(\mathbb{M}_{\mathcal{R}}(\vec{f})(n',n_d+1)-\mathbb{M}_{\mathcal{R}}
(\vec{f})(n',n_d))
\lesssim_{d}\sum\limits_{i=1}^m\|D_lf_i\|_{\ell^1(\mathbb{Z}^d)}
\prod\limits_{j\neq i,1\leq j\leq m}\|f_j\|_{\ell^1(\mathbb{Z}^d)}.
\end{equation}
We only prove \eqref{6.4}, since \eqref{6.5} is analogous. 
For $\vec{r}\in\mathbb{R}_{+}^d$, define $A_{\vec{r}}
(\vec{f}):\mathbb{R}^d\rightarrow\mathbb{R}$ by
$$A_{\vec{r}}(\vec{f})(\vec{x})=\frac{1}{N(R_{\vec{r}}(\vec{x}))^{m}}
\prod\limits_{j=1}^m
\sum\limits_{\vec{k}\in R_{\vec{r}}(\vec{x})\cap\mathbb{Z}^d}f_j(\vec{k})
,\ \ \forall \vec{x}\in\mathbb{R}^d.
$$
We can write
$$
\mathbb{M}_{\mathcal{R}}(\vec{f})(\vec{n})
=\sup\limits_{\vec{r}\in\mathbb{R}_{+}^d\atop\vec{x}\in\mathbb{R}^d,\,
\vec{n}\in R_{\vec{r}}(\vec{x})}A_{\vec{r}}(\vec{f})(\vec{x}),
\ \ \forall \vec{n}\in\mathbb{Z}^d.
$$
%%%%%%%%%%%%% Lemma 7.1 %%%%%%%%%%%%%%%%%%%%%%%%%%%%%%%%%%
\begin{lemma}\label{l6.1}Let $\vec{f}=(f_1,\ldots,f_m)$ 
with each $f_j\in\ell^1(\mathbb{Z}^d)$. Then for any 
$\vec{n}\in\mathbb{Z}^d$, ${\mathbb{M}}_{\mathcal{R}}
(\vec{f})(\vec{n})$ is attained for some $R$ with 
$\vec{n}\in R\in\mathcal{R}$.
\end{lemma}
%%%%%%%%%%% end Lemma 7.1 %%%%%%%%%%%%%%%%%%%%%%%%%%%%%%
\begin{proof}
Fix $\vec{n}\in\mathbb{Z}^d$. If
${\mathbb{M}}_{\mathcal{R}}(\vec{f})(\vec{n})=0$,
then all $f_j\equiv0$. For any $R$
with $\vec{n}\in R\in\mathcal{R}$, it suffices to show that
$$
{\mathbb{M}}_{\mathcal{R}}(\vec{f})(\vec{n})=\frac{1}{N(R)^m}
\prod\limits_{i=1}^m\sum\limits_{\vec{k}\in R\cap\mathbb{Z}^d}|f_i(\vec{k})|=0.
$$
If ${\mathbb{M}}_{\mathcal{R}}(\vec{f})(\vec{n})>0$.
Suppose that ${\mathbb{M}}_{\mathcal{R}}(\vec{f})
(\vec{n})$ is not attained for $R$ with $\vec{n}\in
R\in\mathcal{R}$. Let $\{r_k\}_{k\geq1}$ be an
increasing sequence of positive numbers with
$\lim_{k\rightarrow\infty}r_k=\infty$. By the definition
of ${\mathbb{M}}_{\mathcal{R}}(\vec{f})$ and our assumption, we have
$$
{\mathbb{M}}_{\mathcal{R}}(\vec{f})(\vec{n})
=\sup\limits_{\vec{n}\in R\in\mathcal{R}\atop N(R)\geq r_k}\frac{1}{N(R)^m}
\prod\limits_{i=1}^m\sum\limits_{\vec{k}\in R\cap\mathbb{Z}^d}|f_i(\vec{k})|,
\ \ \ \forall k\geq1.
$$
It follows that
$$
{\mathbb{M}}_{\mathcal{R}}(\vec{f})(\vec{n})\leq\frac{1}{r_k^m}
\prod\limits_{i=1}^m\|f_i\|_{\ell^1(\mathbb{Z})},\ \ \forall k\geq1.
$$
Let $k\rightarrow\infty$, we obtain $\widetilde{\mathbb{M}}_{\mathcal{R}}
(\vec{f})(\vec{n})=0$, which is a contradiction. Thus,
${\mathbb{M}}_{\mathcal{R}}(\vec{f})(\vec{n})$ is
attained for some $R$ with $\vec{n}\in R\in\mathcal{R}$.
\end{proof}
%%%%%%%%%%%%% end Proof of Lemma 7.1 %%%%%%%%%%%%%%%%%%%%%%%%%%%%

Since all $f_j\in\ell^1(\mathbb{Z}^d)$, by Lemma \ref{l6.1},
for any $(n',n_d)\in\mathbb{Z}^d$, there exist
$\vec{x}\in\mathbb{R}^d$ and $\vec{r}(n',n_d)\in
\mathbb{R}_{+}^d$ such that $(n',n_d)\in R_{\vec{r}
(n',n_d)}(\vec{x})$ and $\mathbb{M}_{\mathcal{R}}(\vec{f})
(n',n_d)=A_{\vec{r}(n',n_d)}(\vec{f})(\vec{x})$. Then
$$
\mathbb{M}_{\mathcal{R}}(\vec{f})(n',n_d)
-\mathbb{M}_{\mathcal{R}}(\vec{f})(n',n_d+1)
\leq A_{\vec{r}(n',n_d)}(\vec{f})(\vec{x})-
A_{\vec{r}(n',n_d)}(\vec{f})(\vec{x}+\vec{e}_d).
$$
For convenience, we set $\vec{r}(n',n_d)=(r_1(n',n_d),
\ldots,r_d(n',n_d))$. Note that $(n',n_d)\in R_{\vec{r}
(n',n_d)}(\vec{x})$ and $R_{\vec{r}(n',n_d)}(\vec{x})
\subset R_{2\vec{r}(n',n_d)}(n',n_d)$. These facts
together with \eqref{6.2} yields that

\begin{eqnarray*}
&&A_{\vec{r}(n',n_d)}(\vec{f})(\vec{x})-A_{\vec{r}(n',n_d)}(\vec{f})(\vec{x}+\vec{e}_d)\\
&&\leq\displaystyle\prod\limits_{i=1}^{d}\frac{1}{F(r_i(n',n_d))^{m}}\sum\limits_{\mu=1}^m\Big|\sum\limits_{\vec{k}\in R_{\vec{r}(n',n_d)}(\vec{x})\cap\mathbb{Z}^d}f_\mu(\vec{k})-\sum\limits_{\vec{k}\in R_{\vec{r}(n',n_d)}(\vec{x}+\vec{e}_d)\cap\mathbb{Z}^d}f_\mu(\vec{k})\Big|\\
&&\quad\times\displaystyle\Big(\prod\limits_{i=1}^{\mu-1}\sum\limits_{\vec{k}\in R_{\vec{r}(n',n_d)}(\vec{x})\cap\mathbb{Z}^d}f_i(\vec{k})\Big)\Big(\prod\limits_{\nu=\mu+1}^m\sum\limits_{\vec{k}\in R_{\vec{r}(n',n_d)}(\vec{x}+\vec{e}_d)\cap\mathbb{Z}^d}f_\nu(\vec{k})\Big).
\\
&&\leq\displaystyle\sum\limits_{\mu=1}^m\prod\limits_{\nu\neq\mu,1\leq\nu\leq m}\|f_\nu\|_{\ell^1(\mathbb{Z}^d)}\prod\limits_{i=1}^{d}\frac{1}{F(r_i(n',n_d))^{m}}\sum\limits_{\vec{k}\in R_{2\vec{r}(n',n_d)}(n',n_d)\cap\mathbb{Z}^d}|D_df_\mu(\vec{k})|.
\end{eqnarray*}
Therefore, $\mathbb{M}_{\mathcal{R}}(\vec{f})(n',n_d)-
\mathbb{M}_{\mathcal{R}}(\vec{f})(n',n_d+1)$ can be
controlled by
$$
\displaystyle\sum\limits_{\mu=1}^m\prod\limits_{\nu\neq\mu,1\leq\nu\leq m}\|f_\nu\|_{\ell^1(\mathbb{Z}^d)}
\prod\limits_{i=1}^{d}\frac{1}{F(r_i(n',n_d))^{m}}\sum\limits_{\vec{k}\in R_{2\vec{r}(n',n_d)}(n',n_d)\cap\mathbb{Z}^d}|D_df_\mu(\vec{k})|.
$$
It follows that
\begin{equation}\label{6.6}
\begin{array}{ll}
&\displaystyle\sum\limits_{n'\in\mathbb{Z}^{d-1}}\sum\limits_{n_d\in X_{n'}^{+}}(\mathbb{M}_{\mathcal{R}}f(n',n_d)-\mathbb{M}_{\mathcal{R}}f(n',n_d+1))\\
&\leq\displaystyle\sum\limits_{\mu=1}^m\prod\limits_{\nu\neq\mu,1\leq\nu\leq m}\|f_\nu\|_{\ell^1(\mathbb{Z}^d)}\\
&\quad\times\displaystyle\Big(\sum\limits_{n'\in\mathbb{Z}^{d-1}}\sum\limits_{n_d\in X_{n'}^{+}}\prod\limits_{i=1}^{d}\frac{1}{(F(r_i(n',n_d)))^{m}}\sum\limits_{\vec{k}\in R_{2\vec{r}(n',n_d)}(n',n_d)\cap\mathbb{Z}^d}|D_df_\mu(\vec{k})|\Big).
\end{array}\end{equation}
By direct calculations, we obtain
\begin{equation}\label{6.7}
\begin{array}{ll}
&\displaystyle\sum\limits_{n'\in\mathbb{Z}^{d-1}}\sum\limits_{n_d\in X_{n'}^{+}}\prod\limits_{i=1}^{d}\frac{1}{(F(r_i(n',n_d)))^{m}}\sum\limits_{\vec{k}\in R_{2\vec{r}(n',n_d)}(n',n_d)\cap\mathbb{Z}^d}|D_df_\mu(\vec{k})|\\
&\leq\displaystyle\sum\limits_{(k_1,\ldots,k_d)\in\mathbb{Z}^d}|D_df_\mu(k_1,\ldots,k_d)|\prod\limits_{i=1}^{d}\Big(\sum\limits_{n_i\in\mathbb{Z}}\frac{1}{(F(r_i(n',n_d)))^{m}}\chi_{\{|k_i-n_i|<2r_i(n',n_d)\}}\Big).
\end{array}\end{equation}
For fixed $1\leq i\leq d$ and $k_i\in\mathbb{Z}$, we have
\begin{equation}\label{6.8}
\begin{array}{ll}
&\displaystyle\sum\limits_{n_i\in\mathbb{Z}}\frac{1}{(F(r_i(n',n_d)))^{m}}\chi_{\{|k_i-n_i|<2r_i(n',n_d)\}}\leq\displaystyle9+\sum\limits_{n_i\in\mathbb{Z}}\frac{1}{(2[\frac{|k_i-n_i|-3}{2}]+1)^{m}}\chi_{\{|k_i-n_i|\geq5\}}.
\end{array}\end{equation}
Note that $m\geq2$, then
$
\displaystyle\sum\limits_{n_i\in\mathbb{Z}}\frac{1}{(2[\frac{|k_i-n_i|-3}{2}]+1)^{m}}\chi_{\{|k_i-n_i|\geq5\}}\leq\displaystyle\sum\limits_{n_i\in\mathbb{Z}\atop|n_i|\geq5}\frac{1}{(2[\frac{|n_i|-3}{2}]+1)^{m}}\leq4.
$This together with \eqref{6.8} yields that
\begin{equation}\label{6.9}
\sum\limits_{n_i\in\mathbb{Z}}\frac{1}{(F(r_i(n',n_d)))^{m}}\chi_{\{|k_i-n_i|<2r_i(n',n_d)\}}\leq13.
\end{equation}
Combining \eqref{6.9} with \eqref{6.7} gives that
\begin{equation}\label{6.10}
\sum\limits_{n'\in\mathbb{Z}^{d-1}}\sum\limits_{n_d\in X_{n'}^{+}}\prod\limits_{i=1}^{d}\frac{1}{(F(r_i(n',n_d)))^{m}}\sum\limits_{\vec{k}\in R_{2\vec{r}(n',n_d)}(n',n_d)\cap\mathbb{Z}^d}|D_df_\mu(\vec{k})|\lesssim_{d}\|D_df_\mu\|_{\ell^1(\mathbb{Z}^d)}.
\end{equation}
Then \eqref{6.4} follows immediately from \eqref{6.6} 
and \eqref{6.10}.

%%%%%%%%%%%% The continuity part %%%%%%%%%%%%%%%%%%%%%%%%%%%%%%%%%
\subsection{\bf The continuity part}

Let $g_{i,j}\rightarrow f_j$ in $\ell^1(\mathbb{Z}^d)$
for all $1\leq j\leq m$ when $i\rightarrow\infty$. For
convenience, we set $\vec{g}_i=(g_{i,1},\ldots,g_{i,m})$.
It suffices to show that
\begin{equation}\label{6.11}
\lim\limits_{i\rightarrow\infty}\|D_l\mathbb{M}_{\mathcal{R}}(\vec{g}_i)
-D_l\mathbb{M}_{\mathcal{R}}(\vec{f})\|_{\ell^1(\mathbb{Z}^d)}=0, \quad \hbox{for each\ }1\leq l\leq d.
\end{equation} We only prove \eqref{6.11} for
the case $l=d$ and the other cases are similar.
%%%%%%%%%%%%%%%%%%
Since we have $\|D_\ell |g_{i,j}| - D_\ell |f_{j}|\|_{\ell^1
(\mathbb{Z}^d)}\le 2\||g_{i,j}| - |f_{j}|\|_{\ell^1
(\mathbb{Z}^d)}\le 2 \|g_{i,j} - f_{j}\|_{\ell^1(\mathbb{Z}^d)}$, 
we may assume without loss of generality that all $g_{i,j}\geq0$
and $f_j\geq0$. Given $\epsilon\in(0,1)$, there exists
$N_1=N_1(\epsilon,\vec{f})\in\mathbb{N}$ such that for all
$1\leq j\leq m$ and $i\geq N_1$,
\begin{equation}\label{6.12}
\|g_{i,j}-f_j\|_{\ell^1(\mathbb{Z}^d)}<\epsilon
\ \ {\rm and}\ \|g_{i,j}\|_{\ell^1(\mathbb{Z}^d)}
\leq\|f_j\|_{\ell^1(\mathbb{Z}^d)}+1.
\end{equation}
For all $1\leq j\leq m$ and $i\geq N_1$, it then follows 
that
\begin{equation}\label{6.13}
\|D_dg_{i,j}\|_{\ell^1(\mathbb{Z}^d)}
\leq 2\|g_{i,j}\|_{\ell^1(\mathbb{Z}^d)}
\leq 2(\|f_j\|_{\ell^1(\mathbb{Z}^d)}+1)
\end{equation} On the other
hand, by the boundedness part in Theorem \ref{thm7}, we obtain
$D_d\mathbb{M}_{\mathcal{R}}(\vec{f})\in\ell^1(\mathbb{Z}^d)$.
Hence, for the above $\epsilon>0$, there exists $\Lambda_1>0$
such that
\begin{equation}\label{6.14}
\max\{\|D_d\mathbb{M}_{\mathcal{R}}(\vec{f})
\chi_{(R_{\vec{\Lambda_1}}(\vec{0}))^c}\|_{\ell^1(\mathbb{Z}^d)},
\sup\limits_{1\leq j\leq m}
\|f_j\chi_{(R_{\vec{\Lambda_1}}(\vec{0}))^c}\|_{\ell^1(\mathbb{Z}^d)}\}
<\epsilon.
\end{equation}
Here $\vec{\Lambda_1}=(\Lambda_1,\ldots,\Lambda_1)$.
Since $m>1$, there exists an integer $\Lambda_2>0$
such that
\begin{equation}\label{6.15}
\Lambda_2^{1-m}<\epsilon.
\end{equation}

{\bf Step1: Reduction.} When $i\geq N_1$, we get from 
\eqref{6.13} that
$$\begin{array}{ll}
&|\mathbb{M}_{\mathcal{R}}(\vec{g}_i)(\vec{n})
-\mathbb{M}_{\mathcal{R}}(\vec{f})(\vec{n})|
\\
&\leq\displaystyle\sup\limits_{\vec{n}\in R\in\mathcal{R}}
\frac{1}{N(R)^{m}}\Big|\prod\limits_{i=1}^m
\sum\limits_{\vec{k}\in R\cap\mathbb{Z}^d}g_{i,j}(\vec{k})
-\prod\limits_{i=1}^m\sum\limits_{\vec{k}\in R\cap\mathbb{Z}^d}f_j(\vec{k})
\Big|
\\
&\leq\displaystyle\sum\limits_{\mu=1}^m
\sum\limits_{\vec{k}\in R\cap\mathbb{Z}^d}|g_{i,\mu}(\vec{k})-f_\mu(\vec{k})|
\Big(\prod\limits_{\iota=1}^{\mu-1}\sum\limits_{\vec{k}\in R\cap\mathbb{Z}^d}
g_{i,\iota}(\vec{k})\Big)
\Big(\prod\limits_{\nu=\mu+1}^m\sum\limits_{\vec{k}\in R\cap\mathbb{Z}^d}
f_\nu(\vec{k})\Big)
\\
&\leq\displaystyle\sum\limits_{\mu=1}^m
\|g_{i,\mu}-f_\mu\|_{\ell^1(\mathbb{Z}^d)}\Big(
\prod\limits_{\nu\neq\mu,1\leq\nu\leq m}(\|f_\nu\|_{\ell^1(\mathbb{Z}^d)}+1)\Big).
\end{array}
$$
This implies that $\mathbb{M}_{\mathcal{R}}(\vec{g}_i)
(\vec{n})\rightarrow\mathbb{M}_{\mathcal{R}}(\vec{f})
(\vec{n})$ as $i\rightarrow\infty$ for any
$\vec{n}\in\mathbb{Z}^d$, and
\begin{equation}\label{6.16}
D_d\mathbb{M}_{\mathcal{R}}(\vec{g}_i)(\vec{n})\rightarrow D_d\mathbb{M}_{\mathcal{R}}(\vec{f})(\vec{n})\ \ {\rm as}\ \ i\rightarrow\infty,\ \ \forall\vec{n}\in\mathbb{Z}^d.
\end{equation}
Let $\Lambda=\max\{\Lambda_1,\Lambda_2,6\}$. It follows
from \eqref{6.16} that there exists $N_2=N_2(\epsilon,
\Lambda)\in\mathbb{N}$ such that
\begin{equation}\label{6.17}
|D_d\mathbb{M}_{\mathcal{R}}(\vec{g}_i)(\vec{n})-D_d\mathbb{M}_{\mathcal{R}}
(\vec{f})(\vec{n})|\leq\frac{\epsilon}{(N(R_{\vec{2\Lambda}}(\vec{0})))},
\ \ \forall i\geq N_2\ {\rm and}\ \vec{n}\in R_{\vec{\Lambda}}(\vec{0})
\cap\mathbb{Z}^d.
\end{equation}
Then, by \eqref{6.14} and \eqref{6.17}, we have that for
all $i\geq N_2$,
$$\begin{array}{ll}
\|D_d\mathbb{M}_{\mathcal{R}}(\vec{g}_i)-D_d\mathbb{M}_{\mathcal{R}}
(\vec{f})\|_{\ell^1(\mathbb{Z}^d)}
&=\|(D_d\mathbb{M}_{\mathcal{R}}(\vec{g}_i)-D_d\mathbb{M}_{\mathcal{R}}
(\vec{f}))\chi_{N(R_{\vec{2\Lambda}}(\vec{0}))}\|_{\ell^1(\mathbb{Z}^d)}\\
&\quad+\|(D_d\mathbb{M}_{\mathcal{R}}(\vec{g}_i)-D_d\mathbb{M}_{\mathcal{R}}
(\vec{f}))\chi_{(N(R_{\vec{2\Lambda}}(\vec{0})))^c}\|_{\ell^1(\mathbb{Z}^d)}
\\
&\leq2\epsilon+\|D_d\mathbb{M}_{\mathcal{R}}(\vec{g}_i)
\chi_{(N(R_{\vec{2\Lambda}}(\vec{0})))^c}\|_{\ell^1(\mathbb{Z}^d)}
\end{array}
$$
Thus, to prove \eqref{6.11} for $l=d$,
it suffices to show that
\begin{equation}\label{6.18}
\|D_d\mathbb{M}_{\mathcal{R}}(\vec{g}_i)
\chi_{(N(R_{\vec{2\Lambda}}(\vec{0})))^c}\|_{\ell^1(\mathbb{Z}^d)}
\lesssim_{d,m,\vec{f}}\epsilon, \ \ \ \forall i\geq N_1.
\end{equation}

{\bf Step 2: Proof of \eqref{6.18}.} Note that
$$(R_{\vec{2\Lambda}}(\vec{0}))^c\cap\mathbb{Z}^d\subset
\bigcup\limits_{\mu=1}^dE_\mu
:=\bigcup\limits_{\mu=1}^d\mathbb{Z}^d\setminus
\{\vec{n}=(n_1,\ldots,n_d)\in\mathbb{Z}^d: |n_\mu|\leq 2\Lambda\}.
$$
Fix $j\geq N_1$. Then we have
\begin{equation}\label{6.19}
\|D_d\mathbb{M}_{\mathcal{R}}(\vec{g}_i)\chi_{(N(R_{\vec{2\Lambda}}
(\vec{0})))^c}\|_{\ell^1(\mathbb{Z}^d)}\leq\sum\limits_{\mu=1}^d A_{2,\mu}
:=\sum\limits_{\mu=1}^d\sum\limits_{\vec{n}\in E_\mu}
|D_d\mathbb{M}_{\mathcal{R}}(\vec{g}_i)(\vec{n})|.
\end{equation}

{\bf Step 3: Estimates for $A_{2,d}$.} For each
$n'\in\mathbb{Z}^{d-1}$, let
$$Y_{n'}^{+}=\{|n_d|\geq2\Lambda: \mathbb{M}_{\mathcal{R}}(\vec{g}_i)(n',n_d+1)\leq\mathbb{M}_{\mathcal{R}}(\vec{g}_i)(n',n_d)\}, \quad \hbox{and}$$
$$Y_{n'}^{-}=\{|n_d|\geq2\Lambda: \mathbb{M}_{\mathcal{R}}(\vec{g}_i)(n',n_d+1)>\mathbb{M}_{\mathcal{R}}(\vec{g}_i)(n',n_d)\}.$$
Then, we have
\begin{equation}\label{6.20}
\begin{array}{ll}
A_{2,d}
&\leq\displaystyle\sum\limits_{n'\in\mathbb{Z}^{d-1}}
\sum\limits_{n_d\in Y_{n'}^{+}}(\mathbb{M}_{\mathcal{R}}(\vec{g}_i)(n',n_d)
-\mathbb{M}_{\mathcal{R}}(\vec{g}_i)(n',n_d+1))
\\
&\quad+\displaystyle\sum\limits_{n'\in\mathbb{Z}^{d-1}}
\sum\limits_{n_d\in Y_{n'}^{-}}(\mathbb{M}_{\mathcal{R}}(\vec{g}_i)(n',n_d+1)
-\mathbb{M}_{\mathcal{R}}(\vec{g}_i)(n',n_d)).
\end{array}\end{equation}
We want to show that
\begin{equation}\label{6.21}
\sum\limits_{n'\in\mathbb{Z}^{d-1}}\sum\limits_{n_d\in Y_{n'}^{+}}
(\mathbb{M}_{\mathcal{R}}(\vec{g}_i)(n',n_d)
-\mathbb{M}_{\mathcal{R}}(\vec{g}_i)(n',n_d+1))\lesssim_{d,m,\vec{f}}\epsilon,
\ \ \forall i\geq N_1;
\end{equation}
\begin{equation}\label{6.22}
\sum\limits_{n'\in\mathbb{Z}^{d-1}}\sum\limits_{n_d\in Y_{n'}^{-}}
(\mathbb{M}_{\mathcal{R}}(\vec{g}_i)(n',n_d+1)
-\mathbb{M}_{\mathcal{R}}(\vec{g}_i)(n',n_d))
\lesssim_{d,m,\vec{f}}\epsilon,\ \ \forall i\geq N_1.
\end{equation}

We will only prove \eqref{6.21}, since \eqref{6.22} is 
analogous. Fix $i\geq N_1$. Since all $g_{i,j}\in\ell^1
(\mathbb{Z}^d)$, then for any $(n',n_d)\in\mathbb{Z}^d$, 
there exist $\vec{x}\in\mathbb{R}^d$ and $\vec{r}(n',n_d)
\in\mathbb{R}_{+}^d$ such that $(n',n_d)\in R_{\vec{r}
(n',n_d)}(\vec{x})$ and $\mathbb{M}_{\mathcal{R}}
(\vec{g_i})(n',n_d)=A_{\vec{r}(n',n_d)}(\vec{g_i})(\vec{x})$. 
Let $\vec{r}(n',n_d)=(r_1(n',n_d),\ldots,r_d(n',n_d))$. By 
the similar arguments as in getting \eqref{6.6} and 
\eqref{6.7}, we obtain
\begin{equation}\label{6.23}
\begin{array}{ll}
&\displaystyle\sum\limits_{n'\in\mathbb{Z}^{d-1}}\sum\limits_{n_d\in Y_{n'}^{+}}(\mathbb{M}_{\mathcal{R}}(\vec{g}_i)(n',n_d)-\mathbb{M}_{\mathcal{R}}(\vec{g}_i)(n',n_d+1))\\
&\leq\displaystyle\sum\limits_{\mu=1}^m\prod\limits_{\nu\neq\mu,1\leq\nu\leq m}\|g_{i,\nu}\|_{\ell^1(\mathbb{Z}^d)}\\
&\quad\times\displaystyle\Big(\sum\limits_{n'\in\mathbb{Z}^{d-1}}\sum\limits_{n_d\in Y_{n'}^{+}}\prod\limits_{i=1}^{d}\frac{1}{(F(r_i(n',n_d)))^{m}}\sum\limits_{\vec{k}\in R_{2\vec{r}(n',n_d)}(n',n_d)\cap\mathbb{Z}^d}|D_dg_{i,\mu}(\vec{k})|\Big)
\end{array}\end{equation}\begin{equation*}
\begin{array}{ll}{}&\leq\displaystyle\sum\limits_{\mu=1}^m\prod\limits_{\nu\neq\mu,1\leq\nu\leq m}\|g_{i,\nu}\|_{\ell^1(\mathbb{Z}^d)}\\
&\quad\times\displaystyle\sum\limits_{(k_1,\ldots,k_d)\in\mathbb{Z}^d}|D_dg_{i,\mu}(k_1,\ldots,k_d)|\prod\limits_{i=1}^{d}\Big(\sum\limits_{n_i\in\mathbb{Z}}\frac{1}{(F(r_i(n',n_d)))^{m}}\chi_{\{|k_i-n_i|<2r_i(n',n_d)\}}\Big).
\end{array}\end{equation*}
When $|k_d|>\Lambda$, by \eqref{6.9}, we get
\begin{equation}\label{6.24}
\sum\limits_{n_d\in\mathbb{Z}\atop|n_d|\geq2\Lambda}\frac{1}{(F(r_d(n',n_d)))^{m}}\chi_{\{|k_d-n_d|<2r_d(n',n_d)\}}\leq13.
\end{equation}
When $|k_d|\leq\Lambda$, then $|k_d-n_d|\geq\Lambda\geq6$
for $|n_d|\geq2\Lambda$. This together with \eqref{6.15}
yields that
\begin{equation}\label{6.25}
\begin{array}{ll}
&\displaystyle\sum\limits_{n_d\in\mathbb{Z}\atop|n_d|\geq2\Lambda}\frac{1}{(F(r_d(n',n_d)))^{m}}\chi_{\{|k_d-n_d|<2r_d(n',n_d)\}}\\
&\leq\displaystyle\sum\limits_{n_d\in\mathbb{Z}\atop|n_d|\geq2\Lambda}\frac{1}{(2[\frac{|k_d-n_d|-3}{2}]+1)^{m}}\chi_{\{|k_d-n_d|\geq\Lambda\}}\\
\\&\lesssim_{d,m}\epsilon.
\end{array}\end{equation}
By \eqref{6.12}-\eqref{6.14} and \eqref{6.24}-\eqref{6.25},
we obtain
\begin{equation}\label{6.26}
\begin{array}{ll}
&\displaystyle\sum\limits_{(k_1,\ldots,k_d)\in\mathbb{Z}^d}
|D_dg_{i,\mu}(k_1,\ldots,k_d)|\prod\limits_{i=1}^{d}
\Big(\sum\limits_{n_i\in\mathbb{Z}}\frac{1}{(F(r_i(n',n_d)))^{m}}
\chi_{\{|k_i-n_i|<2r_i(n',n_d)\}}\Big)
\\
&\lesssim_{d,m}\displaystyle\sum\limits_{k'\in\mathbb{Z}^{d-1}}
\sum\limits_{k_d\in\mathbb{Z}\atop|k_d|>\Lambda}|D_dg_{i,\mu}(k',k_d)|
+\sum\limits_{k'\in\mathbb{Z}^{d-1}}\sum\limits_{k_d\in\mathbb{Z}\atop|k_d|
\leq\Lambda}|D_dg_{i,\mu}(k',k_d)|\epsilon
\\
&\lesssim_{d,m}\|g_{i,\mu}-f_\mu\|_{\ell^1(\mathbb{Z}^d)}
+\|f_\mu\chi_{(R_{\vec{\Lambda}}(\vec{0}))^c)}\|_{\ell^1(\mathbb{Z}^d)}
+2(\|f_\mu\|_{\ell^1(\mathbb{Z}^d)}+1)\epsilon
\\
&\lesssim_{d,m,f_\mu}\epsilon.
\end{array}\end{equation}
\eqref{6.26} together with \eqref{6.12}-\eqref{6.13} and
\eqref{6.23} yields \eqref{6.21}. It follows from
\eqref{6.20}-\eqref{6.22} that
\begin{equation}\label{6.27}
A_{2,d}\lesssim_{d,m,\vec{f}}\epsilon,\ \ \forall i\geq N_1.
\end{equation}

{\bf Step 4: Estimates for $A_{2,\mu}$ with $\mu=1,2,\ldots,
d-1$.} We first estimates $A_{2,1}$. For each 
$n'\in\mathbb{Z}^{d-1}$, let
$$
Z_{n'}^{+}=\{n_d\in\mathbb{Z}: \mathbb{M}_{\mathcal{R}}(\vec{g}_i)(n',n_d+1)
\leq\mathbb{M}_{\mathcal{R}}(\vec{g}_i)(n',n_d)\},
$$
$$
Z_{n'}^{-}=\{n_d\in\mathbb{Z}:
\mathbb{M}_{\mathcal{R}}(\vec{g}_i)(n',n_d+1)>\mathbb{M}_{\mathcal{R}}
(\vec{g}_i)(n',n_d)\}.
$$
Then we have
\begin{equation}\label{6.28}
\begin{array}{ll}
A_{2,1}&=\displaystyle\sum\limits_{|n_1|>2\Lambda}\sum\limits_{n'\in\mathbb{Z}^{d-1}}|\mathbb{M}_{\mathcal{R}}(\vec{g}_i)(n_1,\ldots,n_{d-1},n_d+1)-\mathbb{M}_{\mathcal{R}}(\vec{g}_i)(n_1,\ldots,n_d)|\\
&=\displaystyle\sum\limits_{n'\in\mathbb{Z}^{d-1}\atop |n_1|>2\Lambda}\sum\limits_{n_d\in Z_{n'}^{+}}(\mathbb{M}_{\mathcal{R}}(\vec{g}_i)(n',n_d)-\mathbb{M}_{\mathcal{R}}(\vec{g}_i)(n',n_d+1))\\
&\quad+\displaystyle\sum\limits_{n'\in\mathbb{Z}^{d-1}\atop|n_1|>2\Lambda}\sum\limits_{n_d\in Z_{n'}^{-}}(\mathbb{M}_{\mathcal{R}}(\vec{g}_i)(n',n_d+1)-\mathbb{M}_{\mathcal{R}}(\vec{g}_i)(n',n_d)).
\end{array}
\end{equation}
We want to show that
\begin{equation}\label{6.29}
\sum\limits_{n'\in\mathbb{Z}^{d-1}\atop |n_1|>2\Lambda}
\sum\limits_{n_d\in Z_{n'}^{+}}(\mathbb{M}_{\mathcal{R}}(\vec{g}_i)(n',n_d)
-\mathbb{M}_{\mathcal{R}}(\vec{g}_i)(n',n_d+1))
\lesssim_{d,m,\vec{f}}\epsilon,\ \ \forall i\geq N_1;
\end{equation}
\begin{equation}\label{6.30}
\sum\limits_{n'\in\mathbb{Z}^{d-1}\atop|n_1|>2\Lambda}
\sum\limits_{n_d\in Z_{n'}^{-}}(\mathbb{M}_{\mathcal{R}}(\vec{g}_i)(n',n_d+1)
-\mathbb{M}_{\mathcal{R}}(\vec{g}_i)(n',n_d))
\lesssim_{d,m,\vec{f}}\epsilon,
\ \ \forall i\geq N_1.
\end{equation}
We will only prove \eqref{6.29}, since \eqref{6.30} is 
analogous. By the similar arguments as in getting 
\eqref{6.23}, for any $(n',n_d)\in\mathbb{Z}^d$, there 
exists $\vec{r}(n',n_d)=(r_1(n',n_d),\ldots,r_d(n',n_d))
\in\mathbb{R}_{+}^d$ such that
\begin{equation}\label{6.31}
\begin{array}{ll}
&\displaystyle\sum\limits_{n'\in\mathbb{Z}^{d-1}\atop|n_1|>2\Lambda}
\sum\limits_{n_d\in Z_{n'}^{+}}(\mathrm{M}_{\mathcal{R}}f_j(n',n_d)
-\mathrm{M}_{\mathcal{R}}f_j(n',n_d+1))
\\
&\lesssim_{m}\displaystyle\sum\limits_{\mu=1}^m
\prod\limits_{\nu\neq\mu,1\leq\nu\leq m}\|g_{i,\nu}\|_{\ell^1(\mathbb{Z}^d)}
\\
&\quad\times\displaystyle\sum\limits_{(k_1,\ldots,k_d)\in\mathbb{Z}^d}
|D_dg_{i,\mu}(k_1,\ldots,k_d)|\Big(\sum\limits_{|n_1|>2\Lambda}
\frac{1}{(F(r_1(n',n_d)))^{m}}\chi_{\{|k_1-n_1|<2r_1(n',n_d)\}}\Big)
\\
&\quad\times\displaystyle\prod\limits_{i=2}^{d}\Big(
\sum\limits_{n_i\in\mathbb{Z}}\frac{1}{(F(r_i(n',n_d)))^{m}}
\chi_{\{|k_i-n_i|<2r_i(n',n_d)\}}\Big).
\end{array}\end{equation}
When $|k_1|\leq\Lambda$, then $|k_1-n_1|\geq\Lambda\geq6$
for $|n_1|\geq2\Lambda$. By the arguments similar to those
used in deriving \eqref{6.25}, we can obtain
\begin{equation}\label{6.32}
\sum\limits_{|n_1|>2\Lambda}\frac{1}{(F(r_1(n',n_d)))^{m}}
\chi_{\{|k_1-n_1|<2r_1(n',n_d)\}}\lesssim_{d,m}\epsilon.
\end{equation}
When $|k_1|>\Lambda$, by \eqref{6.9}, we get
\begin{equation}\label{6.33}
\sum\limits_{n_d\in\mathbb{Z}\atop|n_1|\geq2\Lambda}
\frac{1}{(F(r_1(n',n_d)))^{m}}\chi_{\{|k_1-n_1|<2r_1(n',n_d)\}}\leq13.
\end{equation}
Then \eqref{6.29} follows from \eqref{6.31}-\eqref{6.33}
and \eqref{6.13}-\eqref{6.14}. By \eqref{6.28}-\eqref{6.30}, 
it holds that 
\begin{equation}\label{6.34}
A_{2,1}\lesssim_{d,m,\vec{f}}\epsilon,\ \ \forall i\geq N_1.
\end{equation}
Similarly, we can get
\begin{equation}\label{6.35}
A_{2,\mu}\lesssim_{d,m,\vec{f}}\epsilon,\ \ \forall i\geq N_1
\ {\rm and}\ \mu=2,\ldots,d-1.
\end{equation}
\eqref{6.35} together with \eqref{6.19}, \eqref{6.27} and
\eqref{6.34} yields \eqref{6.18}. Thus, the proof
of the continuity part is complete. $\hfill\Box$

\section{Properties of $u_{x,\vec{f}}$}\label{S7}

We summrize the properties of $u_{x,\vec{f}}$ into nine
Claims. The proofs of them are thoroughly elementary.
In what follow, we set $\overline{\mathbb{Q}}_{+}=
\mathbb{Q}_{+}\cup\{0\}$.
\par\medskip
{\bf Claim 1.}
Let $1<p_j<\infty$ and $f_j\in L^{p_j}(\mathbb R^2)$ for
$1\leq j\leq m$. Then
\begin{equation}\label{7.1}
\lim_{(r_{1,1},r_{1,2},r_{2,1},r_{2,2})\in\overline{\mathbb{R}}_{+}^4\atop
r_{1,1}+r_{1,2}+r_{2,1}+r_{2,2}\to\infty}
u_{x,\vec{f}}(r_{1,1},r_{1,2},r_{2,1},r_{2,2})=0
\ \text{ for a.e. }x\in\mathbb R^2.
\end{equation}
\par\medskip
\begin{proof}
Fix $1\leq j\leq m$. We note first
$$
\|f_j(x_1,\cdot)\|_{L^{p_j}(\mathbb R)},\,
\|\widetilde{\mathcal{M}}_{1}f_j(x_1,\cdot)\|_{L^{p_j}(\mathbb R)}
<\infty \ \text{ a.e }x_1\in \mathbb R, \quad \hbox{and}
$$
$$
\|f_j(\cdot,x_2)\|_{L^{p_j}(\mathbb R)},\,
\|\widetilde{\mathcal{M}}_{2}f_j(\cdot,x_2)
\|_{L^{p_j}(\mathbb R)}
<\infty \ \text{ a.e }x_2\in \mathbb R.
$$
Let $\vec{r}=(r_{1,1},r_{1,2},r_{2,1},r_{2,2})\in
\overline{\mathbb{R}}_{+}^4$. We consider the following
three cases.

(a) $r_{1,1}+r_{1,2}>0$ and $r_{2,1}+r_{2,2}>0$. By
H\"{o}lder's inequality, one finds that
\begin{align*}
u_{x,\vec{f}}(\vec r)
&\le \prod_{j=1}^{m}\frac{1}{r_{1,1}+r_{1,2}}\int_{x_1-r_{1,1}}^{x_1+r_{1,2}}
\widetilde{\mathcal{M}}_2f_j(y_1,x_2)dy_1
\le \prod_{j=1}^{m}\frac{1}{(r_{1,1}+r_{1,2})^{1/{p_j}}}
\|\widetilde{\mathcal{M}}_2f_j(\cdot,x_2)\|_{L^{p_j}(\mathbb R)}.
\end{align*}
Simiarly we get
\begin{equation*}
u_{x,\vec{f}}(\vec r)\le \prod_{j=1}^{m}\frac{1}{(r_{2,1}+r_{2,2})^{1/{p_j}}}
\|\widetilde{\mathcal{M}}_1f_j(x_1\cdot)\|_{L^{p_j}(\mathbb R)}.
\end{equation*}
So, we have
\begin{equation*}
u_{x,\vec{f}}(\vec r)\le
\min\biggl\{\prod_{j=1}^{m}\frac{1}{(r_{1,1}+r_{1,2})^{1/{p_j}}}
\|\widetilde{\mathcal{M}}_2f_j(\cdot,x_2)\|_{L^{p_j}(\mathbb R)},\,
\prod_{j=1}^{m}\frac{1}{(r_{2,1}+r_{2,2})^{1/{p_j}}}
\|\widetilde{\mathcal{M}}_1f_j(x_1\cdot)\|_{L^{p_j}(\mathbb R)}
\biggr\}.
\end{equation*}

(b) $r_{1,1}=r_{1,2}=0$ and $r_{2,1}+r_{2,2}>0$. By
H\"{o}lder's inequality,
$$u_{x,\vec{f}}(\vec r)=\prod_{j=1}^{m}\frac{1}{r_{2,1}+r_{2,2}}
\int_{x_2-r_{2,1}}^{x_2+r_{2,2}}f_j(x_1,y_2)dy_2\le \prod_{j=1}^{m}\frac{1}{(r_{2,1}+r_{2,2})^{1/{p_j}}}
\|f(x_1,\cdot)\|_{L^{p_j}(\mathbb R)}.$$

(c) $r_{1,1}+r_{1,2}>0$ and $r_{2,1}=r_{2,2}=0$. Similarly
to (b) we get
\begin{equation*}
u_{x,\vec{f}}(\vec r)
\le\prod_{j=1}^{m}\frac{1}{(r_{1,1}+r_{1,2})^{1/{p_j}}}\|f_j(\cdot,x_2)\|_{L^{p_j}(\mathbb R)}.
\end{equation*}

From (a), (b) and (c), we see that \eqref{7.1} holds for
a.e. $x\in \mathbb R^2$.
\end{proof}
%%%%%%%%%%%%%%%% end Proof of Claim 0 %%%%%%%%%%%%%%%%%%%%%%%%%%%%%%%
{\bf Claim  2.} Let $1\leq j\leq m$ and
$f_j\in L_{\mathrm{loc}}^1(\mathbb R^2)$. Then the set $A_j$
\begin{equation*}
A_j:=\Bigl\{
(x_1,x_2)\in\mathbb R^2: \lim\limits_{(r_{1,1},r_{1,2})\in\overline{\mathbb{R}}_{+}^2\atop
r_{1,1}+r_{1,2}\to 0}\frac{1}{r_{1,1}+r_{1,2}}
\int_{x_1-r_{1,1}}^{x_1+r_{1,2}}|f_j(y_1,x_2)-f_j(x_1,x_2)|dy_1=0\Bigr\}
\end{equation*}
is a measurable set in $\mathbb R^2$.
\begin{proof}
$(x_1,x_2)\in A_j$ is equivalent to the following: For any
$k\in \mathbb N$, there exists an $\ell\in\mathbb N$ such
that for $r_{1,1},r_{1,2}\geq0$ with $r_{1,1}+r_{1,2}<1/\ell$,
\begin{equation*}
\frac{1}{r_{1,1}+r_{1,2}}
\int_{x_1-r_{1,1}}^{x_1+r_{1,2}}|f_j(y_1,x_2)-f_j(x_1,x_2)|dy_1<\frac{1}{k}.
\end{equation*}
And this is equivalent to: For any $k\in \mathbb N$,
there exists an $\ell\in\mathbb N$ such that for
$r_{1,1},r_{1,2}\in \mathbb Q_+$ with $r_{1,1}+r_{1,2}<1/\ell$,
\begin{equation*}
\frac{1}{r_{1,1}+r_{1,2}}
\int_{x_1-r_{1,1}}^{x_1+r_{1,2}}|f_j(y_1,x_2)-f_j(x_1,x_2)|dy_1<\frac{1}{k}.
\end{equation*}
Thus, $A_j$ can be written in the following form:
\begin{equation*}
A_j=\bigcap_{k\in\mathbb N}\bigcup_{\ell\in\mathbb N}
\bigcap_{\substack{(r_{1,1},r_{1,2})\in \mathbb Q_+^2\atop r_{1,1}+r_{1,2}<1/\ell}}
\Bigl\{(x_1,x_2)\in\mathbb{R}^2:\frac{1}{r_{1,1}+r_{1,2}}
\int_{x_1-r_{1,1}}^{x_1+r_{1,2}}|f_j(y_1,x_2)-f_j(x_1,x_2)|dy_1<\frac{1}{k}\Bigr\}.
\end{equation*}
Since $f_j\in L_{\mathrm{loc}}^1(\mathbb R^2)$, we see
that the function $\frac{1}{r_{1,1}+r_{1,2}}\int_{x_1-
r_{1,1}}^{x_1+r_{1,2}}|f_j(y_1,x_2)-f_j(x_1,x_2)|dy_1$
is measurable in $\mathbb R^2$, and hence
\begin{equation*}
\Bigl\{(x_1,x_2)\in\mathbb{R}^2:\frac{1}{r_{1,1}+r_{1,2}}
\int_{x_1-r_{1,1}}^{x_1+r_{1,2}}|f_j(y_1,x_2)-f_j(x_1,x_2)|dy_1<\frac{1}{k}
\Bigr\}
\end{equation*}
is a measurable set in $\mathbb R^2$. Thus, $A_j$ is a
measurable set in $\mathbb R^2$.
\end{proof}
{\bf Claim 3.} Let $1\leq j\leq m$, $f_j\in L_{\mathrm{loc}}^1
(\mathbb R^2)$ and $r_{1,1}+r_{1,2}>0$. Then the set $B_j$
$$\begin{array}{ll}
B_j:=\displaystyle\Big\{(x_1,x_2)\in\mathbb R^2:\lim_{(r_{2,1},r_{2,2})\in\overline{\mathbb{R}}_{+}^2\atop
r_{2,1}+r_{2,2}\to 0}\frac{1}{r_{2,1}+r_{2,2}}&\displaystyle\int_{x_2-r_{2,1}}^{x_2+r_{2,2}}\biggl|\frac{1}{r_{1,1}+r_{1,2}}\int_{x_1-r_{1,1}}^{x_1+r_{1,2}}f_j(y_1,y_2)dy_1\\
&-\displaystyle\frac{1}{r_{1,1}+r_{1,2}}\int_{x_1-r_{1,1}}^{x_1+r_{1,2}}f_j(y_1,x_2)dy_1\Big|dy_2=0\Big\}
\end{array}$$
is a measurable set in $\mathbb R^2$.
\par\smallskip
\begin{proof}
As in the proof of Claim 2, we can write
\begin{align*}
B_j=\bigcap_{k\in\mathbb N}\bigcup_{\ell\in\mathbb N}
\bigcap_{\substack{(r_{2,1},r_{2,2})\in\mathbb Q_+^2\atop r_{2,1}+r_{2,2}<1/\ell}}
\Bigl\{(x_1,x_2)\in\mathbb{R}^2:&\frac{1}{r_{2,1}+r_{2,2}}
\int_{x_2-r_{2,1}}^{x_2+r_{2,2}}
\biggl|\frac{1}{r_{1,1}+r_{1,2}}\int_{x_1-r_{1,1}}^{x_1+r_{1,2}}f_j(y_1,y_2)dy_1
\\
&-\frac{1}{r_{1,1}+r_{1,2}}\int_{x_1-r_{1,1}}^{x_1+r_{1,2}}f_j(y_1,x_2)dy_1
\biggr|dy_2<\frac{1}{k}\Bigr\}.
\end{align*}
Hence as in the proof of Claim 2, we see that $B_j$ is a
measurable set in $\mathbb R^2$.

\end{proof}

{\bf Claim 4.} Let $1<p_j<\infty$ and $f_j\in L^{p_j}(\mathbb R^2)$
for $1\leq j\leq m$. Then
\begin{equation}\label{7.2}
\lim\limits_{(r_{1,1},r_{1,2})\in\overline{\mathbb{R}}_{+}^2\atop
r_{1,1}+r_{1,2}\rightarrow0}u_{x,\vec{f}}(r_{1,1},r_{1,2},0,0)=u_{x,\vec{f}}(0,0,0,0)\ \ {\rm for\ a.e.}\ x\in\mathbb{R}^2.
\end{equation}
\par\medskip
\begin{proof}
To prove \eqref{7.2}, it suffices to show that for fixed 
$1\leq j\leq m$, there exists a null set $E_{1}$ in 
$\mathbb R^2$ such that for $(x_1,x_2)\in\mathbb R^2
\setminus E_{1}$,
\begin{equation}\label{7.3}
\lim\limits_{(r_{1,1},r_{1,2})\in\overline{\mathbb{R}}_{+}^2\atop r_{1,1}+r_{1,2}\to 0}\frac{1}{r_{1,1}+r_{1,2}}
\int_{x_1-r_{1,1}}^{x_1+r_{1,2}}|f_j(y_1,x_2)-f_j(x_1,x_2)|dy_1=0.
\end{equation}
From $f_j\in L^{p_j}(\mathbb R^2)$ it follows that 
$f_j(\cdot, x_2)\in L^{p_j}(\mathbb R)$ a.e $x_2\in\mathbb{R}$.
Hence for these $x_2$, by Lebesgue's differentiation theorem
(note $p_j>1$) we see that \eqref{7.3} holds for a.e.
$x_1\in\mathbb R$. By Claim 2, we see that there exists a null
set $E_1$ in $\mathbb R^2$ such that \eqref{7.3} holds for
$x\in\mathbb R^2\setminus E_{1}$.
\end{proof}

\medskip

Applying the arguments similar to those used in deriving Claim 4,
we can get the following claim. The details are omitted.

{\bf Claim 5.} Let $1<p_j<\infty$ and $f_j\in L^{p_j}(\mathbb R^2)$
for $1\leq j\leq m$. Then
$$\lim\limits_{(r_{2,1},r_{2,2})\in\overline{\mathbb{R}}_{+}^2\atop r_{2,1}+r_{2,2}\rightarrow0}u_{x,\vec{f}}(0,0,r_{2,1},r_{2,2})=u_{x,\vec{f}}(0,0,0,0)\ \ {\rm for\ a.e.}\ x\in\mathbb{R}^2.$$
\par\medskip
%%%%%%%%%%%%%%%% end Proof of Claim 3 %%%%%%%%%%%%%%%%%%%%%%%%%%%%%%%

{\bf Claim 6.} Let $1<p_j<\infty$ and $f_j\in L^{p_j}
(\mathbb R^2)$ for $1\leq j\leq m$. Then there exists
a null set $E_{\mathbb Q}$ in $\mathbb R^2$ such that
for $(x_1,x_2)\in\mathbb R^2\setminus E_{\mathbb Q}$
and $(r_{1,1},r_{1,2})\in\overline{\mathbb Q}_+^2$ with
$r_{1,1}+r_{1,2}>0$,
\begin{equation}\label{7.4}
\begin{array}{ll}
&\displaystyle\lim\limits_{(r_{2,1},r_{2,2})\in\overline{\mathbb{R}}_{+}^2\atop r_{2,1}+r_{2,2}\to 0}
\frac{1}{r_{2,1}+r_{2,2}}
\int_{x_2-r_{2,1}}^{x_2+r_{2,2}}
\biggl|\frac{1}{r_{1,1}+r_{1,2}}\int_{x_1-r_{1,1}}^{x_1+r_{1,2}}f_j(y_1,y_2)dy_1\\
&\qquad\qquad\qquad\qquad-\displaystyle\frac{1}{r_{1,1}+r_{1,2}}\int_{x_1-r_{1,1}}^{x_1+r_{1,2}}f_j(y_1,x_2)dy_1
\biggr|dy_2=0.
\end{array}
\end{equation}
\par\medskip
\begin{proof}
Let $(r_{1,1}, r_{1,2})\in\overline{\mathbb{R}}_{+}^2$ with
$r_{1,1}+ r_{1,2}>0$. From $f_j\in L^{p_j}(\mathbb R^2)$ we
see that for all $x_1\in\mathbb R$,
\begin{equation*}
\biggl|
\frac{1}{r_{1,1}+r_{1,2}}\int_{x_1-r_{1,1}}^{x_1+r_{1,2}}f_j(y_1,y_2)dy_1
\biggr|
\le \frac{1}{(r_{1,1}+r_{1,2})^{1/p}}\|f_j(\cdot,y_2)\|_{L^{p_j}(\mathbb R)}.
\end{equation*}
and hence
\begin{equation*}
\biggl(\int_{\mathbb R}\biggl|
\frac{1}{r_{1,1}+r_{1,2}}\int_{x_1-r_{1,1}}^{x_1+r_{1,2}}f_j(y_1,y_2)dy_1
\biggr|^pdy_2\biggr)^{1/p}
\le \frac{1}{(r_{1,1}+r_{1,2})^{1/p}}\|f_j\|_{L^{p_j}(\mathbb R^2)}.
\end{equation*}
Hence, for every $x_1\in\mathbb R$, by Lebesgue's differentiation
theorem, \eqref{7.4} holds for a.e. $x_2\in\mathbb R$. So, by
Claim 2, there exists a null set $E_{r_{1,1},r_{1,2}}\in\mathbb R^2$
such that \eqref{7.4} holds for $(x_1,x_2)\in\mathbb R^2\setminus
E_{r_{1,1}, r_{1,2}}$. Now set
$$
E_{\mathbb Q}=\bigcup\limits_{r_{1,1},r_{1,2}\in\overline{{\mathbb Q}}_{+}\atop
r_{1,1}+r_{1,2}>0}E_{r_{1,1}, r_{1,2}}.
$$
Then $|E_{\mathbb Q}|=0$ and for $(x_1,x_2)\in\mathbb R^2
\setminus E_{\mathbb Q}$, \eqref{7.4} holds for $(r_{1,1},
r_{1,2})\in\overline{\mathbb Q}_+^2$ with $r_{1,1}+r_{1,2}>0$.
\end{proof}
%%%%%%%%%%%%% Claim 5 %%%%%%%%%%%%%%%%%%%%%%%%%%%%%%%%%%%%%
{\bf Claim 7.} Let $1<p_j<\infty$ and $f_j\in L^{p_j}(\mathbb R^2)$
for $1\leq j\leq m$. Then there exists a null set $E_{2}$ in
$\mathbb R^2$ such that for $(x_1,x_2)\in\mathbb R^2\setminus E_{2}$
and $(r_{1,1},r_{1,2})\in\overline{\mathbb{R}}_{+}^2$ with
$r_{1,1}+r_{1,2}>0$,
\begin{equation}\label{7.5}
\lim_{(r_{1,1}',r_{1,2}')\in\overline{\mathbb{R}}_{+}^2\atop
(r_{1,1}',r_{1,2}')\to(r_{1,1},r_{1,2})}
\frac{1}{r_{1,1}'+r_{1,2}'}\int_{x_1-r_{1,1}'}^{x_1+r_{1,2}'}f_j(y_1,x_2)dy_1
=\frac{1}{r_{1,1}+r_{1,2}}\int_{x_1-r_{1,1}}^{x_1+r_{1,2}}f_j(y_1,x_2)dy_1,
\end{equation}
for $1\le j\le m$,
\par\medskip
\begin{proof}
Fix $(r_{1,1}',r_{1,2}')\in\overline{\mathbb{R}}_{+}^2$, we have
\begin{eqnarray*}
&&\biggl| \frac{1}{r_{1,1}'+r_{1,2}'}
\int_{x_1-r_{1,1}'}^{x_1+r_{1,2}'}f_j(y_1,x_2)dy_1-
\frac{1}{r_{1,1}+r_{1,2}}
\int_{x_1-r_{1,1}}^{x_1+r_{1,2}}f_j(y_1,x_2)dy_1\biggr|
\\
&&\le
\biggl|\frac{1}{r_{1,1}'+r_{1,2}'}-\frac{1}{r_{1,1}+r_{1,2}}\biggr|
\int_{x_1-r_{1,1}'}^{x_1+r_{1,2}'}|f_j(y_1,x_2)|dy_1
\\
&&\quad+\frac{1}{r_{1,1}+r_{1,2}}\biggl(
\int_{x_1+\min(r_{1,2}',r_{1,2})}^{x_1+\max(r_{1,2}',r_{1,2})}
|f_j(y_1,x_2)|dy_1+\int_{x_1-\max(r_{1,1}',r_{1,1})}^{x_1-\min(r_{1,1}',r_{1,1})}
|f_j(y_1,x_2)|dy_1\biggr)
\\
&&\le
\frac{|r_{1,1}-r_{1,1}'|+|r_{1,2}-r_{1,2}'|}
{(r_{1,1}+r_{1,2})(r_{1,1}'+r_{1,2}')}(r_{1,1}'+r_{1,2}')^{1/p_j'}
\|f_j(\cdot,x_2)\|_{L^{p_j}(\mathbb R)}
\\
&&
\quad+\frac{|r_{1,1}-r_{1,1}'|^{1/p_j'}+|r_{1,2}-r_{1,2}'|^{1/p_j'}}
{r_{1,1}+r_{1,2}}\|f_j(\cdot,x_2)\|_{L^{p_j}(\mathbb R)}.
\end{eqnarray*}
Now, there exists a null set $E_{2,1}$ in $\mathbb R$ such
that $\|f_j(\cdot,x_2)\|_{L^{p_j}(\mathbb R)}<\infty$ for
$x_2\in\mathbb R\setminus E_{2,1}$. Set $E_2=\mathbb R
\times E_{2,1}$. Then $E_{2}$ ia a null set in $\mathbb R^2$.
And for $(x_1,x_2)\in\mathbb R^2\setminus E_2$,
\eqref{7.5} holds.
\end{proof}

%%%%%%%%%%%%%%%% end Proof of Claim 5 %%%%%%%%%%%%%%%%%%%%%%%%%%%%%%%
{\bf Claim 8.} Let $1<p_j<\infty$ and $f_j\in L^{p_j}
(\mathbb R^2)$ for $1\leq j\leq m$. Then there exists
a null set $E_{3}$ in $\mathbb R^2$ such that for
$(x_1,x_2)\in\mathbb R^2\setminus E_{3}$ and $(r_{1,1},
r_{1,2},r_{2,1},r_{2,2})\in\overline{\mathbb{R}}_{+}^4$
with $r_{1,1}+r_{1,2}>0$ and $r_{2,1}+r_{2,2}>0$,
\begin{equation}\label{7.6}
\begin{array}{ll}
&\displaystyle\lim\limits_{(r_{1,1}',r_{1,2}')\in\overline{\mathbb{R}}_{+}^2\atop
(r_{1,1}',r_{1,2}')\rightarrow(r_{1,1},r_{1,2})}\frac{1}{r_{2,1}+r_{2,2}}
\int_{x_2-r_{2,1}}^{x_2+r_{2,2}}
\biggl|\frac{1}{r_{1,1}+r_{1,2}}
\int_{x_1-r_{1,1}}^{x_1+r_{1,2}}f_j(y_1,y_2)dy_1\\
&\qquad\qquad\qquad\qquad-\displaystyle\frac{1}{r_{1,1}'+r_{1,2}'}\int_{x_1-r_{1,1}'}^{x_1+r_{1,2}'}f_j(y_1,y_2)dy_1\Big|dy_2=0
\end{array}
\end{equation}
\begin{proof}
\begin{align*}
&\text{The left side of \eqref{7.6}}
\\
&\le
\frac{1}{r_{2,1}+r_{2,2}}
\int_{x_2-r_{2,1}}^{x_2+r_{2,2}}
\biggl(\biggl|\frac{1}{r_{1,1}+r_{1,2}}-\frac{1}{r_{1,1}'+r_{1,2}'}\biggr|
\int_{x_1-r_{1,1}}^{x_1+r_{1,2}}|f_j(y_1,y_2)|dy_1
\\
&\hspace{1cm}+\frac{1}{r_{1,1}'+r_{1,2}'}\biggl(
\int_{x_1+\min(r_{1,2}',r_{1,2})}^{x_1+\max(r_{1,2}',r_{1,2})}
|f_j(y_1,y_2)|dy_1+\int_{x_1-\max(r_{1,1}',r_{1,1})}^{x_1-\min(r_{1,1}',r_{1,1})}
|f_j(y_1,y_2)|dy_1\biggr)
\biggr)dy_2
\\
&\le \frac{|r_{1,1}-r_{1,1}'|+|r_{1,2}-r_{1,2}'|}
{(r_{1,1}+r_{1,2})^2(r_{1,1}'+r_{1,2}')}\mathcal{M}_{\mathcal{R}}f_j(x_1,x_2)
+\frac{|r_{1,1}'-r_{1,1}|^{1/p_j'}+|r_{1,2}'-r_{1,2}|^{1/p_j}}{(r_{2,1}+r_{2,2})^{1/p_j}(r_{1,1}'+r_{1,2}')}\|f_j\|_{L^{p_j}(\mathbb{R}^2)}.
\end{align*}
Then \eqref{7.6} follows from this.
\end{proof}

Applying Claim 8, we can obtain the following claim
immediately.

{\bf Claim 9.} Let $1<p_j<\infty$ and $f_j\in L^{p_j}
(\mathbb R^2)$ for $1\leq j\leq m$. Then for $(r_{1,1},
r_{1,2},r_{2,1},r_{2,2})\in\overline{\mathbb{R}}_{+}^4$
with $r_{1,1}+r_{1,2}>0$ and $r_{2,1}+r_{2,2}>0$,
$$\lim\limits_{(r_{1,1}',r_{1,2}')\in\overline{\mathbb{R}}_{+}^2\atop
(r_{1,1}',r_{1,2}')\rightarrow(r_{1,1},r_{1,2})}u_{x,\vec{f}}(r_{1,1}',r_{1,2}',r_{2,1},r_{2,2})=u_{x,\vec{f}}(r_{1,1},r_{1,2},r_{2,1},r_{2,2}).$$

\noindent\textbf{\large{Acknowledgements.}} This work was completed during the second author was visiting the University of Kansas. The
 second author is very grateful to both Professor R. H. Torres and the math department for hospitality and stimulation conditions.
%%%%%%%%%%%%%%%% end Proof of Claim 6 %%%%%%%%%%%%%%%%%%%%%%%%%%%%%%%

%%%%%%%%%%%%%%%%%%%%%%%%%%%%%%%%%%%%%%%%%%%%%%%%%%

\end{document}